\documentclass[graybox, envcountsect
]{svmult}

\usepackage{mathptmx}       
\usepackage{helvet}         
\usepackage{courier}        
\usepackage{type1cm}        
\usepackage{makeidx}         

\usepackage{graphicx,sidecap}  
\usepackage[capbesideposition=right]{floatrow}
\usepackage{multicol}        
\usepackage[bottom]{footmisc}

\usepackage{amssymb,amsmath,comment,mathrsfs,hyper,url}

\usepackage[all]{xy}

\excludecomment{toexclude}

\makeindex             

\spnewtheorem{rem}{Remark}[subsection]{\bf}{\it}

\def\AP{{\rm AP}}
\def\inter{{\mathfrak s}}

\def\Aut{{\rm Aut}}

\def\GL{{\rm GL}}

\def\Hom{{\rm Hom}}

\def\id{{\rm id}}

\def\Ker{{\rm Ker}}

\def\SL{{\rm SL}}

\def\Tr{{\rm Tr}}

\def\A{{\mathbb A}}
\def\B{{\mathbb B}}
\def\C{{\mathbb C}}
\def\D{{\mathbb D}}
\def\F{{\mathbb F}}

\def\N{{\mathbb N}}

\def\Q{{\mathbb Q}}
\def\R{{\mathbb R}}
\def\Z{{\mathbb Z}}
\def\H{{\mathbb H}}
\def\T{{\mathbb T}}

\def\aarith{{\mathfrak A}}
\def\hatz{{\hat\Z^\times}}

\def\Tr{{\rm Tr}}

\def\cA{{\mathcal A}}
\def\cB{{\mathcal B}}
\def\cC{{\mathcal C}}

\def\cF{{\mathcal F}}

\def\cO{{\mathcal O}}

\def\scal{{(\rnt,\cO)}}
\def\dar[#1]{\ar@<2pt>[#1]\ar@<-2pt>[#1]}
\def\qqq{\,,\,~\forall}
\def\dst{{\underline \D^*}}

\def\pie{{\pi_1^{\rm alg}}}

\def\gaq{{\Gamma_\Q}}
\def\caq{{\cC_\Q}}
\def\caqo{{\cC_\Q^o}}

\newcommand{\ie}{{\it i.e.\/}\ }
\newcommand{\eg}{{\it e.g.\/}\ }
\newcommand{\cf}{{\it cf.}}
\newcommand{\opcit}{{\it op.cit.\/}\ }

\def\no{\noindent}
\def\mod{{\rm mod}}

\def\id{{\mbox{Id}}}

\def\im{{\mbox{Im}}}
\def\dim{{\mbox{dim}}}

\def\Hom {{\mbox{Hom}}}

\def\trop{{\R^\flat}}

\def\rma{\R_{\rm max}}

\def\zmax{{\Z_{\rm max}}}

\def\rhmax{{\R_{\rm max}}}

\def\rmax{\R_+^{\rm max}}

\def\trop{{\R^\flat}}

\def\fr{{\rm Fr}}

\def\arith{{(\wnt,\Z_{\rm max})}}

\def\bm2{{\rm Bmod^2}}
\def\b2{{\rm Bmod^{\mathfrak s}}}

\renewcommand{\projlim}{\varprojlim}

\def\tt{\frak{t}}

\newcommand{\nil}[1]{}

\def\nt{\N^{\times}}
\def\wnt{{\widehat{\N^{\times}}}}
\def\rnt{{[0,\infty)\rtimes{\N^{\times}}}}

\def\arith{{(\wnt,\zmax)}}

\def\aarith{{\mathscr A}}
\def\scal{{(\rnt,\cO)}}
\def\scal1{{\hat \aarith}}
\def\scal2{{\mathscr S}}

\begin{document}

\title*{The Riemann-Roch strategy
\\ \vspace{0.3cm}
Complex lift of  the Scaling Site}
\author{Alain Connes and Caterina Consani}
\institute{Alain Connes \at IHES, 35 Route de Chartres, Bures sur Yvette 91440, France.\email{alain@connes.org}
\and Caterina Consani \at Johns Hopkins University, Baltimore MD 21218 USA.  \email{kc@math.jhu.edu}}
%
%

\maketitle

\abstract*{Each chapter should be preceded by an abstract (10--15 lines long) that summarizes the content. The abstract will appear \textit{online} at \url{www.SpringerLink.com} and be available with unrestricted access. This allows unregistered users to read the abstract as a teaser for the complete chapter. As a general rule the abstracts will not appear in the printed version of your book unless it is the style of your particular book or that of the series to which your book belongs.
Please use the 'starred' version of the new Springer \texttt{abstract} command for typesetting the text of the online abstracts (cf. source file of this chapter template \texttt{abstract}) and include them with the source files of your manuscript. Use the plain \texttt{abstract} command if the abstract is also to appear in the printed version of the book.}

\abstract{We describe the Riemann-Roch strategy which consists of adapting in characteristic zero Weil's proof, of RH in positive characteristic, following the ideas of Mattuck-Tate and Grothendieck. As a new step in this strategy we implement the technique of tropical descent that allows one to deduce existence results in characteristic one from the Riemann-Roch result over $\C$. In order to deal  with arbitrary distribution functions this technique involves the results of Bohr, Jessen and Tornehave on almost periodic functions. \newline\indent
Our main result is the construction, at the adelic level, of a complex lift of the ad\`ele class space of the rationals. We interpret this lift as a moduli space of elliptic curves endowed with a {\em triangular} structure. The equivalence relation yielding the noncommutative structure is generated by isogenies. We describe the tight relation of this complex lift with the GL(2)-system. We construct the lift of the Frobenius correspondences using the Witt construction in characteristic $1$. }

\section{Introduction}
\label{sec:1}


This paper presents our latest attempts in the quest of an appropriate geometry to localize the zeros of the Riemann zeta function. The constructions described in this article  define a complex geometry  that is a ``lift'' in characteristic zero, of the (tropical) Scaling Site. This project has undergone in the past years, several  developments that we list here below in order to frame and justify this latest work.\vspace{.04in}

\no -~The interpretation of the explicit formulas of Riemann-Weil as a trace formula for the scaling action on the ad\`ele class space of a global field \cite{Co-zeta,Meyer}.
\vspace{.03in}

\no -~The interpretation of the Riemann zeta function as a Hasse-Weil counting function \cite{CC0,CC0.5}.\vspace{.03in}

\no -~The discovery of the Arithmetic Site, and the identity between the Galois action on its points over the tropical semifield $\rmax$ with  the scaling action on the ad\`ele class space of the rational numbers \cite{CC1,CC2}.\vspace{.03in}

\no -~The discovery of the Scaling Site, the identity between its points and the ad\`ele class space of the rationals \cite{CC3}.
\vspace{.03in}

\no -~In \cite{CC4} we unveiled the tropical structure of the Scaling Site, proved the Riemann-Roch theorem  on its periodic orbits, and developed   the  theory of theta functions on these orbits.\vspace{.04in}

\no At this stage, the geometric framework that we built in characteristic one is well understood. The theory of theta functions and the Riemann-Roch formula with \emph{real valued indices} on the periodic orbits of the Scaling Site, provide a convincing reason in support of the strategy of adapting Weil's proof (of the Riemann Hypothesis in positive characteristic) by following the ideas of Mattuck and Tate, and Grothendieck \cite{grmt,Crh}. However, in the process to formulate a Riemann-Roch theorem on the square of the Scaling Site one faces a substantial difficulty. The problem, which
 is still open at this time, has to do with an appropriate definition of the sheaf cohomology (as idempotent monoid) $H^1$ (the definition of $H^0$ is straightforward and that of $H^2$ can be given  by turning Serre duality into a definition). In \cite{CC5}, we have developed the beginning of a general homological algebra machine in characteristic one (\ie for tropical structures) exactly to aim for a definition of the above $H^1$. In particular, we proved the existence of non trivial Ext-functors  and we were also able  to input the resolution of the diagonal to obtain the tropical analogue of the $\check C$ech complex.
  However, when applied to $\check C$ech covers, the presence of the null elements creates unwanted contributions to the cohomology which so far we are unable to handle. The root of this problem had been already unearthed in the Example 6.5 of \cite{yoshi}. This example provides pairs $(C,D)$, $(C',D')$ of tropical curves and divisors on them, for which the tropical invariants $r(D)$ and $r(D')$, entering in the Riemann-Roch formulas \cite{BN,gathmannrr} as a substitute for the dimension of the modules $H^0(D)$ and $H^0(D')$ are different, while the modules themselves  are isomorphic.\newline
    It is well-known that the hard part of the Riemann-Roch results of \cite{BN,gathmannrr} concerns the existence of non-trivial solutions \ie the proof of a Riemann-Roch inequality. This fact leads us now to concentrate, in our set-up, exactly on the \emph{existence theme} and to develop a technique of ``tropical descent'', with the goal to deduce existence results in characteristic one from available Riemann-Roch theorems in complex geometry.\newline
    Already in the appendix of \cite{CC4}, we pointed out the relevance of the tropicalization map in the non-archimedean resp. archimedean cases. In both cases the tropicalization associates to an analytic function $f$ in a corona a piecewise affine convex function $\tau(f)$, (on a real interval $I$), whose tropical zeros are the valuations $v(z_j)$ (resp. $-\log\vert z_j\vert $) of the zeros $z_j$ of $f$. The ensuing technique of ``tropical descent'' is reported in Sect.  \ref{subsec:3.2}. In the complex (archimedean) case the tropicalization of an analytic function $f$ in the  corona $R_1< \vert z\vert <R_2$ is a convex function in the interval $-\log R_2 < \lambda < -\log R_1$ and one obtains the real half line involved in the definition of  the Scaling Site by taking 
 $R_2=1$ and $R_1=0$. Namely, one works with the punctured unit disk $\dst:=\{q\in \C\mid 0<\vert q\vert \leq 1\}$ in $\C$.
 
  \no Moreover, the action by multiplication of $\nt$ on the real half-line, which is the key structure in the definition of the Scaling Site, lifts naturally to the operation  $f(z) \to f(z^n)$ on analytic functions. This observation provides, as a starting point, the definition of the ringed topos obtained  by endowing the topos $\dst\rtimes \nt$ (for the natural action of $\nt$ on $\dst$ given by $q\mapsto q^n$) with the structure sheaf $\cO$ of complex analytic functions. Given a pair of open sets $\Omega,\Omega'$ in $\D^*$ and an integer $n\in \nt$, with $q^n\in \Omega'$ for any $q\in \Omega$, one has a natural restriction map 
$$
\Gamma(\Omega',\cO)\to \Gamma(\Omega,\cO), \ \ f(q)\mapsto f(q^n).
$$
 The map $u:\dst\to [0,\infty)$ given by $u(q)=-\log\vert q\vert$, extends to a geometric morphism $u:\dst\rtimes \nt\to \rnt$ of toposes.

 \no This development provides a first glimpse of a complex lift of the Scaling Site.   
 The piecewise affine functions obtained as tropicalization using the Jensen formula all have integral slopes, and the   zeros have integral multiplicities. To reach more general types of convex functions requires to generalize the original Jensen framework. This is achieved by the extension of the work of Jensen as developed by Jessen, Tornehave and Bohr \cite{bohr,jessen33,JT},  to the case of analytic \emph{almost periodic functions}. In this work  the Jensen formula, which counts a finite number of zeros, is extended  to measure, by the second derivative of a convex function $\varphi$,   the \emph{density} of the zeros of an analytic function $f(z)$ 
 $$\lim_{T\to \infty} \frac{1}{2T}\{\# z\mid f(z)=0,\ Re(z) = \alpha,\  |\Im z|<T\}=\varphi''(\alpha),$$
  where $f(z)$ is analytic and almost periodic on the lines $Re(z) = \alpha$. In particular, any convex function $\varphi$ (there are  minor restrictions on $\varphi$  on intervals in which the second derivative of $\varphi$ is identical to $0$) can be obtained as the ``tropicalization'' of an analytic almost periodic function. This construction resolves the problem of realizing arbitrary functions  as tropicalizations and shows (see Sect. \ref{sec:4}) how to reach continuous divisors of the form $\int n(\lambda)\delta_\lambda d^*\lambda$ as ``tropical shadows'' of discrete almost periodic divisors. This part is a first step, in our project, in order  to handle the continuous integrals $\int f(\lambda)\Psi_\lambda d^*\lambda$ of the Frobenius correspondences  $\Psi_\lambda$ involved in the implementation of the Riemann Roch strategy to a proof of the Riemann Hypothesis (RH).\vspace{.03in}

This analytic construction supplies the useful hint that in order to construct a complex lift of  the Scaling Site one needs to implement an almost periodic imaginary direction. This amounts to use the covering of the pointed disk $\dst$ by the closed Poincar\' e half plane $\bar\H:=\{ z\in \C\mid \Im(z)\geq 0\}$ defined by the map $q(z):=\exp(2\pi i z)$, and to compactify the real direction in $\bar\H$ to a compact group  $G$.   In fact, the only requirement for the sought for group compactification $\R\subset G$ is to be $\Q^\times$-invariant. We take for $G$ the smallest available choice which is the compact dual of the discrete additive group $\Q$. The compactification of the real direction in $\bar\H$ then yields  the pro-\'etale covering $\tilde \dst$ of the punctured disk $\dst$, described as the projective limit 
$\tilde \dst:=\varprojlim (E_n,p_{(n,m)})$ 
$$
	E_n:=\dst, \ \ p_{(n,m)}: E_m\to E_n, \ \  p_{(n,m)}(z):=z^a \qqq m=na, \ z\in E_m=\dst.
	$$
Here, the indexing set $\nt$ is ordered by divisibility. At the topos level one would then consider the semidirect product $\tilde \dst\rtimes \nt$. \newline
In this paper we prefer to proceed directly at the adelic level  and consider the quotient, by the action of $\Q^\times$,  of the product of  the ad\`eles  $\A_\Q$ by $G$. The noncommutative space which we reach is thus the quotient 
$$
\caq:=\Q^\times \backslash \left( G\times \A_\Q\right)
$$

\no The first key observation of this paper (see Sect. \ref{sec:5}) is that with the above choice of $G$ the quotient space $\caq$ is \emph{identical} to the quotient
$$
\caq=P(\Q)\backslash \bar P(\A_\Q), \ \  \bar P(\A_\Q):=\{\left(
\begin{array}{cc}
 a & b \\
 0 & 1 \\
\end{array}
\right)\mid a,b\in \A_\Q\}
$$
of $\bar P(\A_\Q)$ by the action by  left multiplication of the  affine ``$aX+b$" group $P(\Q)$ of the rationals. Then, the   right action of the affine group $P(\R)$ determines a natural foliation on $\caq$ whose leaves are  one-dimensional complex curves generically isomorphic to the Poincar\' e half plane $\H$. 
The sector of the Riemann zeta function is obtained after division by the right action of $\hatz$ on $\caq$ (which naturally extends its action on the ad\`ele class space).

 In Sect.  \ref{subsec:5.3}, we consider (after division by $\hatz$) the periodic orbit $\Gamma(p)$ associated to a prime $p$. We find that $\Gamma(p)$ is the mapping torus of the multiplication by $p$ in the compact group $G$. This mapping torus is an ordinary compact space and we analyze the restriction of the above foliation by one dimensional complex leaves. We show that this foliation is of type III$_\lambda$ where $\lambda=1/p$ and that the discrete decomposition of the associated factor has natural geometric interpretation. We determine the de Rham cohomology in Proposition \ref{derham}.

In Sect.  \ref{subsec:5.4} we analyze the restriction to the classical orbit of the above foliation by one dimensional complex leaves. We show that it is of type II$_\infty$ and we give an explicit construction, based on the results of Sect.  \ref{sec:4}, of the leafwise discrete lift of continuous divisors. 

The second key observation of this paper (see Sect.  \ref{sec:6})  is the tight relation of  the noncommutative space $\caq=P(\Q)\backslash
\overline{P(\A_\Q)}$  to the $\GL(2)$-system (\cite{cmbook}). The $\GL(2)$-system was conceived as a higher dimensional generalization of the BC-system and its main  feature is its arithmetic subalgebra constructed using modular functions. After recalling in Sect.  \ref{subsec:6.1} the standard notations for the Shimura variety $Sh(\GL_2,\H^\pm)=\GL_2(\Q)\backslash
\GL_2(\A_\Q)/\C^\times$, 
 we consider in Sect.  \ref{subsec:6.2} the natural map
 \begin{equation}\label{P2gl20intro}
	\caq=P(\Q)\backslash
\overline{P(\A_\Q)}\stackrel{\theta}{\to} \GL_2(\Q)\backslash
M_2(\A_\Q)^\bullet /\C^\times	=\overline{Sh^{\rm nc}(\GL_2,\H^\pm)}
\end{equation}
 from $\caq$ to the noncommutative space $\overline{Sh^{\rm nc}(\GL_2,\H^\pm)}$ underlying the $\GL(2)$-system. At the archimedean place, the corresponding inclusion $\overline{P(\R)}\subset (M_2(\R)\smallsetminus
\{0 \})$ induces a bijection of $\overline{P(\R)}$ with the complement in $(M_2(\R)\smallsetminus
\{0 \})/\C^\times$ of the point $\infty$ given by the class of matrices with vanishing second line. This result shows that the nuance between $\caq$ and $\overline{Sh^{\rm nc}(\GL_2,\H^\pm)}$ is mainly due to the non-archimedean components. In Sect.  \ref{subsec:6.3} we use the description of the $\GL(2)$-system in terms of adelic $\Q$-lattices and interpret $\caq$   in terms of {\em parabolic} $\Q$-lattices. The key result is Theorem \ref{comparecomm} which states that the natural inclusion of parabolic $\Q$-lattices among $\Q$-lattices  remains injective, except in degenerate cases, into the space of two dimensional $\Q$-lattices up to scale. The relevance of this fact originates from the richness of the function theory on the space of two dimensional $\Q$-lattices up to scale which involves in particular  modular forms of arbitrary level.

By using the geometric interpretation of $\Q$-lattices up to scale in terms of elliptic curves endowed with pairs of elements of the Tate module, we provide in  Sect.  \ref{subsec:6.4}  the geometric interpretation of the points of  $\caq$ in terms of elliptic curves endowed with a \emph{triangular structure} (reflecting the parabolic structure of the $\Q$-lattice) and modulo the equivalence relation generated by isogenies  (Sect.  \ref{subsec:6.5}). 
In Sect.  \ref{subsec:6.6} we prove that the natural complex structure of the moduli space of triangular elliptic curves is the same as the complex structure on $\caq$ defined in Sect.  \ref{sec:5} using the right action of $P(\R)$. The right action of $P(\hat \Z)$ has a simple geometric interpretation (Sect.  \ref{subsec:6.7}) and allows one to pass from $\caq$ to $\gaq$. Finally we give in the last section (Sect.  \ref{subsec:6.8}) the geometric interpretation of the degenerate cases. 
\vspace{0.05in}
 
The beginning of this paper explains a central philosophy of our strategy which is not to focus on the zeta function itself (as a function ``per se'') but to find from the start, a geometric interpretation of the zero-cycle of its zeros. We record that the original motivation for H. Bohr was the Riemann zeta function and in particular, we recall that V.~Borchsenius and B.~Jessen proved in  \cite{BJ} (Theorems 14 and 15) the ``frightening'' result that for any value $x\neq 0$, the zeros of $(\zeta(z)-x)$ have real parts which admit $1/2$ as a limit point. More precisely, their results show that for any fixed $\sigma_1>\frac 12$, the density of zeros of $(\zeta(z)-x)$, in the strip $\sigma<\Re(z) <\sigma_1$, $\frac 12<\sigma<\sigma_1$, tends to infinity when $\sigma\to \frac 12$.\newline
 In Sect.  \ref{sec:2}, we explain why the ad\`ele class space of the rationals is a natural geometric space underlying the zeros of the Riemann zeta function. First, notice that what is really central in studying these zeros  is the ideal that the Riemann zeta function generates among holomorphic functions in a suitable domain. After applying Fourier transform, the key operation on functions of a positive real variable which generates this ideal is the summation 
\begin{equation}\label{mapEintro}
	f \mapsto E(f), \ \ E(f)(v):=\sum_{\nt} f(nv), \ \ \nt:=\{n\in \N\mid n>0\}.
\end{equation}
In Sect.  \ref{sec:2} we explain the two geometric approaches suggested by this formula. The first one is of adelic nature and derived from Tate's thesis. It consists in replacing the sum over the monoid $\nt$ by a sum over the associated group $\Q_+^\times$, at the expense of crossing the half line involved at the geometric level in \eqref{mapEintro} by a non-archimedean component. This process leads directly to the ad\`ele class space. The second approach is topos theoretic and consists in considering the topos (called the Scaling Site) which is the semi-direct product of the half-line by the  monoid $\nt$. The key fact recalled in Sect.  \ref{subsec:2.2}, is that the points  of the Scaling Site, coincide with the points of the (sector of the) ad\`ele class space. Thus while one could be tempted to dismiss at first the ad\`ele class space, the topos theoretic interpretation of its points endows it with a clear geometric status. 
We explain the unavoidable noncommutative nature of this space in Sect.  \ref{subsec:2.1}.

The Riemann-Roch strategy, and in particular the technique of tropical descent which allows one to deduce existence results in characteristic one from the Riemann-Roch result over $\C$, are explained in Sect.  \ref{sec:3}.

\begin{figure}[t]
\begin{center}
\includegraphics[scale=0.19]{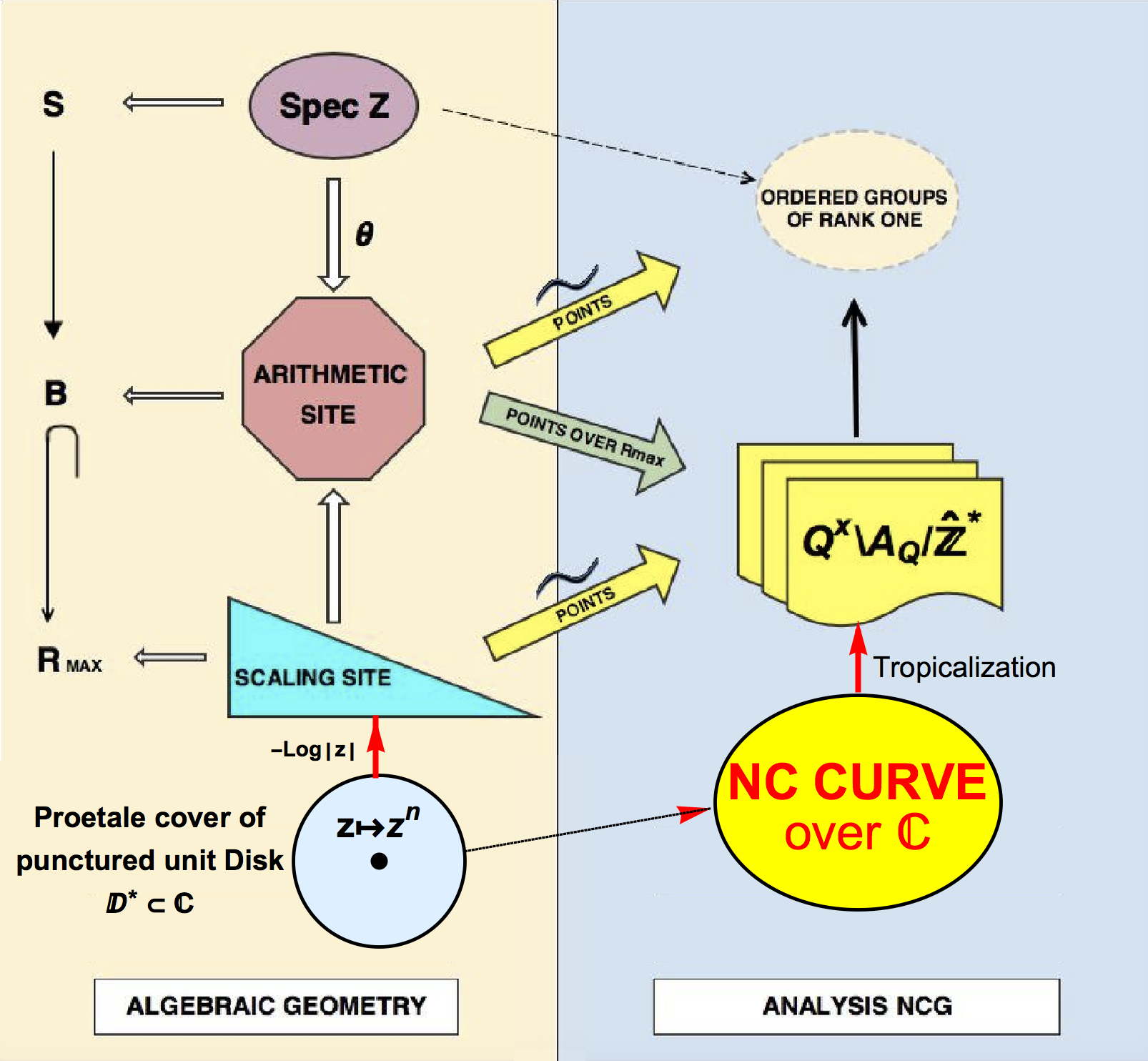}
\end{center}
\caption{Global picture \label{globalpic} }
\end{figure}

\no The framework in characteristic $1$ is perfectly adapted to the geometric role of the Frobenius. For instance, in the interpretation of the ad\`ele class space as the points of the Arithmetic Site defined over $\rmax$, the action by scaling becomes the natural action of $\Aut(\rmax)$ on these points. \newline
In the lift from characteristic $1$ to characteristic $0$ one looses the automorphisms $\Aut(\rmax)=\R_+^*$. We explain in Sect.\ref{sect:7}  the difficulty created by this loss and show in Sect. \ref{subsect:7.1} how it is resolved by the Witt construction in characteristic $1$ achieved in our previous work \cite{CCentropy, Cwitt,ccthickening}. Finally in Sect. \ref{subsect:7.2} we discuss the link between our construction of the complex lift and quantization.

 Figure \ref{globalpic} gives a visually intuitive global picture at the present time. In particular, the counterpart of $\gaq$ on the left column is the semidirect product of the  pro-\'etale cover $\tilde \dst$ of the punctured unit disk $\dst$ in the complex domain, by the natural action of $\nt$.


\section{The geometry behind the zeros of $\zeta$}
\label{sec:2}

It is important to clarify from the start  why the Riemann Hypothesis ($=$ RH), namely the problem of locating the zeros of the Riemann zeta function $\zeta(s)$, is  tightly  related to the geometry of the ad\`ele class space of the rationals $\Q$. First of all we remark that what characterizes the zeros locus of $\zeta(s)$ is not the zeta function itself rather the \emph{ideal} it generates among complex holomorphic functions in a suitable class. A key role in the description of this ideal is played by the map $E$ on functions $f(v)$ of a real positive variable $v$,  that is defined by the assignment									
\begin{equation}\label{mapE}
	f \mapsto E(f), \qquad E(f)(v):=\sum_{n\in\nt} f(nv), \quad \nt:=\{n\in \N\mid n>0\}.
\end{equation}
Notice that the map $E$ becomes, in the variable $\log v$, a sum of translations by $\log n$: $\log v\mapsto \log v+\log n $  (\ie a convolution by a sum of delta functions). Thus, after a suitable Fourier transform, $E(f)$ is a product by the Fourier transform of the sum of the Dirac masses $\delta_{\log n}$,   \ie by the function $\sum e^{-is\log n}=\zeta(is)$. Hence it should not come as a surprise that the cokernel of $E$ determines a spectral realization of the zeros of the Riemann zeta function. 

\no At this point, there are two ways of unveiling the geometric meaning of the map $E$
\begin{enumerate}
\item One may replace in \eqref{mapE} the sum over $\nt$ by a summation over the multiplicative group $\Q^*_+$ of positive rational numbers (obtained from the multiplicative monoid $\nt$ by symmetrization) with the final goal to interpret $E$ as a projection onto a quotient (of the ad\`eles of $\Q$) by the group $\Q^*_+$. This approach leads naturally to the ad\`ele class space of $\Q$, and more precisely to the sector associated to the trivial character. This construction is described  in Sect. \ref{subsec:2.1}.
\item Alternatively, one may keep the monoid $\nt$  and have it acting  on the real half line $[0,\infty)$. In this way one sees the space $\rnt$ as a Grothendieck topos. This process yields the Scaling Site of \cite{CC4} that is reviewed  in Sect. \ref{subsec:2.2}.
\end{enumerate}

 \no  The agreement of these two points of view is stated by the result, recalled in Sect. \ref{subsec:2.3}, that the points of the topos $\rnt$ coincide with the points of the (sector of the) ad\`ele class space of $\Q$. We explain the \emph{unavoidable} noncommutative nature of this space in Sect.~\ref{subsec:2.1}.

\subsection{Adelic approach}
\label{subsec:2.1}

In this part we show how the ad\`ele class space of the rationals arises naturally in connection with the study of the zeros of the Riemann zeta function. First of all notice that the summation  in \eqref{mapE} is not a summation over a group thus, in order to provide a geometric meaning to this process, we replace $\nt$ by its symmetrization \ie the group $\Q^*_+$. Then, we look for a pair $(Y,y)$ of a locally compact space $Y$ on which $\Q^*_+$ acts and a point $y\in Y$ so that the closure $F$ of the orbit $\nt y$ is {\em compact }(and also open) and  the following equivalence holds
\begin{equation}\label{Yy}
	qy \in F\iff q\in  \nt.
\end{equation}
When \eqref{Yy} holds,  one can replace the sum  in  \eqref{mapE} by the summation over the  group $\Q^*_+$ by simply considering the function $1_F\otimes f$ on the product $Y\times \R$.\newline
 A natural solution to this problem is provided by $Y=\A_f$, the finite ad\`eles of $\Q$, and by the principal ad\`ele $y=1$. One eventually achieves the minimal solution after dividing by $\hatz$. Here, we also note that the action of $\Q^*_+$ on the id\`eles cannot be used because it is a proper action.

\no A basic difficulty that one faces at this point is that the quotient of $Y\times \R=\A_\Q$ by the  action of $\Q^*_+$ is noncommutative  in the sense that classical techniques to analyze this space are here inoperative. A distinctive feature  of a noncommutative space is present  already at the level of the underlying ``set" since  a noncommutative space has the cardinality of  the continuum and at the same time it is not possible to put this space constructively in bijection with the continuum. More precisely, any explicitly constructed map from such a set to the real line fails to be injective! From these considerations one perceives immediately a major obstacle if one seeks to understand such spaces using a commutative algebra of functions. The reason why these spaces are named  ``noncommutative" is that if one accepts to use noncommuting coordinates to encode them, and one extends the traditional tools of commutative algebra to this larger noncommutative framework, everything falls correctly in place. The basic principle that one adopts is to take advantage of the  \emph{presentation} of the space as a quotient of an ordinary space (here the ad\`eles) by an equivalence relation (given here by the action of $\Q_+^*$) but then, instead of effecting the quotient in one stroke, one replaces  the equivalence relation by its convolution algebra over the complex numbers. 

\no A distinctive feature of noncommutative spaces can be seen at the level of the Borel structure allowing all sorts of countable operations on Borel functions. In the noncommutative case, the Borel structure is {\em no longer  countably separated}, in the sense that any countable family of Borel functions fails to separate points, \ie fails to be injective.  

\no Our goal in this section is to show that for whatever choice of the pair $(Y,y)$ fulfilling 
\eqref{Yy} the resulting quotient space $\Q^*_+\backslash (Y\times \R)$ is noncommutative.

\no  In Sect. \ref{subsubsec:2.1.1} we shall consider  the easier case obtained by replacing the pair $(\nt, \Q^*_+)$ with  $(\Z_{\geq 0},\Z)$. Then we show that any solution $(Y,y)$ for  $(\Z_{\geq 0},\Z)$ involves a  compactification of the discrete set $\Z_{\geq 0}$.
 In Sect. \ref{subsubsec:2.1.2} we prove, for the pair $(\Z_{\geq 0},\Z)$, the stability of ``noncommutativity". This means that given a locally compact space $R$ and a homeomorphism $S:R\to R$ whose orbit space is not countably separated,   the product action of $T\times S$ on $Y\times R$ is never countably separated, for any  auxiliary action $(Y,T)$ as in Sect. \ref{subsubsec:2.1.1}.  The strategy we follow in the subsequent Sect. \ref{subsubsec:2.1.3}, in the process of extending this result to the case of the action of $(\nt, \Q^*_+)$ on $[0,\infty)$, is reviewed by the following steps\begin{enumerate} 
 \item In the presence of a fixed point $p\in Y$ for the action of 	$\Q^*_+$, the quotient of $\{p\}\times \R_+^*$ by 	$\Q^*_+$ would be $\R_+^*/\Q^*_+$ which is not countably separated and this entails that $(Y\times \R_+^*)/\Q^*_+$ is not countably separated.
 \item If instead of a fixed point $p\in Y$ for the action of 	$\Q^*_+$ one has a fixed probability measure, then the same reasoning applies using Lemma \ref{compactsep}.
\item Using \eqref{Yy}, we construct a $\Q^*_+$-invariant probability measure on $Y$.\end{enumerate}
 Notice that condition \eqref{Yy} is essential for a meaningful development of the full strategy. Indeed, if one takes the action of  	$\Q^*_+$ on 	$Y=\Q^*_+$ by translation, the quotient $(Y\times \R_+^*)/\Q^*_+$ is the standard Borel space $\R_+^*$.

\subsubsection{Forward compactification}\label{subsubsec:2.1.1}

 To understand how to choose the pair $(Y,y)$ as in Sect. \ref{subsec:2.1}, one first considers the simpler case of the  semigroup $\nt$ replaced by the additive semigroup $\Z_{\geq 0}$ of nonnegative integers. Then the symmetrized group is $\Z$ and one looks  for a space $Y$ on which $\Z$ acts by a transformation $T$, and a point $x\in Y$ so that the closure $K$ of $T^\N x$ is  compact in $Y$ and the following equivalence is fullfilled
$$
T^nx \in K\iff n\in\Z_{\geq 0}.
$$
Next lemma states that any solution of this problem involves a compactification of the discrete set $\Z_{\geq 0}$

\begin{lemma}\label{compactif}
Let $Y$ be a locally compact space and $T\in \Aut(Y)$ an automorphism.
Let $x\in Y$ be such that the closure $K$ of the forward orbit $T^\N x$ in $Y$
is compact and the following equivalence holds 
\begin{equation}\label{tnk}
T^nx\in K\iff n\geq 0.	
\end{equation}
Then the map $\N\ni n\mapsto T^nx\in K$ turns $K$ into a compactification of
the discrete set $\Z_{\geq 0}=\{n\in \Z\mid n\geq 0\}$.
\end{lemma}
\begin{proof} It is enough to prove that for $n\in\N$ the subset $\{T^n x\}\subset  K$ is open (\ie that $T^nx$  is isolated in $K$).  Note that the complement $V=K^c$ of $K$ in $Y$ is open as well as $T^nV$ for any
$n\in \Z$, and that the intersection $T^nV\cap K$  is contained in the closure (in $Y$) of $T^nV\cap T^\N x$. Next, note that \eqref{tnk} is equivalent to $T^ux\in V\iff u< 0$ and this equivalence implies   $T^jx\in T^nV\iff j<n$.
Thus for $n>0$, one gets
$$
T^nV\cap K=\{T^jx\mid 0\leq j<n\}.
$$
 Since a point of $K$ is closed, it follows that each $\{T^jx\}$ is open in $K$. \qed\end{proof}

 The simplest compactification $K$ of the discrete set $\Z_{\geq 0}$ is the Alexandrov compactification $K=\Z_{\geq 0}\cup \{\infty\}$ obtained by  adding a limit point $\infty$.  The open subsets of $K$ containing $\infty$ are the complements of finite subsets of  $\Z_{\geq 0}$.  More generally, the  Alexandrov compactification  
of a locally compact space $X$  is obtained by adding a point at infinity and the obtained pointed  space $X\cup \{\infty\}$ admits as open sets the open subsets of $X$ and the complements of compact subsets of $X$. It 
 is described by the following universal property. 
For every pointed compact Hausdorff space $(Y,*)$
 and every continuous map $f:X\to Y$ such that $f^{-1}(K)$  
 is compact for all compact sets 
$K\subset Y$
 not containing the base point $*$, there is a unique basepoint-preserving continuous map 
 that extends $f$. When passing to the associated $C^*$-algebra,  the one point compactification just means  adjoining a unit. At  the $C^*$-level, this is the smallest compactification, but since the functor $X\mapsto C_0(X)$ is contravariant, one needs to express this fact dually. From a categorical point of view, it means that the one-point compactification is a \emph{final} object among the compactifications of a given locally compact space $X$, where morphisms of compactifications are continuous maps $g:X_1\to X_2$ which restrict to the identity on $X\subset X_j$. \newline
 Taking $K=\Z_{\geq 0}\cup \{\infty\}$ yields the following minimal solution $(Y,T)$ of \eqref{tnk}.
 \begin{lemma}\label{compactif1}
Let $\Z(+\infty)$ be the union $\Z\cup \{\infty\}$ endowed with the topology whose restriction to $\Z$ is discrete and where the intervals $[m,\infty]$ form a basis of neighborhoods of $\infty$. Then $\Z(+\infty)$ is locally compact, the translation $T(m):=m+1$, $T(\infty)=\infty$ defines a homeomorphism of $\Z(+\infty)$ and any $x\in \Z$ fulfills \eqref{tnk}.
\end{lemma}
 \begin{proof} By construction $\Z(+\infty)$ is the disjoint union of the discrete space of negative integers with the  Alexandrov compactification $K=\Z_{\geq 0}\cup \{\infty\}$.\qed \end{proof}

\begin{rem}\label{adeliccompactif}
	As a topological space the quotient $Y=\A_f/\hatz$
 is the restricted product of the spaces $\Q_p/\Z_p^*$ each of which is isomorphic to $\Z(+\infty)$ using the $p$-adic valuation. Thus Lemma \ref{compactif} shows that for each rational prime $\Q_p/\Z_p^*$ is the minimal solution of \eqref{tnk} for the multiplication  by $p$. It is  in this sense that $Y=\A_f/\hatz$ is the minimal solution of \eqref{Yy}.
\end{rem}

\subsubsection{Stability of noncommutative nature of quotients} \label{subsubsec:2.1.2}

Let now $R$ be a locally compact space endowed with an action of $\Z$ given by an homeomorphism $S:R\to R$, such that the space of the orbits is \emph{not} countably separated. In this section we show that for any  auxiliary action $(Y,T)$ as in Lemma \ref{compactif}   the product action of $T\times S$ on $Y\times R$ is \emph{never} countably separated.  
In order to prove this result (Proposition \ref{measures}) we first state the following standard fact
 
 \begin{lemma}\label{compactsep} \ \begin{enumerate}\item[(i)] Let $X$ be a compact metrizable space. Then the set of compact subsets of $X$ is countably separated. 
 \item[(ii)] The quotient of a compact metrizable space by an equivalence relation whose orbits are closed is always countably separated.
 \item[(iii)] The space of probability measures on a standard Borel space is countably separated.
 \end{enumerate}
 \end{lemma}
 \begin{proof} {\it (i)}~For any $\epsilon>0$ there exists a finite subset $F$ of $X$ such that the union of open  balls of radius $\epsilon$ centered at points of $F$ cover $X$. For $n\in\N$, let $\epsilon_n=2^{-n}$ and $F_n$ an associated finite set. Let $V\subset X$ be an open set. For each $n$, let $y_n(V)=\{t\in F_n\mid B(t,\epsilon_n)\subset V\}$. The map $V\mapsto (y_n(V))_n$ from open subsets of $X$ to the product $\prod_n 2^{F_n}$ is injective since 
 $V=\cup_n\cup_{y_n(V)}B(t,\epsilon_n)$ and a product of finite sets is countably separated by construction.\newline
 {\it (ii)}~A subset of a countably separated set is also countably separated, and since the orbits are closed they are compact so that they form a subset of the set of compact subsets of $X$ which is countably separated by {\it (i)}.\newline
 {\it (iii)}~The space of probability measures on a standard Borel space is the state space of the separable $C^*$-algebra of continuous functions on a compact metrizable space. Using a countable dense set of functions one gets the assertion.\qed \end{proof} 

 With the notations of Lemma \ref{compactif} one obtains
  
\begin{lemma}\label{compactinv}
	The complement $F$ of the the forward orbit $T^\N x$ in $K$ is a compact  subset of $(Y,T)$ invariant under the action of $\Z$ on $Y$.
\end{lemma}
\begin{proof}  By Lemma \ref{compactif}, for $n\in\N$ the subset $\{T^n x\}\subset  K$ is open, thus $F\subset K$ is closed and hence compact. One has $TT^\N x\subset T^\N x$, $TK\subset K$, and if $y\in F$ and $Ty\notin F$ one has $Ty=T^mx  \in T^\N x$ for some $m\geq 0$. For $m>0$ this contradicts $y\notin T^\N x$. For $m=0$ this gives $T^{-1}x\in K$ which  contradicts \eqref{tnk}. Thus $TF\subset F$. Let then $y\in Y$ with $Ty\in F$. Then $y\in T^{-1}TK=K$ and $y\notin T^\N x$ since $TT^\N x\subset T^\N x$. Thus $y\in F$ and one has $TF=F$.	
\qed\end{proof}

One concentrates on the product action of $T\times S$ on $F\times R$. Note that it is enough to show that this action is not countably separated to obtain the same result for the action of $T\times S$ on $Y\times R$. Since $F$ is compact and $\Z$ is an amenable group, one can find a probability measure $\mu$ on $F$ invariant for the action of $T$. Then one considers the quotient $Z$ of $F\times R$ by the product action of $T\times S$, \ie the space of orbits of this action. Let $\pi:F\times R\to Z$ be the quotient map and denote by $M_1(Z)$ the space of probability measures on $Z$.

\begin{proposition}\label{measures} \
 \begin{enumerate}
\item[(i)] The map $\rho:R\to M_1(Z)$, $\rho(x)=\pi(\mu\times \delta_{x})$  given by the image in $Z$ of the probability measure $\mu\times \delta_x$ is $S$-invariant and induces an injection in $M_1(Z)$ of the orbit space of $S$ in $E$.	
\item[(ii)] If  $Z$ is a standard Borel space then the orbit space of $S$ in $R$ is countably separated. 
\end{enumerate}
\end{proposition}
\begin{proof} {\it (i)}~Since $T\mu=\mu$, and $ \pi\circ (T\times S)=\pi$ one has
$$
\rho(Sx)=\pi(\mu\times \delta_{Sx})=\pi((T\times S)(\mu\times \delta_{x}))=
\pi(\mu\times \delta_{x})=\rho(x).
$$
We show that $\rho$ is an injection in $M_1(Z)$ of the orbit space of $S$ in $R$. Let $x,y\in E$ belong to distinct orbits of $S$. Then the characteristic function $h_x$ of the  Borel subset of $F\times R$ given by $F\times S^\Z(x)$ is $(T\times S)$-invariant and one has 
$$
\rho(x)(h_x)=(\mu\times \delta_x)(F\times S^\Z(x))=1, \quad  
\rho(y)(h_x)=(\mu\times \delta_y)(F\times S^\Z(x))=0.
$$
 Thus one concludes that $\rho(x)\neq \rho(y)$.\newline
 {\it (ii)}~It follows from {\it (i)} that the map $\rho$ is an injection in $M_1(Z)$ of the orbit space of $S$ in $R$. By Lemma \ref{compactsep} {\it (iii)}, the space $M_1(Z)$ is countably separated if $Z$ is a standard Borel space. 
 \qed\end{proof} 
 
 \subsubsection{The need for the NCG point of view}\label{subsubsec:2.1.3}
 
 The quotient of the real half line $[0,\infty)$ by the action of the multiplicative group $\Q^*_+$ is not countably separated. Indeed, this action is ergodic for the Haar measure on the multiplicative group $\R_+^*\subset [0,\infty)$. Thus any Borel function invariant for the action of $\Q^*_+$ is almost everywhere constant. In this section we show (Theorem \ref{needofnc}) that for any auxiliary action  of $\Q^*_+$ on a locally compact space $X$ such that the forward orbit $\nt x$ of some point $x\in X$ has a compact closure in $X$, the quotient of  $X\times [0,\infty)$ by the product action of $\Q^*_+$ is \emph{never} countably separated. The multiplicative group $\Q^*_+$ is the product of  an infinite number of copies of $\Z$ parametrized by the set of primes.  Its action is denoted simply as multiplication: $(q,x)\mapsto qx$. Let $K$ be the compact closure of  $\nt x$  in $X$.  We use the compactness property to construct a probability measure $\mu$ on $X$ invariant under the action of $\Q^*_+$. To achieve this result we define an increasing sequence of finite subsets  $F_k\subset \nt$, $k\in\N$, which fulfill the following properties 
 \begin{enumerate}
 \item For any integer $n$ all elements of $F_k$ are divisible by $n$ for $k$ large enough. 
 \item For any prime $p$  one has
$$
\#(F_k \Delta (pF_k))/\#(F_k) \to 0 \quad \text{for} ~~ k\to \infty 
$$ 	
where for two subsets $A,B$ of a set $C$, we denote by $A\Delta B$ their symmetric difference, \ie the complement of $A\cap B$ in $A \cup B$.
 \end{enumerate}
A way to define the set $F_k$ is, for $p_j$ the $j$-th prime,  
$$
F_k:=\{\prod_{j=1}^k p_j^{\alpha_j}\mid k< \alpha_j\leq 2k \qqq j\}.
$$
By construction all elements in $F_k$ are divisible by any integer $n$ whose prime factorization only involves the first $k$ primes taken with powers less than $k$. This fact holds for $k$ large enough and for any given $n$, thus condition  1. is fullfilled. Moreover, for a given prime $p=p_u$ and any $k\geq u$, one has 
$$
F_k \Delta (pF_k)=\{\prod_{j=1}^k p_j^{\alpha_j}\mid k< \alpha_j\leq 2k \qqq j\neq u, \ \alpha_u\in \{k+1,2k+1\}\}
$$
and from this one derives $\#(F_k \Delta (pF_k))/\#(F_k) =\frac 2k$. Thus also 2. is achieved.
 
\begin{theorem}\label{needofnc} Let $X$ be a locally compact metrizable space on  which $\Q_+^*$ acts by homeomorphisms and assume that for some $x\in X$ the closure of $\nt x$ is compact. Then the quotient of the product $X\times \R_+^*$ by the product action of $\Q_+^*$ is not a standard Borel space.	
\end{theorem}

\begin{proof} We define a probability  measure $\mu$ on the compact space $K$  closure of $\nt x$ in $X$, by taking a limit point $\mu$ in the compact space $M_1(K)$ of the sequence of measures
$$
C(K)\ni f\mapsto \frac{1}{\#(F_k)} \sum_{n\in F_k} f(nx)=\mu_k(f).
$$
For $f\in C(K)$ and any prime $p$, one has $\mu(f_p)=\mu(f)$, where $f_p(y):=f(py)$. The same property $\mu(f_n)=\mu(f)$ thus holds for any integer $n$. This proves that, when viewed as a probability measure on $X$, the measure $\mu$ is invariant under the action of $\Q_+^*$. Assume now that the quotient $Z$ of the product $X\times \R_+^*$ by the product action of $\Q_+^*$ is  a standard Borel space and let $\pi$ be the quotient map. One proceeds as in Lemma \ref{measures} to show that the map $\rho:\R_+^*\to M_1(Z)$ which associates to $\lambda\in \R_+^*$ the image $\pi(\mu\times \delta_\lambda)$ in $Z$ of the probability measure $\mu\times \delta_\lambda$ is $\Q_+^*$-invariant and defines an injection in $M_1(Z)$ of the orbit space of $\Q_+^*$ in $\R_+^*$. Indeed, for any $q\in \Q_+^*$ and $\lambda \in \R_+^*$ one has
$$
\rho(q\lambda)=\pi(\mu\times \delta_{q\lambda})=\pi(q(\mu\times \delta_{\lambda}))=
\pi(\mu\times \delta_{\lambda})=\rho(\lambda).
$$
Moreover, the evaluation on the characteristic function $h_\lambda$ of the  Borel subset of $X\times \R_+^*$ given by $X\times \Q_+^*\lambda$ shows that $\rho$ is an injection in $M_1(Z)$ of the orbit space of $\Q_+^*$ in $\R_+^*$. The conclusion follows since the orbit space of $\Q_+^*$ in $\R_+^*$ is not countably separated.\qed \end{proof} 

\subsubsection{Classical orbit and cohomological meaning of the map $E$}\label{subsubsec:2.1.4}

The role of the crossed product in encoding  noncommutative spaces enters to give a conceptual meaning of the map $E$ as the cyclic homology counterpart of the map between noncommutative spaces connecting the ad\`ele class space to its ``classical orbit" which in turn can be understood as a special case of the ``cooling procedure" described in \cite{cmbook}. The cooling procedure is nothing but a testifier of the thermodynamical nature of noncommutative spaces. When applied to the BC system the cooling amounts to replace the additive Haar measure on the ad\`eles, for which the multiplicative action of $\Q^*$ is ergodic, by the product of the Haar measure of the idele class group by a power of the module. The ad\`ele class space $\Q^\times\backslash\A_\Q$ contains the idele class group and the cooling process provides a conceptual meaning of the restriction map. It turns out that once re-interpreted on cyclic homology $HC_0$, the restriction map coincides with the map $E$. 

\subsection{The Scaling Site}\label{subsec:2.2}

The  map $E$ defined in \eqref{mapE} implements the action of $\nt$ by multiplication on the real half-line $[0,\infty)$. The notion of Grothendieck topos allows one to interpret this  construction geometrically, namely as  
the Grothendieck topos $\rnt$ of $\nt$-equivariant sheaves (of sets) on the real half-line.\newline
The combinatorial skeleton of this topos is the  Arithmetic Site $\aarith=\arith$ \cite{CC1,CC2}. This is a semiringed topos where  $\wnt$ denotes the topos of sets equipped with an action of $\nt$. The structure sheaf of the  Arithmetic Site is given by  the semiring $\zmax$ of ``max-plus'' integers that plays a key role in tropical geometry and idempotent analysis. It is a semiring of characteristic $1$, \ie $1\in \zmax$ fulfills the rule $1+1:=\text{max}(1,1)=1$. Moreover $\zmax$ is the only  semifield whose multiplicative group is infinite cyclic (\cite{CC4} Appendix B2, Proposition B3). The action of $\nt$ on $\zmax$ (which turns $\zmax$ into the structure sheaf of $\aarith$) is an instance of a general result \cite{Golan} stating that in a semifield of characteristic $1$, for any $n\in\N$, the power maps $x\mapsto x^n$ are injective endomorphisms.  These maps provide the right generalization of the Frobenius endomorphisms in finite characteristic.    
By construction, $\aarith$ is a topos defined over $\B=(\{0,1\},\text{max},+)$, the only finite semifield which is not a field. Even though $\aarith$ is    a combinatorial object of countable nature it is nonetheless endowed with a 1-parameter semigroup of correspondences on its square \cite{CC1, CC2}. Two further key properties of the Arithmetic Site  are now recalled. (1) the points of $\aarith$ defined over $\rmax$ (the multiplicative version of the tropical semifield $\rhmax$) form the basic sector $\Q^\times\backslash\A_\Q/\hat\Z^*$ of the ad\`ele class space of $\Q$; (2) the canonical action of $\Aut(\rmax)$ on these points corresponds to the action of the idele class group on $\Q^\times\backslash\A_\Q/\hat\Z^*$. These facts lead us to investigate the semiringed topos obtained from the Arithmetic Site by extension of scalars from $\B$ to $\rmax$. This space admits $\rnt$ as the underlying topos, and  
moreover it inherits, from its construction by extension of scalars, a natural sheaf $\cO$ of regular functions. We call  {\em Scaling Site}  the semi-ringed topos  
\[
\scal2:=\left(\rnt,\cO\right)
\]
 so obtained \cite{CC3, CC4}.  The sections of the sheaf $\cO$ are convex,  piecewise affine functions with integral slopes.

 \subsection{Geometry of the ad\`ele class space}\label{subsec:2.3}
 
The relation between $\scal2$ and the ad\`ele class space of $\Q$ is provided 
by the following result which states that the isomorphism classes of points  of the topos $\rnt$ form the basic sector of the ad\`ele class space of $\Q$ \cite{CC4}.

\begin{theorem}\label{scaltopintro}The  space of  points of the topos $\rnt$
is canonically isomorphic to $\Q^\times\backslash\A_\Q/\hat\Z^*$.
\end{theorem}

\no This theorem provides an algebraic-geometric structure on the ad\`ele class space, namely that of a tropical curve in an extended sense. In \cite{CC4} this structure was examined by considering its restriction  onto the periodic orbit of the scaling flow associated to each rational prime $p$. The output is that of  a tropical structure which describes this orbit as a real variant $C_p=\R_+^*/p^\Z$ of the classical Jacobi description $\C^\times/q^\Z$ of a complex elliptic curve. On $C_p$, a theory of Cartier divisors is available, moreover the structure of the quotient  of the abelian group of divisors by the subgroup of principal divisors has been also completely described in \opcit The same paper also contains a description of the theory of theta functions on $C_p$ and finally a proof of the Riemann-Roch formula stated in terms of real valued dimensions, as in the type-II index theory.

\no The main contribution of the ad\`ele class space to this geometric picture is to provide, through the implementation of  the Riemann-Weil explicit formulas as a trace formula, the understanding of  the Riemann zeta function as a Hasse-Weil generating function. \newline In the function field case the Hasse-Weil formula
writes the  zeta function as a  generating function (the Hasse-Weil zeta function)  
\begin{equation}\label{HW}
\zeta_C(s):=Z(C,q^{-s}), \qquad  Z(C,T) := \exp\left(\sum_{r\geq 1}N(q^r)\frac{T^r}{r}\right).
\end{equation}
For function fields, $q$ is the number of elements of the  finite field $\F_q$ on which the associated curve $C$ is defined.\newline In the case of the Riemann zeta function, the analogue of \eqref{HW} was obtained in \cite{CC0,CC0.5} by considering the limit of the right hand side of \eqref{HW} when $q\to 1$. This process was originally suggested by C. Soul\' e, who introduced the zeta function of a variety $X$ over $\F_1$ using the {\em polynomial} counting function $N(x)\in\Z[x]$ associated to $X$. The definition of the zeta function is as follows
\begin{equation}\label{zetadefn}
\zeta_X(s):=\lim_{q\to 1}Z(X,q^{-s}) (q-1)^{N(1)},\qquad s\in\R.
\end{equation}
 When one seeks to apply \eqref{zetadefn} to get the Riemann zeta function (completed by the gamma factor at the archimedean place) one meets the obvious obstruction that the exponent $N(1)$ is equal to $-\infty$ due to the infinite number of its zeros. In \cite{CC0,CC0.5}  a simple way to by-pass this difficulty is described \ie one considers the logarithmic derivatives of both terms in \eqref{zetadefn} and observes that the Riemann sums of an integral appear from the right hand side. Then, instead of  dealing with \eqref{zetadefn} one works with the equation
\begin{equation}\label{logzetabis}
    \frac{\partial_s\zeta_N(s)}{\zeta_N(s)}=-\int_1^\infty  N(u)\, u^{-s}d^*u.
\end{equation} 
which points out to a precise equation for the counting function $N_C(q)=N(q)$ associated to $C$ namely
\begin{equation}\label{special}
   \frac{\partial_s\zeta_\Q(s)}{\zeta_\Q(s)}=-\int_1^\infty  N(u)\, u^{-s}d^*u.
\end{equation}
In fact, one finds that this equation admits a \emph{distribution} as a solution which is given explicitly as
\begin{equation}\label{Nu}
    N(u)=\frac{d}{du}\varphi(u)+ \kappa(u)
\end{equation}
where $ \varphi(u):=\sum_{n<u}n\,\Lambda(n)$,  and $\kappa(u)$ is the distribution
that appears in the explicit formula
$$
\int_1^\infty\kappa(u)f(u)d^*u=\int_1^\infty\frac{u^2f(u)-f(1)}{u^2-1}d^*u+cf(1)\,, \qquad c=\frac12(\log\pi+\gamma).
$$
The conclusion is that the distribution $N(u)$ is positive on $(1,\infty)$ and  is given  by
\begin{equation}\label{fin2}
    N(u)=u-\frac{d}{du}\left(\sum_{\rho\in Z}{\rm order}(\rho)\frac{u^{\rho+1}}{\rho+1}\right)+1
\end{equation}
where the derivative is taken in the sense of distributions, and the value at $u=1$ of the  term
 $\displaystyle{\omega(u)=\sum_{\rho\in Z}{\rm order}(\rho)\frac{u^{\rho+1}}{\rho+1}}$ is given  by
$\frac 12+ \frac \gamma 2+\frac{\log4\pi}{2}-\frac{\zeta'(-1)}{\zeta(-1)}
$.\newline
As explained in \cite{CC0.5} the ad\`ele class space provides the geometric meaning of the counting distribution $N(u)$ and thus shows the coherence of our geometric approach.

\section{The Riemann-Roch strategy}\label{sec:3}

In relation to the study of the zeros of the Riemann zeta function, the Riemann-Roch strategy consists in trading the question of the location of the zeros  for the problem of proving the \emph{non-positivity}  of a certain quadratic form $\inter(f,f)$ (see \eqref{negcrit}). In the function field case, this inequality derives from an argument of  algebraic geometry in finite characteristic and the most conceptual proof was obtained by applying the Riemann-Roch formula on the square of the curve defining the function field \cite{grmt}.  In that case, the  function $f$ defines a divisor $D$ on the surface, as a linear combination of Frobenius correspondences. Then, if one assumes the positivity of  $\inter(f,f)>0$ for some $f$, it is the existence part of the Riemann-Roch theorem which yields a contradiction. More precisely, the assumed positivity  $\inter(f,f)>0$, together with the appearance of $\inter(f,f)$ as the leading term in the topological side of the Riemann-Roch formula show that one can turn the divisor $nD$ for a suitable $n\in \Z$ into an effective divisor and obtain a contradiction. This argument will be reconsidered  in more details in Sect. \ref{subsec:3.1}.  For function fields, the Riemann-Roch formula relies on algebraic geometry  in the same finite characteristic. In the case of the Riemann zeta function, the structure sheaf of the Scaling Site $\scal2$ is in characteristic $1$, thus it seems reasonable trying to develop a Riemann-Roch formalism in that context. Some very encouraging results are obtained in \cite{CC4}, inclusive of a type-II Riemann-Roch formula for the periodic orbits. In this case, the cohomology $H^0$ is defined using global sections while $H^1$ is introduced by turning Serre duality into a definition.   In order to attack the two dimensional case of the square of the Scaling Site one needs to define the intermediate $H^1$ and a first direct attempt,  based on homological algebra in characteristic $1$, is developed in \cite{CC5}. It is  striking that the  existence results for the Riemann-Roch problem in tropical geometry (\cite{BN,gathmannrr,MZ}) are deeply related to potential theory and game theory (\cite{BS,Shor}) thus pointing to the relevance of these tools in a direct attack to the  Riemann-Roch formula needed for RH. Here we develop yet another approach which is based on the construction of a complex lift from a geometry in characteristic $1$ to the complex world and the use of the tropicalization map. In Sect.\ref{subsec:3.2}, we explain how this {\em tropical descent} allows one, in the context of the Riemann-Roch problem, to prove the existence results in characteristic $1$ from existence results in characteristic $0$. Sect. \ref{subsect:3.3} recalls the classical link, in characteristic zero, between the Hirzebruch-Riemann-Roch theorem and the Index theorem. Finally, Sect. \ref{subsect:3.4} lays down our actual strategy which is based on the complex lift of the Scaling Site.

\subsection{The role of the existence  part of the Riemann-Roch formula in characteristic one}\label{subsec:3.1} 

It is known (\cite{B3}) that the  RH problem is equivalent to an inequality for real valued functions $f$ on $\R_+^*$ of the form 
\begin{equation}\label{negcrit}
{\rm RH} \iff \inter(f,f)\leq 0 \qqq f \mid \int f(u)d^*u=\int f(u)du=0.
\end{equation}
Here, for real compactly supported functions on $\R_+^*$, one lets $\inter(f,g):=N(f\star \tilde g)$, where $\star$ is the convolution product on $\R_+^*$,  $\tilde g(u):=u^{-1}g(u^{-1})$, and
\begin{equation}\label{negcrit1}
N(h):= \sum_{n=1}^\infty \Lambda(n)h(n)+ \int_1^\infty\frac{u^2h(u)-h(1)}{u^2-1}d^*u+c\, h(1)\,, \ c=\frac12(\log\pi+\gamma).
\end{equation}
It follows from the geometric interpretation of the explicit formulas as in \cite{CC0.5} that  the quadratic form $\inter(f,f)$ can be expressed as the self-intersection of the divisor on the square of the Scaling Site by the formula involving the Frobenius correspondences $\Psi_{\lambda}$ 
\begin{equation}\label{negcrit1bis}
\inter(f,f)=D\bullet D, \ \ D:=\int f(\lambda) \Psi_{\lambda}\, d^*\lambda.
\end{equation}
The intersection number of divisors is provided by the formula
$$
D\bullet D':=<D\star\tilde D',\Delta>
$$
where $\tilde D'$ is the transposed of $D'$ and the composition $D\star\tilde D'$ is computed by bilinearity, while the intersection $<D\star\tilde D',\Delta>$ is obtained using the distribution $N(u)$ and the fact that $\Psi_{\lambda}$ is of degree $\lambda$.

\no The Riemann-Roch strategy seeks to obtain a contradiction by assuming that, contrary to \eqref{negcrit}, one has 
 $\inter(f,f)>0$, for some function $f$. The key missing step is provided by the implementation of a Riemann-Roch formula whose topological side is $\frac 12 D\bullet D$ and to conclude from it that one can make the divisor $ D:=\int f(\lambda) \Psi_{\lambda}\, d^*\lambda$ (or its opposite $-D$) effective. 

\no The positivity of the divisor $D+(k)$ would then contradict the fact that the degree and codegree of $D=\int f(\lambda) \Psi_{\lambda}\, d^*\lambda$ is equal to $0$ in view of the hypothesis $\int f(u)d^*u=\int f(u)du=0$. 
 
\subsection{Tropical descent}\label{subsec:3.2}

The new step in our strategy is to obtain the existence part of the Riemann-Roch theorem in the tropical shadow   from the results on the analytic  geometric version of the space.
 Obviously, the advantage of working in characteristic zero is that to have already available all the algebraic and analytical tools needed to test such formula. 

\no We first explain how the Scaling Site appears naturally from the well-known results on the localization of zeros of analytic functions by means of Newton polygons in the non-archimedean case and Jensen's formula in the complex case. These results in fact combine to show that the tropical half line   $(0,\infty)$, endowed with the structure sheaf of convex,  piecewise affine functions with integral slopes, gives a  common framework for the localization of zeros of analytic functions in the  punctured unit disk. The additional structure involved in the Scaling Site, namely the action of $\nt$ by multiplication on the tropical half-line, corresponds, as shown in \eqref{tropicalizationscal} and \eqref{tropicalizationscalbis}, to the transformation on functions given by the composition with the $n$-th power of the variable.
  The tropical notion of ``zeros" of a convex  piecewise affine function $f$ with integral slope is that a zero of order $k$ occurs at a point of discontinuity of the derivative $f'$, with the order $k$ equal to the sum of the outgoing slopes. The conceptual meaning of this notion is understood  by using Cartier divisors.


\subsubsection{Tropicalization in the $p$-adic case, Newton polygons}\label{subsubsec:3.2.1}

Let $K$ be a complete and algebraically closed extension of $\Q_p$ and  $v(x)=-\log\vert x \vert$ be the valuation. The tropicalization of a series with coefficients in $K$   is obtained by applying the transformation $a\mapsto \log\vert a \vert=-v(a)$ to the coefficients and by implementing the change of operations:
$
+ \rightarrow \vee=\text{sup} ,~\times \rightarrow +$, so that $ X^n\rightarrow -nx.
$
In this way a sum of monomials such as $\sum a_n X^n$ is replaced by $\vee (-nx-v(a_n))$. 
\begin{definition}\label{defntropna}  Let $f(X)=\sum a_n X^n$ be a  Laurent  series with coefficients in $K$ and convergent in an annulus  $A(r_1,r_2)= \{z\in K\mid r_1<\vert z\vert <r_2\}$.  The tropicalization  of $f$ is the real valued function of a real parameter
\begin{equation}\label{tropicalization4}
\tau(f)(x):= \max_n\{-nx -v(a_n) \}\qquad \forall x\in (-\log r_2, -\log r_1).
\end{equation}	
\end{definition}
\no Up-to a trivial change of variables, this notion is well-known in $p$-adic analysis, where the function $-\tau(-x)$, or rather its graph, is called the {\em  valuation polygon} of the series \cite{Robert}. 
 This polygon is dual to the Newton polygon of the series which is, by definition, the lower part of the convex hull of the points of the real plane with coordinates $(j,v(a_j))$. By construction, $\tau(f)(x)$ is finite since, using the convergence hypothesis, the terms $-nx -v(a_n)$ tend to $-\infty$ when $\vert n\vert\to \infty$. Thus one obtains a  convex and piecewise affine function. Moreover, the multiplicativity property also holds
$
\tau(fg)(x)=\tau(f)(x)+\tau(g)(x), \forall x\in (0, \infty)$
as well as the following classical result (\cite{Robert})

 \begin{theorem}\label{classical} Let $f(X)=\sum a_n X^n$ be a  Laurent  series with coefficients in $K$, convergent in an annulus  $A(r_1,r_2)= \{z\in K\mid r_1<\vert z\vert <r_2\}$.  	Then the valuations $v(z_i)$ of its zeros $z_i\in A(r_1,r_2)$ (counted with multiplicities) are the zeros (in the tropical sense and counted with multiplicities) of the tropicalization $\tau(f)$ in $(-\log r_2, -\log r_1)$.
\end{theorem}
In particular, one can take $r_1=0$, $r_2=1$ so that $A(r_1,r_2)$ is the punctured open unit  disk $D(0,1)\setminus \{0\}$. In this case,  $\tau(f)$ are convex piecewise affine functions on $(0,\infty)$ and one derives the following compatibility with the action of $\nt$ on functions by $f(X)\mapsto f(X^n)$ 
\begin{equation}\label{tropicalizationscal}
\tau(f(X^n))(x)=\tau(f)(nx)\qqq x\in (0, \infty),\ n\in \nt.
\end{equation}

\subsubsection{Tropicalization in the archimedean case, Jensen's Formula}\label{subsubsec:3.2.2}

Over the complex numbers, unlike the non-archimedean case, it is not true that for a generic radius $r$, the modulus  $\vert f(z)\vert$ (of a complex function $f(z)$) is constant on the sphere of radius $r$. One replaces \eqref{tropicalization4} with the following

\begin{definition}\label{defntrop} Let $f(z)$ be a holomorphic function in an annulus  $A(r_1,r_2)= \{z\in \C\mid r_1<\vert z\vert <r_2\}$. Its tropicalization is the function on the interval $(-\log r_2, -\log r_1)$
\begin{equation*}\label{tropicalization4bis}
\tau(f)(x):=\frac{1}{2 \pi}\int_0^{2\pi} \log\vert f(e^{-x+i\theta})\vert d\theta.
\end{equation*}	
\end{definition}

 \no By construction, the multiplicativity property still holds:
$
\tau(fg)(x)=\tau(f)(x)+\tau(g)(x), \forall x\in (0, \infty)$.\newline
For $x\in (-\log r_2, -\log r_1)$ such that $f$ has no zero on the circle of radius $e^{-x}$, the derivative of $\tau(f)(x)$ is the opposite of the winding number $n(x)$ of the loop $\theta\mapsto f(e^{-x+i\theta})\in \C^\times$. Thus the function $\tau(f)(x)$ is piecewise affine with integral slopes. When the radius $e^{-x}$ of the circle increases, the winding number of the associated loop increases by the number of zeros of $f$ in the intermediate annulus and this shows that the function $\tau(f)(x)$ is convex and fulfills Jensen's formula (\cf~\cite{Rudin} Theorem 15.15). Thus we derive the analogue of Theorem \ref{classical} 

\begin{theorem}\label{classical1} Let $f(z)$ be a holomorphic function in an annulus  $A(r_1,r_2)= \{z\in \C\mid r_1<\vert z\vert <r_2\}$ and $z_i\in A(r_1,r_2)$ its zeros counted with their multiplicities. Then the values $-\log\vert z_i\vert$  are the zeros (in the tropical sense and counted with multiplicities) of the tropicalization $\tau(f)$ in $(-\log r_2, -\log r_1)$.
\end{theorem}
  
\no In particular, one can take $r_1=0$, $r_2=1$ so that $A(r_1,r_2)$ is the open punctured unit  disk $D(0,1)\setminus \{0\}$. In that case the $\tau(f)$ are convex piecewise affine functions on $(0,\infty)$ and 
  one has the following compatibility with the action of $\nt$ on functions by $f(z)\mapsto f(z^n)$ 
\begin{equation}\label{tropicalizationscalbis}
\tau(f(z^n))(x)=\tau(f)(nx)\qqq x\in (0, \infty),\ n\in \nt.
\end{equation}
This fact follows from the equality for periodic functions $h(\theta)$
$$
\frac{1}{2 \pi}\int_0^{2\pi}h(n\theta) d\theta=\frac{1}{2 n\pi}\int_0^{2n\pi}h(\alpha)d\alpha=\frac{1}{2 \pi}\int_0^{2\pi}h(u)du.
$$

\subsubsection{Descent from characteristic zero to characteristic one}\label{subsubsec:3.2.3}

To explain the general technique that allows one to deduce the existence results in characteristic one from a Riemann-Roch formula in characteristic zero, we first develop the following simple example.  Consider an open interval $I$ of the real half-line and an integral (finite) divisor  $D=\sum n_j \delta_{\lambda_j}$, with $n_j\in \Z$ and $\lambda_j\in I$. The Riemann-Roch problem in characteristic one asks for the  construction of  a piecewise affine continuous function $f$ with integral slopes, whose divisor $(f)$ fulfills $D+(f)\geq 0$. Here, $(f)$ is best understood as the second derivative $\Delta(f)$, taken in the sense of distributions. Thus the Riemann-Roch problem in characteristic one corresponds to the solutions $f$, among piecewise affine continuous function $f$ with integral slopes, of the inequality 
\begin{equation}\label{rrchar1}
D+(f):=\sum n_j \delta_{\lambda_j}+\Delta(f)\geq 0.
\end{equation}
The technique we follow is to lift geometrically the divisor $D$ to a divisor $\tilde D$ (in the ordinary complex analytic sense) in the corona 
$$
\cC(I):=\{z\in \C\mid -\log \vert z\vert \in I\}.
$$
This involves a choice, for each $\lambda_j$, of points $z\in \cC(I)$ such that $-\log \vert z\vert=\lambda_j$, and of multiplicities for these points which add up to $n_j$. Now, assume that one has a solution as a meromorphic function $g$ in $\cC(I)$  such that 
$\tilde D+ (g)\geq 0$. We then consider, using Definition 
\ref{defntrop}, the tropicalization $f=\tau(g)$. 
 This formula is in fact extended to meromorphic functions by the multiplicativity rule, \ie using $\tau(h/k):=\tau(h)-\tau(k)$ for $g=h/k$. Then, Theorem \ref{classical1} shows that the divisor $\Delta(\tau(g))$ is the image by the map $u(z):=-\log \vert z\vert$ of the divisor of $g$.  This proves that the tropicalization $f=\tau(g)$ fulfills the inequality $D+(f)\geq 0$ of the Riemann-Roch problem.

\subsection{The Hirzebruch-Riemann-Roch formula and the Index theorem}\label{subsect:3.3}

Here we recall  the Hirzebruch-Riemann-Roch theorem. Let $E$ be a holomorphic complex vector bundle of rank $r$ over a compact complex manifold $X$ of dimension $n$. The Euler characteristic $\chi(E)$ of $E$ is defined by 
\begin{equation}\label{eulerchi}
\chi(E):=\sum_{j\geq 0}(-1)^j\dim(H^j(X,E)).	
\end{equation}
The cohomology $H^j(X,E)$ used in the formula is sheaf cohomology and one uses the equivalence between holomorphic vector bundles and locally free sheaves. It is known that the cohomology $H^j(X,E)$ vanishes for $j>n$.
 The relation with the analytic index is given, with the above notations, by the formula
\begin{equation}\label{eulerchi1}
	\chi(E)={\rm Ind}_a(\overline{\partial}_E)
\end{equation}
The analytic index ${\rm Ind}_a(T)$ of an operator is defined as 
$${\rm Ind}_a(T):=\dim(\Ker(T))-\dim(\Ker(T^*))$$ 
and  $\overline{\partial}_E$ denotes the ``dbar'' operator with coefficients in $E$. The Hirzebruch-Riemann-Roch formula, which is a special case of the Atiyah-Singer Index theorem, is the equality 
\begin{equation}\label{eulerchi2}
	\chi(E)=\langle {\rm Ch}(E){\rm Td}(X), [X]\rangle 
\end{equation}
of the Euler characteristic of $E$ with the topological index which is the evaluation on the fundamental class $[X]$ of $X$ of the cohomology class $ {\rm Ch}(E){\rm Td}(X)$ product of the Chern character ${\rm Ch}(E)$ of the vector bundle $E$ and the Todd genus ${\rm Td}(X)$ of $X$.

\subsection{Potential role of the complex lift of the Scaling Site}\label{subsect:3.4}

In the case of the complex lift of the (square of the) Scaling Site, we expect  $E$ to be a line bundle, the Todd genus  be equal to $1$ and that the relevant term in the topological index  comes from the term $\frac 12 c_1(E)^2$ in the Chern character of $E$.\newline
In this set-up, one difficulty is that the self-intersection of the divisor $D$ appears
as a trace taken in a \emph{relative} situation. This means that one works with the difference between the ad\`ele class space (divided by $\hat\Z^*$), say $X$, and the ideles (also divided by $\hat\Z^*$), which form a subset $Y\subset X$. The explicit formulas are obtained in the form (after a cut-off) 
$$
(\Tr_X-\Tr_Y)(\pi(f))
$$
and this corresponds to the spectral realization as a cokernel of $E:\cF(X)\to \cF(Y)$. Thus, the trace on this cokernel corresponds to the {\em opposite} of $(\Tr_X-\Tr_Y)(\pi(f))$ as required by the minus sign in the Explicit Formulas. In fact, a first task should be to understand how to express this difference of traces as an intersection number and then develop an appropriate intersection theory. The  advantage of working   in a complex framework is that one could replace the naive real intersections by the intersection of complex manifolds and also that everything is compatible with the use of the Fourier transform. In fact, we also speculate that the divergent term in $\log \Lambda$ which enters as coefficient of $f(1)$ for the test function (see \cite{cmbook} Theorem 2.36), is due to the lack of good definition of self-intersection of the diagonal. While one obtains an infinite result when working naively, the implementation of a suitable intersection theory should provide the correct Euler characteristic.
Thus adapting  the Riemann-Roch strategy comprises the following five steps
\begin{enumerate}
\item Construct the complex lift $\Gamma$ of the Scaling Site.
\item Develop intersection theory  in such a way that the divergent term in $\log \Lambda$ (see \cite{cmbook} Theorem 2.36) is eliminated.
\item Formulate and prove a Hirzebruch-Riemann-Roch formula on $\Gamma^2$,  whose topological side part $\frac 12 c_1(E)^2$ is $\frac 12\inter(f,f)$ as in \eqref{negcrit1bis}. This step involves the lifting of the divisor $D(f)=\int f(\lambda) \Psi_{\lambda}\, d^*\lambda$ in characteristic $1$ to a divisor $\tilde D(f)$ in the complex set-up and the use of correspondences.
\item Use the assumed positivity of $\inter(f,f)$ to get an existence result for $H^0(\tilde D(f))$ or 
 $H^0(-\tilde D(f))$.
 \item Use tropical descent to get the effectivity of a divisor equivalent to $D(f)$ and finally get a contradiction.	\end{enumerate}
 The development of step 3 is the most problematic since in the lift from characteristic $1$ to characteristic $0$ one looses the automorphisms $\Aut(\rmax)=\R_+^*$ which are at the origin of the Frobenius correspondences $\Psi_{\lambda}$. We settle this problem in Sect. \ref{sect:7} using the Witt construction in characteristic $1$.

\section{Tropical descent and almost periodic functions}\label{sec:4}

In order to lift a continuous divisor $D(f)=\int f(\lambda) \delta_{\lambda}\, d^*\lambda$ on  the Scaling Site (in characteristic $1$, Sect. \ref{subsec:3.1}) to a \emph{discrete} divisor $\tilde D(f)$ on a complex geometric  space,  one first needs to understand how to generalize Jensen's formula to a case where the Jensen function is no longer a piecewise linear affine convex function with integral slopes but  an arbitrary convex function. \newline
In this part we explain how H. Bohr's theory of almost periodic functions, and the theory developed by B. Jessen on the density of zeros of almost periodic analytic functions, gives a satisfactory answer to this question. This technique  plays a crucial role in the process to extend the tropical descent procedure of Sect. \ref{subsec:3.2} to control the  continuous divisors, in characteristic $1$, following the Riemann-Roch lifting strategy. This procedure will also suggest a further important information on the need of a suitable compactification $G$ of the {\em imaginary direction} required for a correct complex lift of the Scaling Site. This part  will be developed  in Sect. \ref{sec:5}.

\subsection{Almost periodic functions} \label{subsect:4.1}

We recall the definition of almost periodic functions (see \cite{bohr}  and \cite{Besicovitch} for more details). Let $H$ be the locally compact abelian group $\R$ or $\Z$.

\begin{definition}\label{defnalmostper} Let  $f~:~H \longrightarrow \C$ be a bounded continuous function and $\varepsilon > 0$ a real number. An $\varepsilon$-almost period for $f$ is a number $\tau \in H$ such that  
$$
\Vert f(.+\tau)-f(.)\Vert_{\infty}:=\displaystyle \sup_{x \in H}|f(x+\tau)-f(x)| < \varepsilon.
$$
The function  $f$  is said to be almost periodic if for any $\varepsilon> 0$ the set of 
$\varepsilon$-almost periods of $f$ is relatively dense, i.e., there is a real number 
$l = l(\varepsilon) > 0$ such that any interval with length $l$ contains at least one $\varepsilon$-almost period. 
\end{definition}
The space of  almost periodic functions on $H$ is denoted by $\AP(H)$. In the following part we are mostly interested in the case $H=\R$ but will use the case $H=\Z$ when considering sequences.\newline
 By construction, $\AP(H)$ is a $C^*$-subalgebra   of the $C^*$-algebra $C_b(H)$ of bounded continuous functions on $H$. An important characterization 
of almost periodic functions was given by S. Bochner \cite{Bochner}. 
\begin{theorem}\label{Bochner} A bounded continuous function $f\in C_b(H)$ is an almost periodic function 
if and only if 
the family of translates $\{f(.+t)\}_{t \in H}$ is relatively compact in $C_b(H)$, \ie its closure is compact. 
\end{theorem}
Bochner's characterization lead J. von Neumann in \cite{vN} to extend the notion of almost periodic function to arbitrary groups by requiring the relative compactness for the uniform norm of the set of translates of  $f$. This definition does not make use of the topology of the group and von Neumann constructed the mean value of a function $f$  using the translation invariant element in the closed convex hull of the translates of the function.

\subsection{From Jensen to Jessen and the tropical descent}\label{subsec:4.2}

Jensen's formula in the annular case  
allows one to define the tropicalization of a holomorphic function. In \cite{jessen33}, Jessen   extended Jensen's formula to analytic almost periodic functions.\newline
Recall that an analytic function $f(z)$ in the strip $\Re(z)\in [\alpha,\beta]$ is called almost periodic when the function  $\R\ni t\mapsto f(\sigma+it)$ is uniformly almost periodic for $\sigma  \in [\alpha,\beta]$. \newline Jessen showed that, for such a function,  the following limit exists
\begin{equation}\label{jessen1}
	\varphi(\sigma):=\lim_{T\to \infty}\frac {1}{2T}\int_{-T}^T \log\vert f(\sigma+it)\vert dt
\end{equation}
and determines a real convex continuous function $\varphi(\sigma)$ of $\sigma  \in [\alpha,\beta]$. The function $\varphi(\sigma)$ is called the Jensen function of $f$. By convexity, the  derivative $\varphi'(\sigma)$ exists at all points of the interval except for a denumerable set $E$. For $\sigma_j$ outside $E$, Jessen proved that the relative frequency of zeros of $f$ in the strip  $\Re(z)\in [\sigma_1,\sigma_2]$ exists and is given by the variation of the derivative $\varphi'$. More precisely, if $N(T)$ denotes the number of zeros of $f$ with $\Re(z)\in [\sigma_1,\sigma_2]$, and $-T<\Im(z)<T$ one has 
\begin{equation}\label{jessen2}
\lim_{T\to \infty}\frac{N(T)}{2T}= \frac{\varphi'(\sigma_2)-\varphi'(\sigma_1)}{2\pi}.
\end{equation}

\subsection{Discrete lift of a continuous divisor}\label{subsec:4.3}

In this part we describe, following \cite{JT}, the procedure of lifting a tropical continuous divisor (\ie a formal integral of delta functions $\int f(\lambda) \delta_\lambda\, d\lambda$) to a discrete, \emph{integer valued} divisor, using the technique of almost periodic lifting.\newline
 The formal expression $\int f(\lambda) \delta_\lambda\, d\lambda$ replaces the finite discrete sum as in \eqref{rrchar1} of Sect. \ref{subsubsec:3.2.3}. We recall that the basic relation defining the divisor ${\rm div}(\phi)$, in characteristic $1$, of a piecewise affine function $\phi(\sigma)$ is ${\rm div}(\phi)=\Delta(\phi)$, where $\Delta$ is the Laplacian taken in the sense of distributions. Here, we extend this definition to convex functions in terms of the equation (taken in the sense of distributions)
\begin{equation}\label{rrchar1bis}
{\rm div}(\phi):=\Delta(\phi)
\end{equation}
Then, the almost periodic lifting of a convex function is the {\em choice} of an almost periodic analytic function $f$ whose tropicalization gives back the function $\phi(\sigma)$. More precisely (following \cite{JT}, Theorem 25) one has the next characterization of a Jensen function of an almost periodic analytic function
\begin{theorem}\label{Jessen}  (\cite{JT}, Theorem 25) A real function $\phi(\sigma)$, in the interval $\alpha<\sigma<\beta$, is the Jensen function of an almost periodic analytic function in the strip $\Re(z)\in [\alpha,\beta]$ if and only if $\phi(\sigma)$ is convex and for every compact interval $I\subset (\alpha,\beta)$ there exist a finite set $F$ of $\Q$-linearly independent real numbers and a real number $C<\infty$ such that the positive difference  $\phi'(\sigma_2)-\phi'(\sigma_1)$ of slopes of $\phi(\sigma)$ in intervals where it is affine is a rational combination of elements of $F$ with 
$$
\phi'(\sigma_2)-\phi'(\sigma_1)=\sum_{\mu\in F} r(\mu) \mu, \qquad \sum_{\mu\in F} r(\mu)^2\leq C\, \vert\phi'(\sigma_2)-\phi'(\sigma_1)\vert^2.
$$
\end{theorem}
These results are based on the characterization of the asymptotic distribution function of almost periodic sequences $U(k)\in \R$, $k\in \Z$. Here, to be almost periodic for sequences means, with $H=\Z$ in Definition \ref{defnalmostper}, that for any $\epsilon>0$ the set of $\epsilon$-periods
$$
\chi(\epsilon):=\{ j\mid \vert U(k+j)-U(k)\vert<\epsilon \qqq k\in \Z\}
$$
is relatively dense in $\Z$ (see Definition \ref{defnalmostper}). A distribution function is a non-decreasing function $\mu(\sigma)$ of a real variable $\sigma\in \R$ whose limit, when $\sigma\to -\infty$ is $0$ and whose limit, when $\sigma\to \infty$ is $1$. One defines $\mu(\sigma\pm 0)$ respectively as the limits on the left and on the right and one disregards the choice of a precise value in the interval $[\mu(\sigma-0),\mu(\sigma+0)]$ when the two values are different. This situation only occurs on a denumerable set of values of $\sigma$. The asymptotic distribution function of  an almost periodic sequence $U(k)$ of real numbers is defined using the densities of the subsets   $E_-(\sigma):=\{k\mid U(k)<\sigma\}$ and  $E_+(\sigma):=\{k\mid U(k)\leq\sigma\}$. For an arbitrary subset $E\subset \Z$, one defines first the lower and upper densities by the formulas 
\begin{equation}\label{density}
	\underline{\rho}(E):=\liminf_I \frac{\#\{E\cap I\}}{\#I}, \qquad 
	\overline{\rho}(E):=\limsup_I \frac{\#\{E\cap I\}}{\#I}
\end{equation}
where the limits are taken over all intervals $I=[a,b]\subset \Z$ whose length $b-a$ tends to $\infty$. Finally, the asymptotic distribution of  an almost periodic sequence $U(k)$, when it exists, is uniquely determined (as a distribution function in the above sense) as the non-decreasing function $\mu(\sigma)$ such that 
\begin{equation}\label{density1}
	\mu(\sigma-0)\leq \underline{\rho}(E_-(\sigma))\leq \overline{\rho}(E_+(\sigma)) \leq\mu(\sigma+0).
\end{equation}

A. Wintner showed that such a distribution function  exists for all almost periodic sequences $U(k)$ of real numbers  (see \cite{JT} Theorem 10 for a simple proof).  The almost periodic sequences are the  continuous functions on the almost periodic compactification  of $\Z$ which is the dual of the additive  group  $(\R/2\pi\Z)_{\rm dis}$  endowed with the discrete topology. This abelian group is uncountable but the Fourier transform of an almost periodic sequence $U(k)$ 
$$
\hat U(s):=\lim_{T\to \infty} \frac{1}{2T}\sum_{-T}^T U(k)e^{isk}, \ s\in (\R/2\pi\Z)_{\rm dis}
$$
vanishes except on a countable subset, called the set of exponents of $U$ in \opcit The subgroup $ M  \subset \R/2\pi\Z$ generated  by the exponents of $U$ is called the ``modul" $M_U$ of $U$. Its intersection with $2\pi \Q$ plays a role in particular in the following result 
\begin{theorem}\label{Jessen11} (\cite{JT}, Theorem 11) The asymptotic distribution function $\mu(\sigma)$ of  an almost periodic sequence $U(k)$ of real numbers is constant in an open interval if and only if the sequence does not take any value in this interval. In this case the value of $\mu(\sigma)$ in this interval is a rational number which belongs to  $\frac{1}{2\pi} M_U$.
\end{theorem}
The interesting outcome of this result is the rationality of the value $\mu(\sigma)=r$ in the interval where the spectrum is empty. The proof uses the fact that an $\epsilon$-period, for $\epsilon$ smaller than the size of the gap, is a true period for the subset where $U(k)<\sigma$ and that the density of a periodic set is a rational number. \newline
The following result (see \cite{JT}, Theorem 12) characterizes the distribution functions of the almost periodic sequences whose exponents belong to a fixed subgroup 
$
 M \subset \R/2\pi\Z
$
which is assumed to be everywhere dense for the usual topology
\begin{theorem}\label{Jessen12} Let $M$ be a given dense subgroup of $\R/2\pi\Z$. A distribution function is the asymptotic distribution function $\mu(\sigma)$ of  an almost periodic sequence $U(k)$ with exponents in $M$ if and only if it has compact support and the values  $\mu(\sigma)$ in constancy intervals belong to $\frac{1}{2\pi} M$. 
\begin{figure}[t]
\begin{center}
\includegraphics[scale=0.65,bb=150 10 138 260]{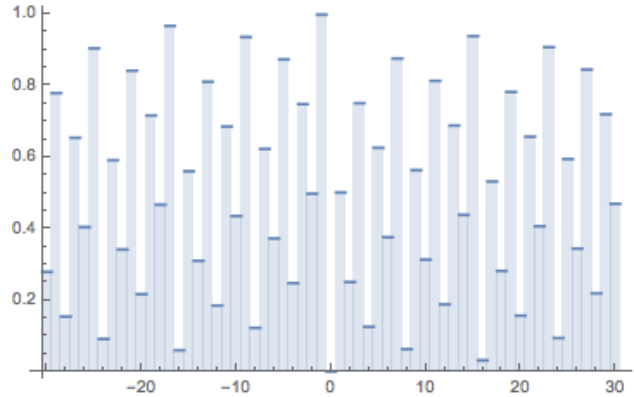}
\end{center}
\caption{Jessen sequence $U(k)$ \label{jensen2} }
\end{figure}

\end{theorem}
To phrase this result in modern terms, note that the asymptotic distribution functions $\mu(\sigma)$ of   almost periodic sequences $U(k)$ with exponents in $M$ are the same as the functions of the form 
$$
\mu(\sigma)=\nu(\{u\mid h(u)\leq \sigma\}
$$
where $h\in C(\hat M)$ is an arbitrary continuous function on the compactification of $\Z$ given by the Pontrjagin dual $\hat M$ of $M$, and where $\nu$ is the normalized Haar measure on $\hat M$. The constancy intervals of $\mu(\sigma)$ are gaps in the spectrum of $h$ and the corresponding spectral projection given by the Cauchy integral of the resolvent through the gap gives an idempotent in $C(\hat M)$. 
 If $M$ is torsion free (\ie $M$ intersects trivially with $2\pi\Q/2\pi\Z$), the compact group $\hat M$ is connected and the spectrum of any continuous function  $h$ is connected. \newline 
 The gaps in the spectrum of $h$ arise only from the torsion of $M$, and Theorem \ref{Jessen12} suggests that in order to describe all  asymptotic distribution functions $\mu(\sigma)$ of   almost periodic sequences $U(k)$ it is enough  
  to consider, instead of the almost periodic compactification  of $\Z$, (\ie the Pontrjagin dual of $(\R/2\pi\Z)_{\rm dis}$) the dual $G$ of  infinite torsion subgroups, \ie   groups of the form $D=H/\Z$ where $H\subset \Q$ is a subgroup of $\Q$ containing $\Z$ as a subgroup of infinite index.\newline   One  writes $D=\varinjlim D_n$ as a colimit of finite groups, thus its Pontrjagin dual $G=\hat D$ is a projective limit of finite groups and hence  a totally disconnected space. Theorem \ref{Jessen12} displays the role of  the idempotents in $C(G)$. This accounts for the existence of the many idempotents in $C(G)$ associated to constancy intervals of asymptotic distribution functions $\mu(\sigma)$. \newline
 To understand in explicit terms how to obtain general distribution functions,  we first work out the description of the distribution function $\mu(\sigma):=\sigma$ for $\sigma\in [0,1]$, for the group $H_p$ of rational numbers with denominator a power of $p$, with $p$ a fixed prime number. \newline
 In this case, the group $D=H/\Z$ is the colimit of the finite groups $D_n:=\Z/p^n\Z$, viewed as groups of roots of unity of order dividing $p^n$, so that the inclusions $D_n\subset D_{n+1}$ turn $D$ into a dense subgroup of $U(1)$. Dually, one thus gets a projective system of the finite cyclic groups $\hat D_n$ and an homomorphism 
$\Z\to \projlim \hat D_n$. The projective limit $K:=\projlim \hat D_n$  is topologically a Cantor set
\begin{lemma}\label{exampl} Let $p$ be a prime. \newline
$(i)$~There exists a unique sequence $U(x)$, $x\in \Z$ such that 
\begin{equation}\label{seqq}
	U(\sum_{j\geq 0} a_jp^j)=\sum_j a_jp^{-j-1},\quad \forall a_j\in \{0,\ldots, p-1\}, \ \ U(x):=\lim_{n\to \infty} U(x+p^n).
\end{equation}
$(ii)$~One has $U(x)\in H_p$, $\forall x\in \Z$, and
\begin{equation}\label{seqq1}
	\vert U(x+np^m)-U(x)\vert \leq p^{-m}\qqq m\in \N, \ x,n\in \Z
\end{equation}
$(iii)$~The  sequence $U(x)$ is almost periodic with modul $M=2\pi H_p$ 
	and has as distribution function $\mu(\sigma):=\sigma$, for $\sigma\in [0,1]$
\end{lemma}
\begin{proof} $(i)$~Let $x\in \N$ be a positive integer with expansion in base $p$ given by $x=\sum_{j= 0}^k a_jp^j$. One has for $m>k$, $U(x+p^m)=U(x)+p^{-m-1}$ thus $U(x)$ fulfills the continuity condition as in \eqref{seqq}. For $x\in \Z$, $x<0$ let $k>0$ be such that $y=x+p^k>0$ and let $y=\sum_{0\leq j<k} b_jp^j$ be its expansion in base $p$. One has 
$$
\lim_{n\to \infty} U(y+p^n-p^k)=\sum_{0\leq j<k} b_jp^{-j-1}+\sum_{k}^\infty (p-1)p^{-\ell-1}=\sum_{0\leq j<k} b_jp^{-j-1}+p^{-k}
$$
Replacing $k$ by  $k'\geq  k$ replaces $y$ by $y'=\sum_{0\leq j<k} b_jp^j+\sum_{k}^{k'-1}(p-1)p^m$  and gives the same result for $U(x)$ since $\sum_{k}^{k'-1}(p-1)p^{-m-1}=p^{-k}-p^{-k'}$. Thus for any $x\in \Z$ the limit $\lim_{n\to \infty} U(x+p^n)$ exists and  this shows the existence and uniqueness of the sequence fulfilling \eqref{seqq}.\newline
$(ii)$~The proof of $(i)$ shows that  $U(x)\in H_p$, $\forall x\in \Z$. Let us prove 
 \eqref{seqq1}. One can assume $n>0$ by symmetry and $x>0$ using \eqref{seqq}. 
Replacing $x$ by $x+np^m$ does not alter the  digits $a_j$ of 
$x$  in base $p$ for  $j<m$. Thus one has
$$
\vert U(x+np^m)-U(x)\vert \leq \sum_{j\geq m} (p-1)p^{-j-1}=p^{-m}.
$$
$(iii)$~By \eqref{seqq1} the  sequence $U(x)$ is almost periodic with modul $M=2\pi H_p$. 
Let $k>0$ and $a_j\in \{0,\ldots, p-1\}$ for $0\leq j\leq k-1$. Let $I\subset [0,1]$ be the interval
$$
I=[\sum_{j<k} a_jp^{-j-1},\sum_{j<k} a_jp^{-j-1}+p^{-k})
$$
One has using the almost periodicity
$$
\lim_{T\to \infty} \frac{1}{2T}\# \{x\in \Z \mid \vert x\vert\leq T, \ U(x)\in I\}
=\lim_{T\to \infty} \frac{1}{T}\# \{x\in \N \mid x\leq T, \ U(x)\in I\}
$$
Moreover 
$$
\lim_{T\to \infty} \frac{1}{T}\# \{x\in \N \mid x\leq T, \ U(x)\in I\}=p^{-k}
$$
since the condition $U(x)\in I$ for $x\in \N$ means that the first $k$ digits of $x$ in base $p$ are equal to the $a_j$. Thus the density of the subset $\{x\in \Z\mid U(x)\in I\}$ is $p^{-k}$ and coincides with the length (and hence the  Lebesgue measure) of the interval $I$ in the range of $U$. This shows that $U$ has distribution function $\mu(\sigma)=\sigma$ for $\sigma\in [0,1]$. \qed\end{proof} 
 
\begin{rem}\label{symhypo}
 In the above example we have chosen $D=H_p/\Z=\Q_p/\Z_p$ and its dual is the group $\Z_p$ of $p$-adic integers (using the self-duality of the $p$-adic numbers $\Q_p$). The sequence $U$ is thus obtained by mapping $\Z_p$ to real numbers by means of (up to an overall factor $p$)  $\phi(\sum_{j\geq 0} a_jp^j)=\sum a_jp^{-j}$. This map is continuous but  not additive. It fulfills however the restricted additivity $\phi(x+y)=\phi(x)+\phi(y)$ when no carry over is involved in computing $x+y$.
		\end{rem}

\begin{corollary}\label{exampl1}\
\begin{enumerate}
\item[(i)] Let $h\in C[0,1]$ be a real valued function and $U$ as in \eqref{seqq}. Then the sequence $h(U)(n):=h(U(n))$ is almost periodic and its distribution function $\mu$ is the primitive of the image by $h$ of the Lebesgue measure $m$ on $[0,1]$, \ie one has $d\mu=h(m)$.
\item[(ii)] Let $\mu$ be a strictly increasing continuous function in a real interval $[\alpha,\beta]$ with $\mu(\alpha)=0$ and $\mu(\beta)=1$. Then with $h\in C[0,1]$ as in  (i), its inverse function, $\mu$ is the distribution function of $h(U)$. 
\end{enumerate}
\end{corollary}
\begin{proof} {\it (i)} The almost periodicity follows from the continuity of $h$. One has
$$
h(U(n))\in (a,b)\iff U(n)\in h^{-1}((a,b))
$$
and the density of this set of integers is the Lebesgue measure of $h^{-1}((a,b))$. This shows that  the measure $d\mu$ is the image by $h$ of the Lebesgue measure on $[0,1]$.\newline
{\it (ii)} follows from {\it (i)} since for any pair of real numbers $a<b$ one has
$$
h(u)\in (a,b)\iff u\in (\mu(a),\mu(b)).
$$\qed \end{proof}
In the construction provided in  \cite{JT} of an almost periodic analytic function $f$ whose associated Jensen function is a given convex function $\varphi$, the zeros of $f$ can be taken of the form $z_k=V(k)+ik$ for $k\in \Z$,  where $V(k)$ is an almost periodic sequence. The  Jensen function $\varphi$ is  related to the asymptotic distribution function $\mu$ of $V$ by the equation 
\begin{equation}\label{phiprime}
\varphi'(\sigma)=\mu(\sigma).
\end{equation}
Next, we apply Corollary \ref{exampl1} to construct a lifted divisor associated to the formal expression $\int f(\lambda) \delta_\lambda\, d\lambda$.  We write $f=f_+-f_-$, where $f_\pm$ is positive, and we assume that the support of $f_\pm$ is a compact real interval $I_\pm=[\alpha_\pm,\beta_\pm]$. We also assume for simplicity that $\int_{I_\pm} f_\pm(\lambda) \, d\lambda=1$.  Let us define 
\begin{equation}\label{lifted0}
h_\pm:[0,1]\to [\alpha_\pm,\beta_\pm], \qquad h_\pm(u)=a\iff \int_{\alpha_\pm}^a f_\pm (v)dv=u.
\end{equation}
Then the construction of \cite{JT} provides the following lift 

\begin{lemma}\label{lift} The following discrete integral divisor lifts the continuous divisor $\int f(\lambda) \delta_\lambda\, d\lambda$
\begin{equation}\label{lifted}
D:=\sum_{k\in \Z} 	\delta_{h_+(U(k))+ik}-\delta_{h_-(U(k))+ik}.
\end{equation}	
\end{lemma}
\begin{proof} For a complex test function $\psi$ defined  on the real half-line $(0,\infty)$ one has by definition
$$
<\int f(\lambda) \delta_\lambda\, d\lambda,\psi>=\int f(\lambda) \psi(\lambda)\, d\lambda.
$$
It is enough to show that when evaluated on the function $z\mapsto \psi\circ \Re(z)$, the discrete divisor $D$ gives the same result after averaging. One has 
\begin{multline*}
\lim_{T\to \infty}\frac{1}{2T}\sum_{k\in \Z,\vert k\vert\leq T} 	<\delta_{h_+(U(k))+ik}-\delta_{h_-(U(k))+ik},\psi\circ \Re >=\\
=\lim_{T\to \infty}\frac{1}{2T}\sum_{k\in \Z,\vert k\vert\leq T}(\psi(h_+(U(k)))- 
\psi(h_-(U(k)))=\\
=\int_0^1 (\psi \circ h_+)(u)du-\int_0^1 (\psi \circ h_-)(u)du=\\
=\int_{\alpha_+}^{\beta_+}\psi(a)d\left(\int_{\alpha_+}^a f_+ (v)dv\right)
-\int_{\alpha_-}^{\beta_-}\psi(a)d\left(\int_{\alpha_-}^a f_- (v)dv\right)
=\int f(\lambda) \psi(\lambda)\, d\lambda.
\end{multline*}
\qed\end{proof}

\begin{rem}
The construction as in \eqref{seqq} generalizes when $H_p$ is replaced by any infinite subgroup of $\Q$  not isomorphic to $\Z$. This is also   explained in \cite{JT} and the connection with the Scaling Site should be explored further. In fact, note that for the periodic orbits of the Scaling Site the restriction on the slopes of the convex functions of the structure sheaf was stated in terms of these slopes belonging to $H_p$, and this condition corresponds to \eqref{phiprime}.
\end{rem}

\no In the above development (following \cite{JT}) we have worked with the  complex right-half plane, however in relation with the action of $\SL(2,\R)$, the use of the upper-half plane $\H:=\{z\in \C\mid \im(z)>0\}$ is more convenient.\newline
  Figure \ref{jensen3} shows the discrete almost periodic (in the horizontal direction) distribution of points in the upper half plane $\H$ associated to $k+iU(k)\in \H$
  
\begin{figure}[t]
  \begin{center}
\includegraphics[scale=0.65,bb=150 10 138 250]{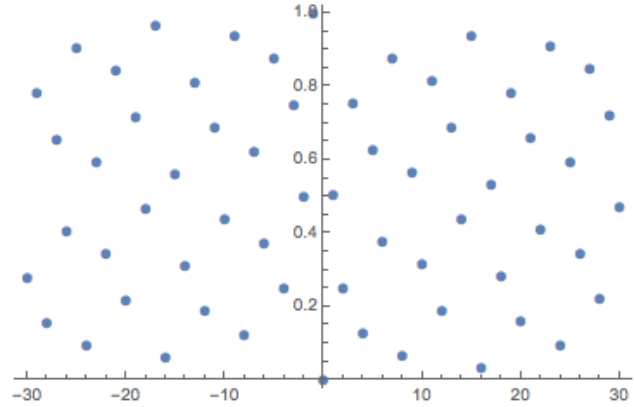} \end{center}
\caption{Jessen sequence  $k+iU(k)\in \H$}
\label{jensen3}
\end{figure}

\no Next Figure \ref{jensen4}   represents the divisor (zeros in blue, poles in red) associated to $f(x)=-2 x^3+3 x^2-x$, written as $f=f_+-f_-$ with $f_+(x)=x(1-x)$ and $f_-(x)=2x(1-x)^2$
\begin{figure}[h!]
\begin{center}
\includegraphics[scale=0.65,bb=150 10 138 250]{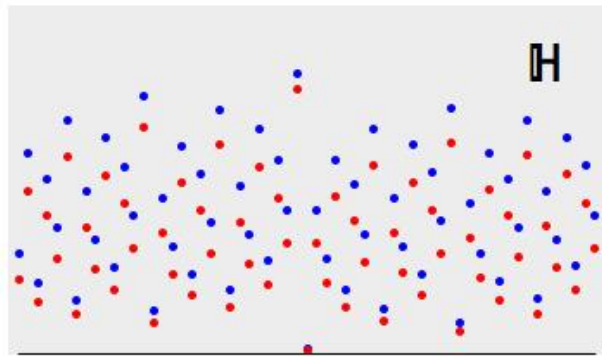}
\end{center}
\caption{Divisor for $f(x)=-2 x^3+3 x^2-x$ written as $f=f_+-f_-$, with $f_+(x)=x(1-x)$ and $f_-(x)=2x(1-x)^2$ \label{jensen4} }
\end{figure}

\section{The complex lift of the Scaling Site}\label{sec:5}

In this section we refine the framework of Jensen's formula to construct the complex lift of the Scaling Site. The original set-up of Jensen's theory and the related tropicalization map, explained in Sect. \ref{subsubsec:3.2.2}, provide us with   the following starting point.\newline 
We let $\dst:=\{q\in \C\mid 0<\vert q\vert \leq 1\}$ be the punctured unit disk in $\C$. The monoid $\nt$ acts naturally on $\dst$ by means of the map $q\mapsto q^n$. In this way, one defines a ringed topos by endowing the topos $\dst\rtimes \nt$ with the structure sheaf $\cO$ of complex analytic functions.\newline
 Given a pair of open sets $\Omega,\Omega'$ in $\D^*$ and an integer $n\in \nt$ with $q^n\in \Omega'$ for  $q\in \Omega$, there is a natural restriction map 
$$
\Gamma(\Omega',\cO)\to \Gamma(\Omega,\cO), \qquad f(q)\mapsto f(q^n).
$$

\begin{lemma}\label{mapu} The map $u: \dst\to [0,\infty)$, $u(q):=-\log\vert q\vert$, extends to a geometric morphism of toposes	$u:\dst\rtimes \nt\to \rnt$.
\end{lemma}
\begin{proof} This follows from the continuity  and $\nt$-equivariance of the map 
	$u(q):=-\log\vert q\vert$.
\qed\end{proof}

\no In order to work with continuous divisors as explained in Sect. \ref{subsec:3.2}, we consider the covering of  $\dst$ defined by the closed upper half plane $\bar\H$ 
$$
q:\bar\H\to  \dst, \qquad q(z)=e^{2\pi i z},\qquad \forall z\in \bar\H=\{z\in \C\mid \Im(z)\geq 0\},
$$ 
and make a compactification of the real direction in $\bar\H$ using a group compactification $G$ of $\R$ motivated by the results of Sect. \ref{sec:4}. 
   The compact group $G\supset \R$ used in this construction is the smallest compactification of $\R$ on which $\Q^*$ acts (by unique divisibility of the dual discrete group) by extending its natural action on $\R$. The dual of  $G$ is the additive discrete group $\Q$. \newline
   At this point one could proceed in terms of ringed toposes and this would amount, in the construction of $\dst\rtimes \nt$, to replace the punctured unit disk $\dst$ by its proetale cover given by the projective limit 
    $$
\tilde \dst:= \varprojlim_{\nt}( \dst,z\mapsto z^n).
 $$
 on which $\nt$ is acting by lifting the above action (see Proposition \ref{proetcov}). \newline
 In this paper we prefer to proceed at the adelic level and 
 the complex lift that we are going to describe in details is  obtained as the fibered product of the ad\`ele class space of $\Q$ and  $G$. It is thus  the quotient of  $\A_\Q\times G$ by the diagonal action of $\Q^*$. \newline
    The construction of the compactification $G$ is developed, in adelic terms, in Sect. \ref{subsec:5.1}. It leads, in Sect. \ref{subsec:5.2}, to the adelic definition of the complex lift. In Sects. \ref{subsec:5.3}, \ref{subsec:5.4}, we analyse the restriction of the so obtained  complex lift to the periodic orbits and to the classical orbit of the ad\`ele class space. In particular, the restriction to the classical orbit turns out to be the projective limit $\tilde \D^*$ (Proposition \ref{proetcov}) of the open unit disks (whose closed version was mentioned above in the topos theoretic context).

 The complex lift is naturally  endowed with a one dimensional complex foliation that we describe and  analyze both on the periodic and classical orbits. It is this foliation that provides the geometric meaning of the discrete lift of continuous divisors of Lemma \ref{lift}. This part is explained in Sect. \ref{subsec:5.4}. This foliation retains a meaning on the full complex lift, it is still one dimensional (complex) with leaves being the orbits of the right action of the $aX+b$ group.

\no By construction of the group $G$, the scaling action on $\A_\Q\times G$ exists  for rational values of $\lambda$ and it extends the action by scaling on $\R\subset G$. At the  archimedean place, this action is simply given by (\ie induces a) multiplication of the complex variable by $\lambda\in\Q^*$ and thus preserves the complex structure.  
  At the geometric level what really matters in this construction  is to have a compact abelian group $G$ which compactifies the additive locally compact group $\R$ and is such that the action of $\Q^*$ on $\R$ by multiplication extends  to $G$. After implementing such data one forms the double quotient
  \begin{equation}\label{doubq}
  C(G):=	\Q^*\backslash (\A_\Q\times G)/(\hatz\times \id).
  \end{equation}
  By construction, the space $C(G)$ maps onto the principal factor $\Q^*\backslash \A_\Q/\hatz$ of the ad\`ele class space and the fibers of this projection only involve the compact freedom in $G$. Writing $\A_\Q=\A_f\times \R$, one sees that the element $-1\in \Q^*$ acts as identity on $\A_f/\hatz$ but \emph{non trivially} on $\R\times G$. Moreover, $[0,\infty)\times G\subset \R\times G$ is a fundamental domain for the action of $\pm 1$ and in the quotient the boundary $\{0\}\times G$ is divided by the symmetry $u\mapsto -u$. This fact accounts for the use of the complex half-plane $\H$ in the above discussion with this nuance on the boundary. The gain, before the division by $\hatz$, is that one retains the additive group structure that we expect to play a key role in the definition of the de Rham complex, since the $H^2$ should be generated by the Haar measure. 

 \subsection{Adelic almost periodic compactification of $\R$}\label{subsec:5.1}

The main requirement on an almost periodic compactification $G$ of $\R$  is that the action of $\Q^*$ on $\R$ by multiplication extends  to $G$. Then, by turning to the Pontrjagin duals one derives a morphism $\rho:\hat G \to \hat \R$ with dense range. The fact that the scaling action of $\Q^*$ on $\R$  extends  to $G$ means here that the subgroup  $\rho(\hat G)$ is stable under multiplication 	by $\Q^*$ and hence is a $\Q$-vector subspace of $\R$. \newline
The simplest case is when this vector space is one dimensional. We shall now describe this special case in details. Thus, and up to an overall scaling, we assume  $\rho(\hat G)=\Q\subset \hat\R$.

\no We denote by $G:=\A_\Q/(\Q,+)$  the compact group quotient of the additive group of  ad\`eles by the discrete subgroup $\Q\subset \A_\Q$. We first recall how one can interpret the group $G$ as the projective limit of the compact groups $G_n:=\R/n\Z$ under the natural morphisms 
\begin{equation}\label{projgn}
  \gamma_{n,m}:G_m\to G_n, \qquad \gamma_{n,m}(x+m\Z)=x+n\Z,\qquad\forall n\vert m.
  \end{equation}
  First, notice that there is a natural group isomorphism 
  $$
  (\hat\Z\times \R)/\Z\simeq \A_\Q/(\Q,+)
  $$
  that is deduced by using the inclusion of  additive groups $\hat\Z\times \R\subset \A_\Q$   together with the two equalities
  $$
  \Q\cap  (\hat\Z\times \R)=\Z, \qquad \Q +  (\hat\Z\times \R)=\A_\Q
  $$
  where the latter derives from the density of $\Q$ in finite ad\`eles $\A_f$. 
  
 \no Next, recall that by construction the group $\hat\Z$ is the projective limit of the finite groups $\Z/n\Z$, and hence one obtains
  $$
   (\hat\Z\times \R)/\Z=\varprojlim (\Z/n\Z\times \R)/\Z.
  $$
 Moreover one has a group isomorphism 
  $$
  \rho_n:(\Z/n\Z\times \R)/\Z\to G_n=\R/n\Z, \qquad \rho_n(j,x)=x-j+n\Z
  $$
  and when $n\vert m$ one also has 
  $$
  \gamma_{n,m}\circ \rho_m(j,x)=x-j+n\Z=\rho_n(j,x).
  $$
  The outcome is that the projective system  $ (\hat\Z\times \R)/\Z=\varprojlim (\Z/n\Z\times \R)/\Z$ is isomorphic to the projective system $\{\gamma_{n,m}:G_m\to G_n\}$ and one derives in this way a natural isomorphism of the corresponding projective limits. \vspace{0.005cm}
    
   Let us now understand the action of $\Q^*$ on $G$ by checking that $G$ is a uniquely divisible group. We show that the multiplication by an integer $n>0$ defines a bijection of $G$ on itself. First, notice that its range contains the subgroup of $(\hat\Z\times \R)/\Z$  of the classes of the elements $(0,t)\in \{0\}\times \R\subset \hat\Z\times \R$ and is therefore dense in $G$ since $\Z$ is dense in $\hat \Z$. Since $G$ is compact, the image of the multiplication by $n$ is  closed and thus equal to $G$. Let us now show that the kernel of the multiplication by $n$ is trivial. The equality $(na,ns)=(m,m)$ with $a\in \hat Z$ and $m\in \Z$ implies that $m$ is divisible by $n$ since one has $m/n\in \hat\Z$. Hence one obtains $(a,s)\sim (0,0)$ and the multiplication by $n$ is therefore proven to be bijective.\vspace{.05in}
   
   \no Next lemma provides several relevant details on the chosen almost periodic compactification $G$ of $\R$
   
\begin{lemma}\label{adelicomp} Let $G:=\A_\Q/(\Q,+)$ be the compact group described above. \begin{enumerate}
\item[(i)] The homomorphism $\R\ni t\mapsto a(t)=(0,t)\in \A_\Q/(\Q,+)$ determines an almost periodic compactification of $\R$.
\item[(ii)] Let $\alpha \in \hat \A_\Q\cap \Q^\perp$ be a non-trivial additive character of the ad\`eles which restricts to the identity on $\Q\subset \A_\Q$. The map $\Q\ni q\mapsto \alpha(q\cdot~)$ identifies the additive group $\Q$ with the Pontrjagin dual $\hat G$.
\item[(iii)] Let $\alpha=\prod \alpha_v$ be the standard choice of the additive character of the ad\`eles, with $\alpha_\infty(s)=e^{2 \pi i s}$. The  Pontrjagin dual $\rho$ of the map $a$ as in {\it (i)} identifies $r\in\Q$ with the character 
 (element of $\hat\R$) given by $\R\ni s\mapsto e^{2 \pi i rs}$.	
 \end{enumerate}
\end{lemma}
\begin{proof} {\it (i)} By construction of the ad\`eles, the subgroup $\Q\subset \A_\Q$ is discrete and cocompact thus $G$ is a compact abelian group. The kernel of the homomorphism $a:\R\to \A_\Q/(\Q,+)$ is trivial since a non-zero rational has non-zero components at every local place. The homomorphism $a$ has also dense range since  the subgroup $a(\R)+\Q$ is dense in the ad\`eles as a consequence of the density of $\Q$ in the finite ad\`eles. \newline
{\it (ii)} The pairing of $\A_\Q$ with itself given by  $\alpha(xy)$ identifies $\A_\Q$ with its Pontrjagin dual and the quotient of $\A_\Q$ by $\Q^\perp=\Q$ with the Pontjagin dual of $\Q$ (see \cite{weilbasic}).\newline
{\it (iii)} One has $\alpha(r\,a(s))=\alpha_\infty(rs)=e^{2 \pi i r s}$.\qed\end{proof}

\subsection{The adelic complex lift}\label{subsec:5.2}

Next, we assemble together the ad\`ele class space and the adelic almost periodic compactification of $\R$. Our primary goal is to describe  the complex structure that arises  from the pair $(x,y)$ of variables at the archimedean place, and to verify that these variables are rescaled by the same rational number under the action of $\Q_+^*$. \newline
 In the following part we shall work with the full ad\`ele class space and postpone the division by $\hatz$ after this development 
\begin{lemma}\label{adelicomp1} Let $P(\Q)$ be the $ax+b$ group over $\Q$. The left action of $P(\Q)$ on the adelic affine plane $\A_\Q^2$ defined by 
\begin{equation}\label{actionpq}
	\ell\left(
\begin{array}{cc}
 a & b \\
 0 & 1 \\
\end{array}
\right)(x,y):=(ax+b,ay)
\end{equation}
	preserves the complex structure at the archimedean place given by  $\bar\partial=\partial_x+i\partial_y$.
\end{lemma}
\begin{proof} The statement holds because the translation by $b$ commutes with the operator $\bar\partial$ and the multiplication by $a$, being the same on both entries, just rescales the operator.\qed\end{proof} 

\begin{definition}\label{adelicomp2} The {\rm adelic complex lift} is the adelic quotient 
	\begin{equation}\label{adelicomp3}
	\caq:=P(\Q)\backslash 	\A_\Q^2.
	\end{equation}
	We denote by $\gaq$ the further quotient obtained  by implementing the action of $\hatz$ on  $\A_\Q^2$ given by multiplication on the second adelic variable $y$.
\end{definition}

\no Recall that for any commutative ring $R$ the algebraic group $P$ is defined as
\begin{equation}\label{alggroupP}
P(R):=\{\left(
\begin{array}{cc}
 a & b \\
 0 & 1 \\
\end{array}
\right)\mid a\in R^{-1},\ b\in R\}.
\end{equation}
One has a canonical inclusion of groups $\GL_1(R)\subset P(R)$ given by 
\begin{equation}\label{alggroupPbis}
\GL_1(R)\ni a\mapsto \left(
\begin{array}{cc}
 a & 0 \\
 0 & 1 \\
\end{array}
\right).
\end{equation}
For any commutative ring $R$, we introduce the notation
\begin{equation}\label{algspaceP}
\overline{P(R)}:=\{\left(
\begin{array}{cc}
 a & b \\
 0 & 1 \\
\end{array}
\right)\mid a\in R,\ b\in R\}.
\end{equation}
By construction $\overline{P(R)}\subset M_2(R)$, moreover there is a canonical set-theoretic  identification $\overline{P(R)}\simeq R^2$ given by the first line of the matrix. In the following part we prefer to work with the identification $\iota:R^2\to \overline{P(R)}$ defined by 
\begin{equation}\label{algspaceiota}
\iota(x,y):=\left(
\begin{array}{cc}
 y & x \\
 0 & 1 \\
\end{array}
\right).
\end{equation}

\begin{lemma}\label{Pdescrofgamma}\
\begin{enumerate}
\item[(i)] Let $R=\A_\Q$, then the bijection $\iota$ of \eqref{algspaceiota} is equivariant for the left action of $P(\Q)$ and induces a bijection
\begin{equation}\label{algspaceP1}
j:\caq=P(\Q)\backslash 	\A_\Q^2\stackrel{\sim}{\to} 	P(\Q)\backslash
\overline{P(\A_\Q)}.
\end{equation}
\item[(ii)] Let $K$ be the  compact subgroup $\hatz\subset \GL_1(\A_\Q)$, then $j$ induces a canonical bijection 
	\begin{equation}\label{algspaceP2}
j:\gaq \stackrel{\sim}{\to} 	P(\Q)\backslash
\overline{P(\A_\Q)}/K.
\end{equation}
\item[(iii)] The action of $P(\R)$  by right multiplication on  $P(\Q)\backslash
\overline{P(\A_\Q)}$ is free on the open subset $V$ determined by the conditions $y_f\neq 0$ and $y_\infty \neq 0$, where $y_\infty\in \R$ (resp. $y_f\in \A_f$) is the archimedean (resp. non archimedean) component of $y$ for $\left(
\begin{array}{cc}
 y & x \\
 0 & 1 \\
\end{array}
\right)\in \overline{P(\A_\Q)}$. 
\end{enumerate}
\end{lemma}
\begin{proof} {\it (i)} The multiplication rule 
$$
\left(
\begin{array}{cc}
 a & b \\
 0 & 1 \\
\end{array}
\right)\left(
\begin{array}{cc}
 y & x \\
 0 & 1 \\
\end{array}
\right)=
\left(
\begin{array}{cc}
 ay & a x +b\\
 0 & 1 \\
\end{array}
\right)
$$
shows the equivariance of the map $\iota$ with respect to the action \eqref{actionpq}. \newline
{\it (ii)} The   right action of $\GL_1(\A_\Q)$ on $\overline{P(\A_\Q)}$ is given by the formula
$$
\left(
\begin{array}{cc}
 y & x \\
 0 & 1 \\
\end{array}
\right)\left(
\begin{array}{cc}
 u & 0 \\
 0 & 1 \\
\end{array}
\right)=
\left(
\begin{array}{cc}
 uy & x\\
 0 & 1 \\
\end{array}
\right).
$$
Thus, under the isomorphism $j$ of \eqref{algspaceP1} the right action of $\hatz$ corresponds to the action of $ \hatz$ by multiplication on the $y$-component and hence it defines the required isomorphism.\newline
{\it (iii)} The conditions $y_f\neq 0$ and $y_\infty \neq 0$ are invariant under left multiplication by  $P(\Q)$ since this action replaces $y$ by $ay$ for a non-zero rational number $a$. Thus these conditions define an open subset $V$ of the quotient $P(\Q)\backslash
\overline{P(\A_\Q)}$. Let $\pi:\overline{P(\A_\Q)}\to P(\Q)\backslash
\overline{P(\A_\Q)}$ be the canonical quotient map.    The right action of  $P(\R)$  
$$
\left(
\begin{array}{cc}
 y & x \\
 0 & 1 \\
\end{array}
\right)\left(
\begin{array}{cc}
 u & v \\
 0 & 1 \\
\end{array}
\right)=
\left(
\begin{array}{cc}
 uy & vy+x\\
 0 & 1 \\
\end{array}
\right)
$$
leaves $y_f$ and $x_f$ unchanged. The open set  $V\subset P(\Q)\backslash
\overline{P(\A_\Q)}$ is invariant under this action. Let $z\in V$ and $g\in P(\R)$ be such that $zg=z$, with 
$$
z=\pi(\left(
\begin{array}{cc}
 y & x \\
 0 & 1 \\
\end{array}
\right)),\qquad  g=\left(
\begin{array}{cc}
 u & v \\
 0 & 1 \\
\end{array}
\right).
$$
The equality $zg=z$ means that there exists $h\in P(\Q)$ with $zg=hz$, thus one derives 
$$
h=\left(
\begin{array}{cc}
 a & b \\
 0 & 1 \\
\end{array}
\right),\qquad \left(
\begin{array}{cc}
 y & x \\
 0 & 1 \\
\end{array}
\right)\left(
\begin{array}{cc}
 u & v \\
 0 & 1 \\
\end{array}
\right)=\left(
\begin{array}{cc}
 a & b \\
 0 & 1 \\
\end{array}
\right)\left(
\begin{array}{cc}
 y & x \\
 0 & 1 \\
\end{array}
\right).
$$
The equality $y_f=ay_f$ shows that $a=1$ since $y_f\neq 0$.   Then, the equality $x_f=ax_f+b$ forces $b=0$. Hence one gets $h=1$. In turn, the equality $uy_\infty=y_\infty$ proves that $u=1$ since $y_\infty\neq 0$. Finally, $vy_\infty+x_\infty=x_\infty$ implies $v=0$.  
 \qed\end{proof} 
 
 Let $P_+(\R)\subset P(\R)$ be the connected component of the identity, in formulas
 $$
 P_+(\R):=\{\left(
\begin{array}{cc}
 a & b \\
 0 & 1 \\
\end{array}
\right)\mid a>0\}.
 $$
 With the notations of Lemma \ref{Pdescrofgamma}, we shall use the right action of $P_+(\R)\subset P(\R)$ to obtain a foliation of $V$ by one dimensional complex leaves endowed with a natural metric. 
  We let 
 $$
 X=\left(
\begin{array}{cc}
 0 & 1 \\
 0 & 0 \\
\end{array}
\right), \qquad Y=\left(
\begin{array}{cc}
 1 & 0 \\
 0 & 0 \\
\end{array}
\right)
 $$
 be the generators of the Lie algebra of $P_+(\R)$: one has $[Y,X]=X$.
 
 \begin{proposition}\label{ax+bact}\
 \begin{enumerate}
 \item[(i)] The  free right action of  $P_+(\R)$ on the (open) space $V$ endows the orbits with a unique Riemannian metric so that the vector fields $X,Y$ form an orthonormal basis (at each point) and a unique complex structure such that  $$\bar \partial(f)=0\iff (X+iY)f=0.$$
\item[(ii)] Each orbit as in (i) is isomorphic to the complex upper-half  plane $\H=\{x+iy\mid y>0\}$ with the Poincar\'e metric.
\item[(iii)] The Laplacian $\Delta=-\bar\partial^*\bar \partial$, for the  Riemannian metric as in (ii), is equal to $X^2+(Y-\frac 12)^2-\frac 14$.
\end{enumerate}	
\end{proposition}
\begin{proof} Recall that the Poincar\'e complex half-plane $\H$ is a one dimensional complex manifold endowed with the Riemannian metric 
$$
ds^2=\frac{dx^2+dy^2}{y^2}.
$$
The group $\GL(2,\R)^+$ acts by automorphisms of $\H$ as follows
$$
\left(
\begin{array}{cc}
 a & b \\
 c & d \\
\end{array}
\right)\cdot z:=\frac{az+b}{cz+d}.
$$
Using the inclusion $P_+(\R)\subset \GL(2,\R)^+$ obtained by setting $c=0$ and $d=1$, and selecting the point $z=i$, one obtains the left invariant Riemannian metric $ds^2=a^{-2}(da^2+db^2)$ on $P_+(\R)$, and the complex structure such that $2\bar \partial f= (\partial_a f-i \partial_b f)(da+i db)$. The vector fields which provide the right action of $P$ on itself are $Y=a\partial_a$ and $X=a\partial_b$. Using these fields, the Laplacian $\Delta=a^2(\partial_a^2+\partial_b^2)$ is given by 
\begin{equation}\label{lapxy}
	\Delta=X^2+(Y-\frac 12)^2-\frac 14
\end{equation}
Indeed, one has $Y^2-Y=(a\partial_a)(a\partial_a)-a\partial_a=a^2(\partial_a^2)$ and $X^2=a^2(\partial_b^2)$. 

\no For the orbit $L$ though $x\in V$ one has a bijection defined by $\phi_x:P_+(\R)\stackrel{\sim}{\to} L$, $\phi_x(g):=xg$, while for another point $y=xg_0$ of the same orbit one has
$
\phi_y(g)=\phi_x(g_0g).
$
Thus, the full geometric structure of $P_+(\R)$, invariant under left translations, carries over unambiguously to the orbit and the three statements follow from their validity on $P_+(\R)$.\qed\end{proof}

\subsection{The periodic orbits}\label{subsec:5.3}

The complex foliation of the open invariant set $V\subset P(\Q)\backslash
\overline{P(\A_\Q)}$ as in Proposition \ref{ax+bact} is, by construction, invariant under the right action of $K=\hatz$. To describe the geometric structure induced on the quotient $\gaq = 	P(\Q)\backslash
\overline{P(\A_\Q)}/\hatz$,  we start by investigating the induced structure on the  periodic orbit associated to a prime $p$ in the ad\`ele class space.\newline
 More precisely, we consider the  subset $\prod(p)$ of $\A_\Q/\hatz$  of classes modulo $\hatz$ of ad\`eles 
$$
a=(a_v), \ \ a_v\in \Z_v^*,\quad \forall v\notin\{p,\infty\}, \ \ a_p=0, \ \ a_\infty=\lambda >0.
$$
Any such  class is uniquely determined by $\lambda$  and will be denoted by $\pi(\lambda)\in \prod(p)$.

\begin{lemma}\label{lemper}The image in  $\gaq = 	P(\Q)\backslash
\overline{P(\A_\Q)}/\hatz$ of $G\times \prod(p)\subset \overline{P(\A_\Q)}/\hatz$ is the compact space
	\begin{equation}\label{gammap}
\Gamma(p):=p^\Z\backslash\left((\A_\Q/(\Q,+))\times \prod(p)\right)
\end{equation}
which is described by the mapping torus of the homeomorphism $\psi:G\to G$ given by multiplication by $p$.
\end{lemma}
\begin{proof} We recall that the left action of $P(\Q)$ on $\overline{P(\A_\Q)}$ is given by
$$
\left(
\begin{array}{cc}
 a & b \\
 0 & 1 \\
\end{array}
\right)\left(
\begin{array}{cc}
 y & x \\
 0 & 1 \\
\end{array}
\right)=
\left(
\begin{array}{cc}
 ay & a x +b\\
 0 & 1 \\
\end{array}
\right).
$$
 Two elements $y=\pi(\lambda)$ and $ y'=\pi(\lambda')$ in $ \prod(p)$ are equivalent under the action of $a\in\Q^\times$ if and only if $\lambda/\lambda'\in p^\Z$, \ie  $a\in p^\Z$.
Thus the orbits of the left action of $P(\Q)$ are the same as the orbits of $p^\Z$ in $(\A_\Q/(\Q,+))\times \prod(p)$. The group $G$ is compact and the multiplication by $p$ defines an automorphism $\psi$ of $G$ as can be seen on the Pontrjagin dual $\Q$ which is a uniquely divisible group. Thus, as a topological space $\Gamma(p)=p^\Z\backslash\left(G\times \prod(p)\right)$ is the mapping torus of the homeomorphism $\psi$ and is a compact space.\qed\end{proof}

\no The geometric structure of the space $\Gamma(p)$ as in \eqref{gammap} is described as follows

\begin{theorem}\label{gammap1}\
\begin{enumerate}	
\item[(i)] The foliation of $V$ as in Proposition \ref{ax+bact} induces on $\Gamma(p)$ 
	the foliation of $G$ by 
	the cosets of the subgroup $a(\R)$ combined with the action of $\R_+^*$ on $\prod(p)$.
	\item[(ii)] The foliation $F$ on $\Gamma(p)$ as in (i) is by one dimensional complex   leaves which are Riemann surfaces of curvature $-1$. All leaves of $F$, except one, are isomorphic to $\H$. The exceptional leaf is the quotient  $p^\Z\backslash \H$. 
	\item[(iii)] The foliated space $(\Gamma(p),F)$ is, at the measure theory level, a factor of type III$_\lambda$, for $\lambda=\frac 1p$.
	\end{enumerate}
\end{theorem}
\begin{proof} {\it (i)} By Lemma \ref{lemper}, $\Gamma(p)$ is the quotient of $G\times \R_+^*$ by the action of powers of the map $\theta$, given  by $G\times \R_+^*\ni (x,y)\mapsto \theta(x,y):=(\psi(x),py)$, where $\psi$ is as in Lemma \ref{lemper}.
The  right action of  $P_+(\R)$ on $\gaq$  
induces on $p^\Z\backslash\left((\A_\Q/(\Q,+))\times \prod(p)\right)$ the following right  action 
$$
\phi_{u,v}(x,y):=(x+vy,uy)\qqq (x,y)\in G\times \R_+^*, \qquad \left(
\begin{array}{cc}
 u & v \\
 0 & 1 \\
\end{array}
\right)\in P_+(\R).
$$
This is a translation $x\mapsto x+a(vy)$ in the variable $x\in G$, where $vy\in \R$ by construction, and is the scaling $y\mapsto uy$  by $u>0$ in the variable $y\in \R_+^*$. The compact group $G=\A_\Q/(\Q,+)$ is foliated by the cosets of the subgroup $a(\R)$ of Lemma \ref{adelicomp}. This foliation is globally invariant under the action of $\psi$ because the subgroup $a(\R)$ is globally invariant under this action. More precisely
the foliation of $G$ by the cosets of $a(\R)$ derives from the flow $\phi_t(x):=x+a(t)$, $t\in \R$, $x\in G$ and   
one has 
$$
\psi(\phi_t(x))=p \,(x+a(t))=px +pat=\psi(x)+a(pt)=\phi_{pt}(\psi(x)).
$$
The right action of  $P_+(\R)$ on $G\times \R_+^*$  commutes with  $\theta$ and  thus drops down to the quotient $\Gamma(p)$ 
 $$
 \theta(\phi_{u,v}(x,y))=(\psi(\phi_{vy}(x)),p(yu))=(\phi_{vpy}(\psi(x)),pyu)=\phi_{u,v}(\theta(x,y)).
 $$

\begin{figure}[t]
\begin{center}
\includegraphics[scale=0.45,bb=190 70 150 350]{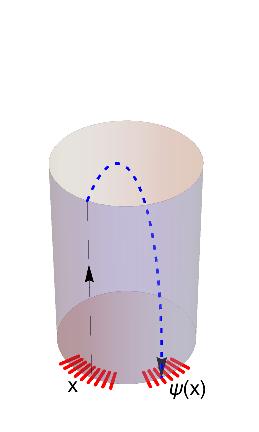}
\end{center}
\caption{The foliation of $G$ by the cosets of the subgroup $a(\R)$ is preserved by the map $\psi$ and thus extends to a two dimensional foliation of the mapping torus of $\psi$. \label{maptorus} }
\end{figure}

\no The  orbits of the right action of $P_+(\R)$ on $\Gamma(p)$ coincide with the leaves of the foliation of the mapping torus of $\psi$ induced by the foliation of $G$ by the cosets of the subgroup $a(\R)$,  as in Figure \ref{maptorus}. \newline
{\it (ii)} For $(x,y)\in G\times \R_+^*$, $\left(
\begin{array}{cc}
 u & v \\
 0 & 1 \\
\end{array}
\right)\in P_+(\R)$ and $n\in \Z$ one has 
$$
\phi_{u,v}(x,y)=\theta^n(x,y) \iff u=p^n, \qquad x+a(vy)=p^n x.
$$
For $n=0$ this gives $u=1$ and  $v=0$. Assume now $n\neq 0$. One has $G=\A_\Q/\Q$, $p^n x\in x+\Q +a(\R)$. Let $x=(x_f,x_\infty)$ correspond to the decomposition $\A_\Q=\A_f\times \R$. The above condition means that $p^n x_f\in x_f+\Q $, \ie (since $n\neq 0$) that $x_f\in \Q$.
This shows that the right action of  $P_+(\R)$ on  $\Gamma(p)$ is free on the orbit of $(x,y)$ provided $x_f\notin \Q$, \ie equivalently $x\notin a(\R)$. Thus, as in Proposition \ref{ax+bact}, these orbits of the right action of  $P_+(\R)$ inherit a canonical structure of Riemann surface isomorphic to $\H$. The right action of  $P_+(\R)$ gives the two vector fields
\begin{equation}\label{vectorfields}
X(f)(x,y)=y\partial_tf(x+a(t),y), \qquad Y(f)(x,y)=y\partial_yf(x,y). 
\end{equation}
The vector fields $X$ and $Y$ verify the Lie algebra of the affine $ax+b$ group: 
\begin{equation}\label{vectorfields1}
[Y,X]=X.
\end{equation}
  Assume now that $x\in a(\R)$. Then the orbit of $(x,y)\in \Gamma(p)$ under the right action  of  $P_+(\R)$ is
$$
p^\Z\backslash \left(a(\R)\times \prod(p)\right)\simeq p^\Z\backslash \H
$$
and does not depend upon the choice of the base point $(x,y)$. 
The complex structure makes sense and as a complex space one gets an open subset of the elliptic curve $E=p^\Z\backslash \C^\times$. One has $E\simeq \C/\Gamma$ by the isomorphism $e:\C/\Gamma\to E$, $e(z):=e^{2 \pi i z}$ and $\Gamma=\Z+ \frac{\log p}{2\pi i}\Z$.\newline
	{\it (iii)} At the measure theory level, the space of leaves of the foliation of $G$ by the cosets of the subgroup $a(\R)$ is the same as the quotient of the finite ad\`eles $\A_f$ by the additive subgroup $\Q$. The action of $\Q$ by addition on the finite ad\`eles $\A_f$ is ergodic and measure preserving. In fact since $\hat \Z$ is open in
	$\A_f$ and $\Q$ is dense, every orbit meets  $\hat \Z$. Moreover, if $b\in a+\Q$ with $a,b\in \hat\Z$ one has $b-a\in \hat\Z\cap \Q=\Z$. Also, the action of $\Z$ on $\hat \Z$ by translation is ergodic by uniqueness of Haar measure on a compact group and density of $\Z$ in $\hat \Z$. 	
	
\no In $\Gamma(p)$ the leaves of the two dimensional foliation all meet the fiber of the projection $\Gamma(p)\to p^\Z\backslash \prod(p)$ over the point $\pi(1)$ and in the leaf space a leaf of the foliation of $G$ by the cosets of the subgroup $a(\R)$ gets identified with its image by the map $\psi$ (see Figure \ref{maptorus}). Thus the leaf space is the quotient of $\A_f/\Q$ (the quotient of the finite ad\`eles $\A_f$ by the additive subgroup $\Q$) by the further action by multiplication by powers of $p$. This latter action rescales the invariant measure by a factor of $p$ and thus one obtains a factor of type  III$_\lambda$, where $\lambda=\frac 1p$.\qed\end{proof}
 
\begin{rem}\label{transverse}\
\begin{enumerate}
 \item The Haar measure $dn(x)$ on $G$ gives an invariant transverse measure $\Lambda$ for the flow $\phi_t$, moreover $dn(x)$ is invariant under the automorphism of multiplication by $p$. But the above transverse measure $\Lambda$ is not invariant under multiplication by $p$ because it is obtained as the contraction of $dn(x)$ by the flow $\phi_t$ and this flow is rescaled by multiplication by $p$.
\item Both $dn(x)$ and the measure on $\prod(p)$ given by $dy/y$, are invariant under multiplication by $p$ and thus the product measure descends to a measure on $\Gamma(p)$ given by 
	\begin{equation}\label{measurep}
\int f(x,y) dm(x,y):= \int_1^p\int_G f(x,y)dn(x)\frac{dy}{y}.
\end{equation}
\end{enumerate}
\end{rem}

\vspace{.1in}

 Next, we use the  basis of differential forms along the leaves which is dual to the vector fields \eqref{vectorfields1}. It is given by  $\alpha=y^{-1}dx,\,\beta= y^{-1}dy$ in the cotangent space to the leaves.\newline
 Next statement computes the de Rham cohomology of $\Gamma(p)$

\begin{proposition}\label{derham}
	The canonical projection $\Gamma(p)\to p^\Z\backslash \prod(p)= \R_+^*/p^\Z$ is an isomorphism in the de Rham cohomology.
\end{proposition}
\begin{proof} The de Rham complex on $\Gamma(p)$ is described as follows using the Lie algebra $L$ of the affine group and its dual $L^*$. We take the basis $(X,Y)$ for $L$ and the dual basis $(\alpha,\beta)$ for $L^*$. One lets $\Omega^j:=\cA\otimes \wedge^j L^*$ where $\cA$ is an algebra of functions on $\Gamma(p)$ stable under the derivations $X,Y$. The differential is given by 
$$
df=X(f)\alpha+Y(f)\beta \qqq f\in \Omega^0, \ \ d(f\alpha+g\beta)= df\wedge \alpha+fd\alpha+dg\wedge \beta, \ \ d\alpha=\alpha\wedge \beta
$$
We first describe the algebra $\cA$ of functions on $\Gamma(p)$ stable under the derivations $X,Y$. Let $\cB$ be the algebra of functions on $G$ linearly generated by the characters $e_q$ for $q\in \Q$. Thus the multiplication rule is $e_qe_{q'}=e_{q+q'}$ for all $q,q'\in \Q$. Let $f(y,q)$ be a function on $\R_+^*\times \Q$ and define
$$
\hat f(x,y):=\sum_\Q f(y,q)e_q(x)
$$
This definition of  $\hat f$ is meaningful if one assumes that for each $y\in \R_+^*$ the function $q\mapsto f(y,q)$ has finite support. The condition that $\hat f$ defines a function on $\Gamma(p)$ is  $\hat f(px,py)=\hat f(x,y)$ and since $e_q(px)=e_{pq}(x)$, the condition means that 
\begin{equation}\label{functionsgp}
f(py,p^{-1}q)=f(y,q)\qqq y \in \R_+^*, \ q\in \Q.
\end{equation}
In terms of the function $f(y,q)$ the derivations\footnote{The product is the convolution in the variable $q$ and the ordinary product in the variable $y$.} $X$ and $Y$ become
\begin{equation}\label{derivationsgp}
X(f)(y,q)=2\pi i yq f(y,q), \qquad Y(f)(y,q)=y\partial_yf(y,q).
\end{equation} 
The group $p^\Z$ acts on $\Q$ by multiplication and \eqref{functionsgp} and \eqref{derivationsgp} show that the de Rham complex is a direct sum over the orbits $\cO$ of this action
$$
(\Omega,d)=\bigoplus_{\cO \in p^\Z\backslash\Q}(\Omega(\cO),d).
$$
The trivial orbit of $0\in \Q$ corresponds to the pull back by the projection $\Gamma(p)\to p^\Z\backslash \prod(p)= \R_+^*/p^\Z$ and the vector field $X$ gives $0$, thus the contribution of this orbit reduces to the following complex of functions on $\R_+^*/p^\Z$
$$
df=Y(f)\beta \qqq f\in \Omega^0, \ \ d(f\alpha+g\beta)= (-Y(f)+f) \alpha\wedge \beta, \ \ d(f\alpha\wedge \beta)=0. 
$$
The map $\id-Y$ is diagonalized in the basis of characters of $\R_+^*/p^\Z\simeq U(1)$ and the eigenvalues are the complex values $1-2\pi i n/\log p$, $n\in \Z$. Indeed, with $u=\log y$ the condition $f(py)=f(y)$ becomes periodicity of period $\log p$ and $Y$ becomes $\partial_u$. It follows that $\id-Y$ is an isomorphism on smooth functions, since it does not affect the rapid decay of the Fourier coefficients. This shows that the extra part due to the  presence of the sub-complex of the $f\alpha$ and $f\alpha\wedge \beta$ does not contribute to the cohomology.\newline 
Next, we consider the contribution of a non-trivial orbit $\cO=p^\Z q_0$ with $q_0\neq 0$. A function $f(y,q)$ restricted to this orbit can be seen as a function on $\R_+^*\times \Z$ given by $h(y,n)=f(y,p^n q_0)$. Then, condition \eqref{functionsgp} becomes $h(py,n-1)=h(y,n)$. This shows that the restriction of $f$ to the orbit is entirely specified by the function on $\R_+^*$ given by $\phi(y)=f(y,q_0)$. Moreover this function is smooth and its support intersects finitely each orbit $p^\Z y$. We thus deal with the space $C_c^\infty(\R_+^*)$ of smooth compactly supported functions on $\R_+^*$. Let us compute the operators $X,Y$ in terms of the functions $\phi(y)$. Using \eqref{derivationsgp} we get 
$$
(X\phi)(y)=(2\pi i q_0)y\phi(y), \ \ (Y\phi)(y)=y\partial_y\phi(y). 
$$
Thus the operator $X$ is invertible, and using its inverse $X^{-1}$ one defines a homotopy $s:(\Omega(\cO),d)\to (\Omega(\cO),d)$, by
$$
s(f\alpha+g\beta):=X^{-1}(f)\in \Omega^0(\cO), \ \ s(f\alpha\wedge \beta)=X^{-1}(f)\beta.
$$
Next, we check that $ds+sd=\id$. This is clear on $\Omega^0$ since $sdf=X^{-1}X(f)=f$. On $\Omega^2$ is also clear since $ds(f\alpha\wedge \beta)=dX^{-1}(f)\beta=XX^{-1}(f)\alpha\wedge \beta=f\alpha\wedge \beta$. On $\Omega^1$ one has 
$$
(ds+sd)(f\alpha+g\beta)=dX^{-1}(f)+s\left((-Y(f)+f+X(g)) \alpha\wedge \beta\right)=
$$
$$
=f\alpha+YX^{-1}(f)\beta +X^{-1}(-Y(f)+f+X(g))\beta=f\alpha+g\beta.
$$
This is because $YX^{-1}(f)+X^{-1}(-Y(f)+f)=0$ which follows from the commutation relation \eqref{vectorfields1} by multiplying on both sides by $X^{-1}$.\qed\end{proof}

\subsection{The classical orbit}\label{subsec:5.4}

Consider the  subset $J\subset \A_\Q/\hatz$ of classes of ad\`eles modulo $\hatz$
\begin{equation}\label{Jdefn1}
	a=(a_v), \ \ a_v\in \Z_v^*\qqq v\neq\infty, \qquad  a_\infty=\lambda >0.
\end{equation}
A class as in \eqref{Jdefn1} is uniquely determined by $\lambda$  and will be denoted $j(\lambda)\in J$. Two such classes are in the same (classical) orbit for the left action of $\Q^*$ if and only if they are equal. Thus the structure of $\gaq$ over a classical orbit is simply that of the product 
$$
\Gamma_{\Q,{\rm cl}}\simeq G\times \R_+^*.
$$
Thus, in order to exploit measure theory and   de Rham theory on this plain product it is enough to supply this description for $G$ foliated by the cosets of the subgroup $a(\R)$. 
The  right action of  $P_+(\R)$ on $\gaq$  
induces on $\Gamma_{\Q,{\rm cl}}$ the right  action 
$$
\phi_{u,v}(x,y):=(x+a(vy),uy)\qqq (x,y)\in G\times \R_+^*, \qquad \left(
\begin{array}{cc}
 u & v \\
 0 & 1 \\
\end{array}
\right)\in P_+(\R).
$$

\begin{proposition}\label{classorb1}\
\begin{enumerate}
\item[(i)] The space $\Gamma_{\Q,{\rm cl}}$ is locally compact. 
\item[(ii)] The  right action of  $P_+(\R)$ on $\Gamma_{\Q,{\rm cl}}$ is free and defines a foliation $F$ by Riemann surfaces  isomorphic to $\H$. 
\item[(iii)] The foliated space $(\Gamma_{\Q,{\rm cl}},F)$ is, at the measure theory level, a factor of type II$_\infty$.
\item[(iv)] The de Rham complex of $(\Gamma_{\Q,{\rm cl}},F)$ is the tensor product of the de Rham complex of $\R_+^*$ by the de Rham complex of the foliation $(G,W)$ of $G$ by 
	the cosets of the subgroup $a(\R)$.
\item[(v)] The de Rham cohomology of $(G,W)$ is one dimensional in degree $0$ and $1$ and vanishes in higher degrees.
\end{enumerate} 
\end{proposition}
\begin{proof} {\it (i)} follows from the isomorphism $\Gamma_{\Q,{\rm cl}}\simeq G\times \R_+^*$. \newline
	{\it (ii)} The freeness of the right action is clear. As in Proposition \ref{ax+bact} the orbits of the right action of  $P_+(\R)$ inherit a canonical structure of Riemann surface isomorphic to $\H$. 	\newline
	{\it (iii)} The space of leaves of the foliated space $(\Gamma_{\Q,{\rm cl}},F)$ is the same as for  the foliation of $G$ by the cosets of the subgroup $a(\R)$ and is thus the quotient of   the finite ad\`eles $\A_f$ by the additive subgroup $\Q$. It is thus ergodic of type II$_\infty$.\newline
	{\it (iv)} follows because the foliation $(\Gamma_{\Q,{\rm cl}},F)$ is the product of $(G,W)$ by the trivial foliation of $\R_+^*$. \newline
	{\it (v)} Let $\cB$ be the algebra of functions on $G$ linearly generated by the characters $e_q$ for $q\in \Q$. The operator  $d_W$ of differentiation along the flow lines fulfills $d_W(e_q)=2\pi i q\, e_q$. Thus its kernel is one dimensional and spanned by $e_0$. Its cokernel is given by the linear form $L$ associated to the Haar measure of $G$, \ie $L(e_q)=0$ for $q\neq 0$ and $L(e_0)=1$. Thus de Rham cohomology of $(G,W)$ is one dimensional in degree $0$ and $1$ and vanishes in higher degrees.\qed\end{proof}
	\begin{theorem}\label{liftclass} 
	Let  $C=\int f(\lambda) \delta_\lambda\, d\lambda$ be a continuous divisor on $\R_+^*$ with compact support. There exists a finite union of graphs $G_j^\pm$  of maps $g_j^\pm$
	 \begin{equation}\label{liftedthm}
g_j^\pm: D_j^\pm\to \R_+^*, \ \  D_j^\pm\subset \hat \Z\subset G
\end{equation}	
such that the leafwise discrete divisor $D:=\sum \pm G_j^\pm$ is a lift of $C$.
\end{theorem}
	 \begin{proof} This follows from Lemma \ref{lift}. \qed\end{proof}
	
	We next give a canonical isomorphism of the classical orbit $\Gamma_{\Q,{\rm cl}}$ 
	with the proetale cover $\tilde \D^*$ of the punctured open unit disk $\D^*$ constructed from the projective system defined as follows
	$$
	E_n:=\D^*, \ \ p_{(n,m)}: E_m\to E_n, \ \  p_{(n,m)}(z):=z^a \qqq m=na, \ z\in E_m=\D^*
	$$
where the indexing set $\nt$ is ordered by divisibility.	 By construction, $\Gamma_{\Q,{\rm cl}}\simeq G\times \R_+^*$ is the projective limit
\begin{equation}\label{projgnlim1}
G\times \R_+^*=\varprojlim (G_n,\gamma_{n,m})\times \R_+^*= \varprojlim \H/n\Z 
  \end{equation}
  where the projective limit in the right hand side uses the canonical projections $\H/m\Z\to \H/n\Z$ for $m=na$ corresponding to the $\gamma_{n,m}$,
$$
\gamma_{n,m}:G_m\to G_n, \ \ \gamma_{n,m}(x+m\Z)=x+n\Z \qqq n\vert m
$$
		\begin{proposition}\label{proetcov} Let $\tilde \D^*:=\varprojlim (E_n,p_{(n,m)})$. 
		The maps $e_n:\H/n\Z\to \D^*$, 
		$
		e_n(z)= \exp(2 \pi i \frac zn)
		$
		assemble into an isomorphism  $\exp:\Gamma_{\Q,{\rm cl}}\to \tilde \D^*$.
	\end{proposition}
	 \begin{proof}  For each integer $n$ the map $e_n:\H/n\Z\to \D^*$ is an isomorphism. One has the compatibility for $m=na$ 
	 $$
	 p_{(n,m)}(e_m(z))=e_m(z)^a=\exp(2 \pi i \frac{a z}{m})
	 =\exp(2 \pi i \frac{z}{n})=e_n(z).
	 $$
	Thus this gives an isomorphism of the projective systems. \qed\end{proof}
	
	\begin{rem} \label{additivestruct} 
	The formulation of Proposition \ref{classorb1} does not reflect the additive structure at the archimedean place. 	Instead of $J$ as in \eqref{Jdefn1}, one can consider the subset $\tilde J\subset \A_\Q/\hatz$ formed of classes of ad\`eles  modulo $\hatz$
$$
a=(a_v), \ \ a_v\in \Z_v^*\qqq v\neq\infty.
$$
Then, the value of $a_\infty$ gives an isomorphism $\tilde J\simeq \R$ and the equivalence for the multiplicative action of $\Q^*$ is reduced to the orbit of $\pm 1$. In this way one obtains the following refinement of $\Gamma_{\Q,{\rm cl}}$
$$
\bar\Gamma_{\Q,{\rm cl}}\simeq (G\times \R)/\pm 1=\Gamma_{\Q,{\rm cl}} \cup (G/\pm 1).
$$
Thus the only additional piece is $G/\pm 1$.
	\end{rem}

\section{The moduli space interpretation}\label{sec:6}

In this section we relate the noncommutative space $\caq=P(\Q)\backslash
\overline{P(\A_\Q)}$ described in  Sect. \ref{sec:5} to the $\GL(2)$-system (see \cite{cmbook}). This system was conceived as a higher dimensional generalization of the BC-system \cite{BC} and its main new feature is provided by its arithmetic subalgebra of modular functions. The classical Shimura scheme $Sh(\GL_2,\H^\pm):=\GL_2(\Q)\backslash \GL_2(\A_{\Q})/\C^\times$ recalled in Sect. \ref{subsec:6.1} appears as the set of classical points of the noncommutative space $\overline{Sh^{\rm nc}(\GL_2,\H^\pm)}$ underlying the $\GL(2)$-system.    
This noncommutative space  admits a simple description as the double quotient
\[
\overline{Sh^{\rm nc}(\GL_2,\H^\pm)}= \GL_2(\Q)\backslash
M_2(\A_\Q)^\bullet /\C^\times
\]
 obtained by replacing  in the construction of $Sh(\GL_2,\H^\pm)$ the middle term $\GL_2(\A_{\Q})$ by $M_2(\A_\Q)^\bullet:= M_2(\A_{\Q,f})\times (M_2(\R)\smallsetminus
\{0 \})$ \ie the space of matrices (of ad\`eles) with non-zero archimedean component. \newline
 In Sect. \ref{subsec:6.2} we construct a map $\theta:\caq\to \overline{Sh^{\rm nc}(\GL_2,\H^\pm)}$ using the natural inclusion $\overline{P(\A_\Q)}\subset M_2(\A_\Q)^\bullet$ (Lemma \ref{fundinclusion}). The important feature of this inclusion is that, at the archimedean place, the inclusion $\overline{P(\R)}\subset M_2(\R)\smallsetminus
\{0 \})$ induces a bijection of $\overline{P(\R)}$ with the complement of a single point in $(M_2(\R)\smallsetminus
\{0 \})/\C^\times$.\newline
  In Sect. \ref{subsec:6.3} we use the description of the $\GL(2)$-system in terms of $\Q$-lattices to give a geometric interpretation of a generic element of $\caq$ in terms of commensurability classes of {\em parabolic} $\Q$-lattices. More precisely, we show in Theorem \ref{comparecomm} that the space of parabolic $\Q$-lattices, up to commensurability, is canonically isomorphic to the quotient 
$
\caqo:=P(\Q)\backslash (\overline{P(\A_{\Q,f})}\times P(\R))
$.\newline
 In Sect. \ref{subsec:6.4} we interpret these results in terms of  elliptic curves $E$ endowed with a  {\em triangular} structure, \ie a pair of elements of the Tate module $T(E)$ fulfilling an orthogonality relation (Definition \ref{triangdefn}). In Theorem \ref{ellipticcurve} we prove that the triangular condition characterizes the range of the map $\theta$.
  The equivalence relation of commensurability of $\Q$-lattices is then interpreted in terms of isogenies of triangular elliptic curves in Sect. \ref{subsec:6.5}. \newline
    In Sect. \ref{subsec:6.6} we show that the complex structure on $\caq$ inherited from the right action of $P^+(\R)$ coincides with the natural complex structure as a moduli space of elliptic curves. In Sect. \ref{subsec:6.7} we briefly describe the right action of $P(\hat \Z)$, while the boundary cases are described in Sect. \ref{subsec:6.8}.

 \subsection{Notations}\label{subsec:6.1}

In this part we fix the notations for the Shimura scheme $Sh(\GL_2,\H^\pm)$.
The group $\GL_2^+ (\R)$ acts on the complex upper-half
plane $\H$ by fractional linear transformations 
\begin{equation}\label{actalpha1}
\alpha(z) = \frac{az+b}{cz+d} \,,\quad \forall \alpha=\begin{pmatrix} a &b \\
c &d \end{pmatrix} \in \GL_2^+ (\R).
\end{equation}
We identify the multiplicative group $\C^\times$ as the subgroup  ${\rm SO}_2(\R)\times \R_+^*\subset\GL_2^+ (\R)$ by the map 
\begin{equation}\label{inclGL2R}
 a+i b \in \C^\times \mapsto \begin{pmatrix} a &b \\ -b &a
\end{pmatrix} \in \GL_2^+ (\R).
\end{equation}
The quotient $\GL_2^+ (\R)/\C^\times$ gets thus identified with the upper-half
plane $\H$ by the map
\begin{equation}\label{upper}
\alpha\in \GL_2^+ (\R)\mapsto z=\alpha(i)\in \H.
\end{equation}
In fact, the same map identifies the quotient $\GL_2(\R)/\C^\times$ with the disjoint union $\H^\pm$ of the upper and lower half
planes. 
By definition $Sh(\GL_2,\H^\pm)$ is  the
quotient
\begin{equation}\label{ShGL2Hpm}
Sh(\GL_2,\H^\pm):=\GL_2(\Q)\backslash \GL_2(\A_{\Q})/\C^\times=\GL_2(\Q)\backslash (\GL_2(\A_{\Q,f})\times
\H^\pm),
\end{equation}
where the left action of $\GL_2(\Q)$ in $(\GL_2(\A_{\Q,f})\times
\H^\pm)$ is via the diagonal embedding in the product $\GL_2(\A_{\Q,f})\times \GL_2 (\R)$. $Sh(\GL_2,\H^\pm)$ is a scheme over $\C$ (see \cite{milne}, Remark 2.10) which is the inverse limit of the Shimura varieties obtained as quotients by compact open subgroups  $K\subset \GL_2(\A_{\Q,f})$.
The space $Sh(\GL_2,\H^\pm)$ has infinitely many connected components. They are the fibers of the map
\begin{equation}\label{pi0Sh}
\det \times {\rm sign}: Sh(\GL_2,\H^{\pm}) \to Sh(\GL_1,\{\pm 1\}),
\end{equation}
where the determinant $\det: \GL_2(\A_{\Q,f})\to \GL_1(\A_{\Q,f})$ gives a map to  the group of finite ideles. Passing to the quotient  gives a map to the idele class group modulo its archimedean component, \ie here the group $\hatz$. The fiber of the map  \eqref{pi0Sh} over the point $(1,1)\in Sh(\GL_1,\{\pm 1\})$ is the connected quotient 
\begin{equation}\label{ShGL2Hpm0}
Sh_0(\GL_2,\H^\pm):=\SL_2(\Q)\backslash (\SL_2(\A_{\Q,f})\times
\H).
\end{equation}
By strong approximation (see \opcit Theorem 1.12) $\SL_2(\Q)$ is dense in $\SL_2(\A_{\Q,f})$,  thus one derives 
\begin{equation}\label{ShGL2Hpm0bis}
Sh_0(\GL_2,\H^\pm)=\SL_2(\Z)\backslash (\SL_2(\hat\Z)\times
\H).
\end{equation}
Using the identification
$
\SL_2(\hat\Z) =
\varprojlim_N \SL_2(\Z/N\Z),
$
the above quotient is associated to the {\em modular tower}, that is the tower of modular curves. More 
precisely, for $N \in \N$, let $Y(N)=\Gamma(N)\backslash \H$ be the
 modular curve of level $N$, where $\Gamma(N)$
is the principal congruence subgroup of $\Gamma=\SL_2(\Z)$. One has
\begin{equation}\label{ShGL2Hpm0ter}
Sh_0(\GL_2,\H^\pm)=\varprojlim_N \Gamma(N)\backslash 
\H=\varprojlim_N Y(N).
\end{equation}

\subsection{The relation with the $\GL(2)$-system}\label{subsec:6.2}

In the following part we explain the relation between the arithmetic construction of  $\caq=P(\Q)\backslash
\overline{P(\A_\Q)}$ and the $\GL(2)$-system.    The noncommutative  space underlying the $\GL(2)$-system contains the
quotient
\begin{equation}\label{ShGL2HpmNC}
Sh^{\rm nc}(\GL_2,\H^\pm):= \GL_2(\Q)\backslash
(M_2(\A_{\Q,f})\times \H^\pm),
\end{equation}
 and enlarges it by taking cusps into account. It is defined as the double quotient
\begin{equation}\label{ncShGL2comp}
\overline{Sh^{\rm nc}(\GL_2,\H^\pm)}:= \GL_2(\Q)\backslash
M_2(\A_\Q)^\bullet /\C^\times,
\end{equation}
where one sets
$$ M_2(\A_\Q)^\bullet:= M_2(\A_{\Q,f})\times (M_2(\R)\smallsetminus
\{0 \}). $$

\no Next lemma defines a canonical map $\caq \stackrel{\theta}{\to}\overline{Sh^{\rm nc}(\GL_2,\H^\pm)}$

\begin{lemma}\label{fundinclusion} \
\begin{enumerate}
\item[(i)] The inclusion $\overline{P(\R)}\subset (M_2(\R)\smallsetminus
\{0 \})$ induces a bijection of $\overline{P(\R)}$ with the complement in $(M_2(\R)\smallsetminus
\{0 \})/\C^\times$ of the point $\infty$ given by the class of matrices with vanishing second line.
\item[(ii)] The inclusion $\overline{P(\A_\Q)}\subset M_2(\A_\Q)^\bullet$ induces a morphism of noncommutative spaces
\begin{equation}\label{P2gl20}
	\caq=P(\Q)\backslash
\overline{P(\A_\Q)}\stackrel{\theta}{\longrightarrow} \GL_2(\Q)\backslash
M_2(\A_\Q)^\bullet /\C^\times	=\overline{Sh^{\rm nc}(\GL_2,\H^\pm)}.
\end{equation}
\end{enumerate}
\end{lemma}
\begin{proof} {\it (i)} Note that for real matrices the following implication holds
$$
\left(
\begin{array}{cc}
 a & b \\
 0 & 1 \\
\end{array}
\right)=\left(
\begin{array}{cc}
 a' & b' \\
 0 & 1 \\
\end{array}
\right)\left(
\begin{array}{cc}
 x & y \\
 -y & x \\
\end{array}
\right)\Longrightarrow \, y=0\ \& \ x=1 \Longrightarrow a=a'\ \&\  b=b'.
$$
Thus the induced map $\overline{P(\R)}\to(M_2(\R)\smallsetminus
\{0 \})/\C^\times$ is injective.\newline
Also, and again for real matrices one has, provided $c$ or $d$ is non-zero
\begin{equation}\label{PCdec}
\left(
\begin{array}{cc}
 a & b \\
 c & d \\
\end{array}
\right)=\left(
\begin{array}{cc}
 \frac{a d-b c}{c^2+d^2} & \frac{a c+b d}{c^2+d^2} \\
 0 & 1 \\
\end{array}
\right).\left(
\begin{array}{cc}
 d & -c \\
 c & d \\
\end{array}
\right).
\end{equation}
Thus all matrices in $M_2(\R)$ whose second line is non-zero 
belong to $\overline{P(\R)}/\C^\times$.
When both $c$ and $d$ are zero, one derives 
$$
\left(
\begin{array}{cc}
 a & b \\
 0 & 0 \\
\end{array}
\right)=\left(
\begin{array}{cc}
 1 & 0 \\
 0 & 0 \\
\end{array}
\right).\left(
\begin{array}{cc}
 a & b \\
 -b & a \\
\end{array}
\right).
$$
This means that when $c=d=0$  the right action of $\C^\times$ determines  a single orbit $\{\infty\}$ provided one stays away from the matrix $0$. Thus one obtains a canonical bijection
$$
(M_2(\R)\smallsetminus
\{0 \})/\C^\times=\overline{P(\R)}\cup \{\infty\}.
$$
{\it (ii)} By construction one has the inclusion 
$
\overline{P(\A_\Q)}\subset M_2(\A_\Q)^\bullet
$
and moreover the groups involved on both sides of the double quotient $\overline{Sh^{\rm nc}(\GL_2,\H^\pm)}$ are larger than those involved on the left hand side, thus one gets the required map $\theta$.\qed\end{proof} 

\no The proof of Lemma \ref{fundinclusion} shows that one has the identification
\begin{equation}\label{P2gl2}
M_2(\A_\Q)^\bullet/\C^\times= M_2(\A_{\Q,f})\times (\overline{P(\R)}\cup \{\infty\})
\end{equation}
By construction, one has the factorization
$$
\overline{P(\A_\Q)}=\overline{P(\A_{\Q,f})}\times \overline{P(\R)}
$$
Thus the map $\theta$ as in \eqref{P2gl20}, when considered at the archimedean place, only misses the point at infinity of $\overline{P(\R)}\cup \{\infty\}$.

\subsection{Commensurability classes of parabolic $\Q$-lattices}\label{subsec:6.3}

In this section we give  a geometric interpretation of the subspace
$$
\caqo:=P(\Q)\backslash (\overline{P(\A_{\Q,f})}\times P(\R))\subset P(\Q)\backslash
\overline{P(\A_\Q)}=:\caq.
$$
To this end, we introduce the notion of {\em parabolic} $\Q$-lattice in Definition \ref{defnpara}. Then, by implementing the commensurability equivalence relation we provide, in Proposition \ref{comparecomm}, the geometric description of $\caqo$
in terms of commensurability classes of  parabolic $\Q$-lattices.  The condition $\det(g_\infty)\neq 0$ which defines the subspace $\caqo\subset \caq$ is invariant under the left action of $P(\Q)$ and defines a dense open set in the naive quotient topology.  One obtains the canonical identification 
\begin{equation}\label{adelquot2dQlattpar}
\caqo=P^+(\Q)\backslash (\overline{P(\A_{\Q,f})}\times P^+(\R)).
\end{equation}
 We recall (see \cite{cmbook}, III Definition 3.17) that a two dimensional $\Q$-lattice is  a pair $(\Lambda,\phi)$ where $\Lambda\subset \C$ is a lattice and $\phi:\Q^2/\Z^2 \to \Q\Lambda / \Lambda$ is an arbitrary morphism of abelian groups. The morphism $\phi$ encodes the non-archimedean components of the lattice. The action of $\C^\times$ by scaling on $\Q$-lattices is given by
\begin{equation}\label{scaleQlat2}
\lambda(\Lambda,\phi)=(\lambda\Lambda,\lambda\phi) \,,\quad \forall
\lambda \in \C^\times.
\end{equation}
 The set of $2$-dimensional $\Q$-lattices is (see \cite{cmbook}, III Proposition 3.37) the quotient space
\begin{equation}\label{2dQlat}
\Gamma\backslash (M_2(\hat\Z)\times \GL_2^+(\R)),
\end{equation}
where $\Gamma=\SL_2(\Z)$. The set of 2-dimensional $\Q$-lattices up
to scaling is therefore identified with 
\begin{equation}\label{2dQlatScale}
\Gamma\backslash (M_2(\hat\Z)\times \GL_2^+(\R))/\C^\times=\Gamma\backslash (M_2(\hat\Z)\times \H).
\end{equation}
In this part, we provide the details of this identification. We use, as in \opcit  the basis $\{e_1=1,e_2=-i\}$ of $\C$ as a
2-dimensional $\R$-vector space to let $\GL_2 (\R)$ act on $\C$ as
$\R$-linear transformations. More precisely,
\begin{equation}\label{actalpha}
\alpha(x e_1 + y e_2) = (a x+ b y) e_1 + (cx+ dy) e_2, \qquad \alpha=\begin{pmatrix} a &b \\
c &d \end{pmatrix} \in \GL_2 (\R).
\end{equation}
Every 2-dimensional $\Q$-lattice $(\Lambda,\phi)$ can then be
described by the data 
\begin{equation}\label{data2Qlat}
(\Lambda,\phi)=(\alpha^{-1}\Lambda_0, \alpha^{-1}\rho),  \qquad \Lambda_0:=\Z e_1 + \Z e_2= \Z + i \Z
\end{equation}
for some $\alpha\in\GL_2^+(\R)$ and some $\rho\in M_2(\hat\Z)$ unique up to the left diagonal action of $\Gamma=\SL_2(\Z)$. Let us explain the notation $\alpha^{-1}\rho$ used in \eqref{data2Qlat}. The action of $\Z$ by multiplication on the abelian group $\Q/\Z$ extends to an isomorphism of rings $\hat\Z=\Hom(\Q/\Z,\Q/\Z)$ and  $M_2(\hat\Z)=\Hom(\Q^2/\Z^2,\Q^2/\Z^2)$. This gives meaning to the notation $a x\in \Q/\Z$ for $a\in \hat\Z$ and $x\in \Q/\Z$. 
We associate to $\rho\in
M_2(\hat\Z)$ the map
\begin{equation}\label{rhoLambda0}
\rho :\Q^2/\Z^2 \to \Q\Lambda_0 / \Lambda_0, \qquad  \rho(u)=
\rho_1(u)e_1 +\rho_2(u)e_2,
\end{equation}
where $\rho_1(u)=a x+b y$ and $\rho_2(u)=c x+d y$ for $u=(x,y)\in \Q^2/\Z^2$.
The action of $\rho$ is similar to the action of $\alpha$ as in \eqref{actalpha}
\begin{equation}\label{actrho}
\rho((x,y)) = (a x+ b y) e_1 + (cx+ dy) e_2\qqq (x,y)\in (\Q/\Z)^2,\quad \rho=\begin{pmatrix} a &b \\
c &d \end{pmatrix} \in M_2 (\hat\Z).
\end{equation}
To understand the extra structure on $\Q$-lattices which reduces the group $\GL(2)$ down to the parabolic subgroup $P$, we first consider the archimedean component.
The natural characterization of the subgroup $P^+(\R)\subset\ \GL_2^+(\R)$ is that its elements $g$ fulfill $\tau\circ g=\tau$, where $\tau$ is the projection on the imaginary axis
$$
\tau:x e_1 + y e_2\mapsto y e_2, \qquad \tau=\begin{pmatrix} 0 &0 \\
0 &1 \end{pmatrix}.
$$ 
For $z=x+iy\in \C$, we let  $\Im(z):=y$ denote the imaginary part of $z$, thus with our choice of basis one has $\Im(xe_1+ye_2)=-y$:  we shall keep track of this minus sign here below. This projection defines (Lemma \ref{extrastructure} $(ii)$) a	character 
of the elliptic curve $E=\C/\Lambda$ where the lattice $\Lambda$ is of the form 
\begin{equation}\label{Lambda0}
\Lambda=\alpha^{-1}\Lambda_0, \qquad \Lambda_0:=\Z e_1 + \Z e_2= \Z + i \Z.
\end{equation}
We define the orthogonal of a lattice $\Lambda$ by the formula
$$
\Lambda^\perp=\{z\in\C\mid <z,z'>\in \Z,\qquad \forall z'\in \Lambda.
$$
Here we use the standard non-degenerate pairing defining the duality, given by 
 $$
 <z,z'>:=\Re(z\bar z')=xx'+yy',\qquad \forall z=x+iy, \quad z'=x'+iy'.
 $$

\begin{lemma}\label{extrastructure} Let $\Lambda=\alpha^{-1}\Lambda_0$ be a $\Q$-lattice, with $\alpha\in P(\R)$. Then
 \begin{enumerate}
\item[(i)] $\Im(\Lambda)=\Z$. 
\item[(ii)] The linear map $\Im$ induces a group homomorphism $\Im:E\to \R/\Z$ from the elliptic curve 
$E=\C/\Lambda$ to the abelian group $U(1):=\R/\Z$, \ie a character  of the abelian group $E$.
\item[(iii)] The orthogonal lattice $\Lambda^\perp$ contains the vector $e_2$.
\end{enumerate}
\end{lemma}
\begin{proof} {\it (i)} One has the implications
$$
\alpha=\begin{pmatrix} a &b \\
0 &1 \end{pmatrix}\Rightarrow \alpha^t=\begin{pmatrix} a &0 \\
b &1 \end{pmatrix}\Rightarrow \alpha^t e_2=e_2,
$$
 thus $\alpha^t$, the transpose of the matrix $\alpha$, fulfills $\alpha^t e_2=e_2$ and
$$
\Im(\Lambda)=<\Lambda,e_2>=<\alpha^{-1}\Lambda_0,\alpha^t e_2>=<\Lambda_0,e_2>=\Z.
$$
{\it (ii)} For $\xi\in E=\C/\Lambda$ the value $\Im \xi$ is meaningful modulo $\Im(\Lambda)=\Z$, thus the group homomorphism $\Im:E\to \R/\Z$ is well defined. 
\newline
{\it (iii)} For $\Lambda$  as in \eqref{Lambda0},   $\Lambda^\perp=\alpha^{t}\Lambda_0$. This follows from $\Lambda_0=\Lambda_0^\perp$ and 
$$
<\alpha^{-1}\xi,\eta>=<\xi,(\alpha^{-1})^{t}\eta>, \ (\alpha^{-1})^{t}\eta \in \Lambda_0 \iff 
\eta \in \alpha^{t}\Lambda_0.
$$
Then the orthogonal lattice always contains the vector $e_2$, since one has $\alpha^t e_2=e_2$.\qed
\end{proof}

 Next, we restrict the homomorphisms $\phi$ for $\Q$-lattices $(\Lambda,\phi)$ in the same way as we restricted the lattices  in Lemma \ref{extrastructure}.   From  \eqref{actrho} and the definition of $\overline{P(R)}$  in \eqref{algspaceP}, one has
  \begin{equation}\label{actrhoeq}
\rho \in \overline{P(\hat\Z)} \iff \rho_2(u)=y \qqq u=(x,y)\in \Q^2/\Z^2.
\end{equation}
To write this condition in terms of  $\phi:\Q^2/\Z^2 \to \Q\Lambda / \Lambda$, with $\phi=\alpha^{-1}\rho$ and for $\Lambda=\alpha^{-1}\Lambda_0$,  $\alpha\in P(\R)$, we use the character $\chi=-\Im: E\to \R/\Z$ (sending torsion points to torsion points). One has $\chi \circ \alpha^{-1}=\chi$, since $\alpha^{-1}\in P(\R)$ and 
$$
\chi\circ \phi:\Q^2/\Z^2\to \Q/\Z,   \qquad \chi\circ \phi=\chi \circ \alpha^{-1}\circ \rho=\chi \circ \rho=\rho_2.
$$
One thus obtains
\begin{equation}\label{chiphi}
\chi\circ \phi=\rho_2.
\end{equation}

\begin{lemma}\label{extrastructure1} Let $(\Lambda,\phi)$ be a two dimensional $\Q$-lattice described by data $(\Lambda,\phi)=(\alpha^{-1}\Lambda_0, \alpha^{-1}\rho)$, 
for some $\alpha\in\GL_2^+(\R)$ and some $\rho\in M_2(\hat\Z)$. Then, 
\begin{equation}\label{charofP}
(\rho,\alpha)\in \Gamma\backslash\left(
\overline{P(\hat\Z)}\times P^+(\R)\right)\iff\Im(\Lambda)=\Z \ \&  \ \chi\circ \phi(u)=y \qqq u=(x,y)\in \Q^2/\Z^2
\end{equation}
where $\chi:\C/\Lambda\to \R/\Z$ is given by $\chi=-\Im$ and  $\Gamma= \SL_2(\Z)$ acts diagonally.
\end{lemma}
\begin{proof} By Lemma \ref{extrastructure}, one has $\alpha \in P(\R)\Rightarrow \Im(\Lambda)=\Z $. Moreover, it follows from the above discussion that $\rho \in \overline{P(\hat\Z)}$ is equivalent to $\rho_2(u)=y$, thus one gets $\chi\circ \phi(u)=y$ for any $ u=(x,y)\in \Q^2/\Z^2$. This shows the implication $\Rightarrow$ in \eqref{charofP}, since the $\Q$-lattice  $(\Lambda,\phi)$ associated to $(\rho,\alpha)$ only depends upon the orbit of this pair under the left diagonal action of $\Gamma= \SL_2(\Z)$. The character $\chi:\C/\Lambda\to \R/\Z$ is induced by $-\Im:\C\to \R$ and only depends upon the lattice $\Lambda$ so that the conditions on the right hand side of \eqref{charofP} only depend on the $\Q$-lattice  $(\Lambda,\phi)$. \newline
Conversely, let us now assume that these conditions hold. The condition $\Im(\Lambda)=\Z$ means that $e_2\in \Lambda^\perp$ is not divisible, \ie $e_2$ does not belong to any multiple of $\Lambda^\perp$. It follows (using Bezout's theorem) that there exists $\xi\in \Lambda^\perp$ such that $\Z \xi+ \Z e_2=\Lambda^\perp$. Define $\beta\in\GL_2^+(\R)$ as $\beta(e_1)=\pm\xi$ and $\beta(e_2)=e_2$. One has $\Lambda^\perp=\beta(\Lambda_0)$ by construction and thus one derives
$$
\Lambda=(\Lambda^\perp)^\perp=(\beta(\Lambda_0))^\perp
\quad =(\beta^t)^{-1}\Lambda_0.
$$
For $\beta=\begin{pmatrix} a &b \\
c &d \end{pmatrix} \in \GL_2^+(\R)$, one derives from \eqref{actalpha},  $\beta(e_2)=b e_1+ d e_2$ and thus $\beta(e_2)=e_2$ is equivalent to $b=0$ and $d=1$. In turns these conditions mean that $\alpha'=\beta^t\in P^+(\R)$. In this way we have proven that, using the condition $\Im(\Lambda)=\Z$, we can find $\alpha'\in P^+(\R)$ such that $\Lambda=\alpha'^{-1}\Lambda_0$. The equality $\alpha'^{-1}\Lambda_0=\alpha^{-1}\Lambda_0$ shows that  $\gamma=\alpha'\alpha^{-1}\in \Gamma=\SL_2(\Z)$ and $\gamma\alpha\in P^+(\R)$. Thus by replacing   $(\rho,\alpha)$ with $(\gamma\rho,\gamma\alpha)$ we can assume that $(\Lambda,\phi)=(\alpha^{-1}\Lambda_0, \alpha^{-1}\rho)$, with $\alpha\in P^+(\R)$. The second hypothesis $\chi\circ \phi(u)=y \qqq u=(x,y)\in \Q^2/\Z^2$ implies,  using \eqref{chiphi}, that $\rho_2(u)=y \qqq u=(x,y)\in \Q^2/\Z^2$ and thus, by \eqref{actrhoeq}, that $\rho \in \overline{P(\hat\Z)}$.\qed\end{proof}

\begin{figure}[t]
\begin{center}
\includegraphics[scale=0.65,bb=150 10 138 200]{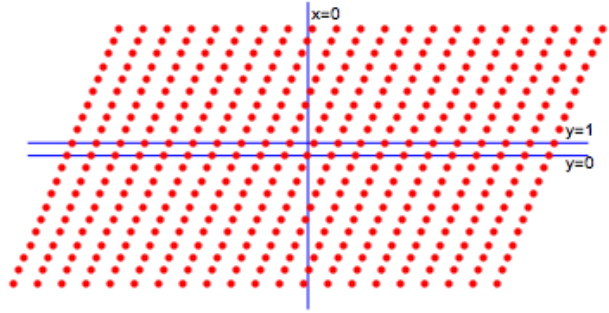}
\caption{A parabolic lattice \label{lattice1} }
\end{center}
\end{figure}

\begin{definition}\label{defnpara} A {\em parabolic} 2-dimensional $\Q$-lattice is a $\Q$-lattice of the form $(\Lambda,\phi)=(\alpha^{-1}\Lambda_0, \alpha^{-1}\rho)$, where $\rho\in \overline{P(\hat \Z)}$ and $\alpha\in P^+(\R)$. \newline	
We say that a parabolic $\Q$-lattice $(\Lambda,\phi)$ is degenerate when $\rho_1=0$.
\end{definition}
Notice that for a parabolic 2-dimensional $\Q$-lattice,  the pair $(\rho, \alpha)$ such that $(\Lambda,\phi)=(\alpha^{-1}\Lambda_0, \alpha^{-1}\rho)$ is unique up to the diagonal action of $\Gamma\cap P^+(\R)=P^+(\Z)\simeq \Z$. Thus the space of parabolic 2-dimensional $\Q$-lattices is described by the quotient
\begin{equation}\label{parab}
\Pi:=P^+(\Z)\backslash (\overline{P(\hat \Z)}\times P^+(\R)).
\end{equation}

\begin{rem}\label{arithprogremark} \
\begin{enumerate}
\item When the parabolic 2-dimensional $\Q$-lattice $(\Lambda,\phi)$ is degenerate, there exists a unique $\alpha\in P^+(\R)$ such that $(\Lambda,\phi)=(\alpha^{-1}\Lambda_0, \alpha^{-1}p)$,
where $p=\begin{pmatrix} 0 &0 \\
0 &1 \end{pmatrix}\in \overline{P(\hat \Z)}$. 
\item Let $\Lambda\subset \C$ be a $\Q$-lattice such that $\Im(\Lambda)=\Z$, then $\Lambda$ is characterized by the following arithmetic progression in $\R$ with associated lattice $L=\Lambda\cap \R$
 \begin{equation}\label{apoflambda}
 	A={\rm prog}(\Lambda):=\{u\in \R\mid i+u\in \Lambda\}.
 \end{equation}
 Let $(\Lambda,\phi)=(\alpha^{-1}\Lambda_0, \alpha^{-1}\rho)$ be a parabolic $\Q$-lattice, with $\alpha=\begin{pmatrix} y & x \\ 0 &1 \end{pmatrix}\in P^+(\R)$. 
Then one has	${\rm prog}(\Lambda)=y^{-1}(\Z+x)$, with 
 $L=y^{-1}\Z$. The pair $(L,\xi)$, with 
 $\xi:\Q/\Z\to \R/L$,  $\xi(u):=\phi(u,0)$  determines a one dimensional $\Q$-lattice $(L,\xi)$.
 \end{enumerate}
\end{rem}

The next step we undertake is to describe the meaning of commensurability for parabolic $\Q$-lattices. We recall from \cite{cmbook} the following (see Definition 3.17)

\begin{definition}\label{defncommens}
	Two $\Q$-lattices are said to be commensurable ( 
$ (\Lambda_1, \phi_1) \sim (\Lambda_2, \phi_2)$)
when
\begin{equation}\label{commensrel}
\Q\Lambda_1=\Q\Lambda_2 \ \ \ \text{ and } \ \  \phi_1 = \phi_2 \ \mod. \
\Lambda_1 + \Lambda_2 .
\end{equation}
\end{definition}
Commensurability is an equivalence relation (\!\cite{cmbook} Lemma 3.18). By applying Proposition 3.39 of \opcit, the space of commensurability classes of 2-dimensional $\Q$-lattices
up to scaling is given by the quotient space
\begin{equation}\label{adelquot2dQlatt}
\GL_2^+(\Q)\backslash (M_2(\A_{\Q,f})\times \H)
\end{equation}
here $\GL_2^+(\Q)$ acts diagonally by $(\rho,z)\mapsto
(g\rho,g(z))$. We then continue by providing the description of the orbits for this action. Let us first consider the orbit of $\Lambda_0:=\Z e_1 + \Z e_2= \Z + i \Z$.

\begin{lemma}\label{lambdazero} Let $g\in \GL_2(\Q)$ and assume that $\Lambda=g\Lambda_0$ fulfills $\Im \Lambda=\Z$. Then there exist $h\in P(\Q)$ and $k\in \Gamma= \SL_2(\Z)$ such that $g=hk$.	
\end{lemma}
\begin{proof} For $g=\begin{pmatrix} a &b \\
c &d \end{pmatrix}\in \GL_2(\Q)$, one then obtains 
$$
g(x e_1 + y e_2) = (a x+ b y) e_1 + (cx+ dy) e_2, \ \ \Im(g(x e_1 + y e_2))=-(cx+ dy).
$$ 
Thus $\Im \Lambda=\Z$ means  that $c,d\in  \Z$ and they are relatively prime. Let then $u,v\in \Z$ be such that $cu+dv=1$. Set $w=\begin{pmatrix} d &u \\
-c &v \end{pmatrix}$, then one has $w\in \Gamma$ and
$$
gw=\left(
\begin{array}{cc}
 a d-b c & a u+b v \\
 0 & c u+d v \\
\end{array}
\right)=\left(
\begin{array}{cc}
 a d-b c & a u+b v \\
 0 & 1 \\
\end{array}
\right)\in P(\Q).
$$
Thus by taking $h=gw$ and $k=w^{-1}$ one obtains the required factorization. \qed\end{proof}

We recall   (\!\cite{cmbook}, Sect. III.5) that the equivalence relation of commensurability on the space of 
2-dimensional $\Q$-lattices, is induced by the partially defined action of
$\GL_2^+(\Q)$.  Indeed, for $g\in \GL_2^+(\Q)$
and $(\Lambda,\phi)=(\alpha^{-1}\Lambda_0,\alpha^{-1}\rho)$ such
that $g\rho\in M_2(\hat\Z)$, the $\Q$-lattice $(\alpha^{-1} g^{-1}
\Lambda_0, \alpha^{-1}\rho)$ is commensurable to $(\Lambda,\phi)$.
Moreover all $\Q$-lattices commensurable to a given
$(\Lambda,\phi)$ are of this form. Here we used, as done above, the description of 2-dimensional $\Q$-lattices as 
$$
\Gamma\backslash (M_2(\hat\Z)\times \GL_2^+(\R)),\ \ (\rho,\alpha)\mapsto (\Lambda,\phi)=(\alpha^{-1}\Lambda_0, \alpha^{-1}\rho).
$$
We can now state the following key result on commensurability

\begin{theorem}\label{comparecomm}\
\begin{enumerate}
\item[(i)] Two parabolic two dimensional $\Q$-lattices $(\Lambda_j,\phi_j)=(\alpha_j^{-1}\Lambda_0, \alpha_j^{-1}\rho_j)$, $j=1,2$, with $\rho_j\in \overline{P(\hat \Z)}$ and $\alpha_j\in P^+(\R)$ are commensurable (as $\Q$-lattices) if and only if there exists $g\in P^+(\Q)$ such that $\rho_2=g\rho_1$ and $\alpha_2=g\alpha_1$.	
\item[(ii)] The space of parabolic $\Q$-lattices up to commensurability is canonically isomorphic to the quotient $\caqo$ as in  \eqref{adelquot2dQlattpar}.
\end{enumerate}
\end{theorem}
\begin{proof} {\it (i)} Since $ P^+(\Q)\subset \GL_2^+(\Q)$, the existence of $g\in P^+(\Q)$ with $\rho_2=g\rho_1$ and $\alpha_2=g\alpha_1$ implies the commensurability. Conversely, if  two parabolic $\Q$-lattices are commensurable, there exists $g\in \GL_2^+(\Q)$ such that $\rho_2=g\rho_1$ and $\alpha_2=g\alpha_1$.	 Because $P^+(\R)$ is a group and $\alpha_j\in P^+(\R)$, one gets   
$g\in P^+(\R)$ and thus $g\in \GL_2^+(\Q)\cap P^+(\R)=P^+(\Q)$. \newline
{\it (ii)} Let $h=\begin{pmatrix} a &b \\
0 &1 \end{pmatrix}\in \overline{P(\A_{\Q,f})}$. We  show that there exists $g\in P^+(\Q)$ such that $gh\in \overline{P(\hat \Z)}$. Let $\alpha\in \Q_+^\times$ such that $\alpha a\in \hat \Z$.  Since $\hat \Z$ is open and $\Q$ is dense in $\A_{\Q,f}$ with $\hat \Z+\Q=\A_{\Q,f}$, there exists $\beta\in \Q$ such that $\alpha b+\beta\in \hat \Z$. One then derives 
$$
\begin{pmatrix} \alpha &\beta \\
0 &1 \end{pmatrix}\begin{pmatrix} a &b \\
0 &1 \end{pmatrix}= \begin{pmatrix} \alpha a &\alpha b +\beta\\
0 &1 \end{pmatrix}\in \overline{P(\hat \Z)}.
$$
It follows that all the left $P^+(\Q)$ orbits of elements of $\overline{P(\A_{\Q,f})}\times P^+(\R)$ intersect the open subset $\overline{P(\hat \Z)}\times P^+(\R)$ whose elements yield parabolic $\Q$-lattices. Thus by applying $(i)$ one obtains the required isomorphism.\qed
\end{proof}

\begin{rem}\label{bcremark} From the point of view of noncommutative geometry, the quotient space 
derived by applying the commensurability relation on the space \eqref{parab} of parabolic $\Q$-lattices is best described by considering the crossed product, in the sense of \cite{Neshveyev},  by the Hecke algebra of double classes of the subgroup $P^+(\Z)\subset P^+(\Q)$. We find quite remarkable (and encouraging) that this Hecke algebra is precisely the one on which the BC-system is based. 	
\end{rem}

\subsection{$\caq$ and $\gaq$ as moduli spaces of elliptic curves}\label{subsec:6.4}

In this section we use the description of $Sh^{\rm nc}(\GL_2,\H^\pm)$ as a moduli space of elliptic curves to obtain a similar interpretation  for the spaces $\caq=P(\Q)\backslash
\overline{P(\A_\Q)}$ and  $\gaq=P(\Q)\backslash
\overline{P(\A_\Q)}/\hatz$. We first formulate Lemma \ref{extrastructure1} in terms of the global Tate module\footnote{The global Tate module $TE$ is best described at the conceptual level as the etale fundamental group $\pie(E,0)$, where $E$ is viewed as a curve over $\C$. Given $\rho\in \Hom (\Q/\Z, E_{\rm tor })$ the corresponding element
of $\pie(E,0)$ is given by the $(\rho(\frac 1 n))_{n\in \N}$.} of the elliptic curve $E=\C/\Lambda$.   For the theory developed in this paper, we think of the global Tate module as the abelian group
\begin{equation}\label{torEhatZbis}
 TE= \Hom (\Q/\Z,
E_{\rm tor })\,.
\end{equation}
We denote by $E_{\rm tor }=\Q\Lambda/\Lambda$ 
the torsion subgroup
 of the elliptic curve $E$. In this way, the morphism 
$\phi$ in the definition of a two-dimensional $\Q$-lattice is now seen as the map $\phi:\Q^2/\Z^2 \to E_{\rm tor }$. By applying the covariant functor $T:=\Hom
(\Q/\Z, -)$, we rewrite  $\phi$ as a $\hat\Z$-linear map
\begin{equation}\label{phiH1}
T(\phi):\hat\Z\oplus\hat\Z \to TE,
\end{equation}
that is given by  a pair of elements $(\xi,\eta)$ of $TE$. The character $\chi:\C/\Lambda\to U(1)$ induces a homomorphism of torsion subgroups $\chi:E_{\rm tor }\to \Q/\Z$, and by applying $\Hom
(\Q/\Z, -)$, we obtain a morphism 
$$
T(\chi): TE\to \hat\Z.
$$
The following result characterizes the parabolic $\Q$-lattices among two-dimensional $\Q$-lattices.

\begin{proposition}\label{extrastructure2} Let $(\Lambda,\phi)$ be a two dimensional $\Q$-lattice, 
$E=\C/\Lambda$ the associated  elliptic curve, and
 $(\xi,\eta)$ the related pair of points in the total Tate module $TE$. Then $(\Lambda,\phi)$ is a parabolic $\Q$-lattice  if and only if 
\begin{equation}\label{charofP1}
\Im(\Lambda)=\Z \ \ \& \ T(\chi)(\xi)=0, \ \ T(\chi)(\eta)=\id, \ \text{for} \ \chi=-\Im.
\end{equation} 
\end{proposition}
\begin{proof} The result follows from Lemma \ref{extrastructure1}, using the faithfulness of the functor $T$ in the form $$
h\in \Hom
(\Q/\Z, \Q/\Z)\ \& \ T(h)=0 \Rightarrow h=0.
$$
This is clear, by using $T(h)(\id)=h$. \qed
\end{proof} 

Given two elliptic curves $E=\C/\Lambda$ and $E'=\C/\Lambda'$, an isomorphism $j:E\to E'$ is given by the multiplication map by  $\lambda\in \C^\times$, such that $\Lambda'=\lambda \Lambda$. It follows that the  elliptic curve $E$ endowed with  the pair $(\xi,\eta)\in T(E)$ associated to a two dimensional $\Q$-lattice $(\Lambda,\phi)$ determines the latter  up to scale. In particular, passing from a parabolic $\Q$-lattice to the associated triple  $(E;\xi,\eta)$   is equivalent to assign the map  $\theta$ from parabolic $\Q$-lattices to $\Q$-lattices up to scale
 	\begin{equation}\label{P2gl2scal}
P^+(\Z)\backslash (\overline{P(\hat \Z)}\times P^+(\R))\stackrel{\theta}{\longrightarrow} \Gamma\backslash
(M_2(\hat\Z)\times \GL_2^+(\R)) /\C^\times.
\end{equation}

\begin{rem}
	If one ignores the non-archimedean components,  the map $\theta$ restricts to the map  $\theta_\infty:P^+(\Z)\backslash P^+(\R)\to \Gamma\backslash\GL_2^+(\R) /\C^\times$ induced by the inclusion $P^+(\R)\subset \GL_2^+(\R) /\C^\times$. Notice that this restriction is far from being injective. Indeed, let $\alpha\in P^+(\R)$ and $\gamma\in \Gamma$. Let  $\gamma\alpha=\alpha'\lambda$ be the $P\C^\times$ decomposition  of $\gamma\alpha$ as in \eqref{PCdec}. Then $\alpha'\in P^+(\R)$ and  $\theta_\infty(\alpha')=\theta_\infty(\alpha)$, while  $\alpha'\notin P^+(\Z)\alpha$, unless $\gamma\in P^+(\Z)$.
\end{rem}

Next proposition shows that the implementation of the non-archimedean components makes $\theta$ injective except in a well understood case, corresponding to the vanishing of the non-archimedean components

\begin{proposition}\label{comparescal}
 	The natural map $\theta$ as in \eqref{P2gl2scal} 
 	is injective except when the parabolic $\Q$-lattices are degenerate (Definition \ref{defnpara}). Furthermore, the following formula defines a free action of $\Z$ on the degenerate parabolic $\Q$-lattices, whose orbits are the fibers of the map $\theta$
 	\begin{equation}\label{P2gl2scal1}
\tau(c)(p,\alpha)=(p,t_c(\alpha)), \ \ 	t_c(\alpha)=\begin{pmatrix} 1 &0 \\
c &1 \end{pmatrix}\alpha\ (1+c\bar z)^{-1}\qqq c\in \Z.\end{equation}
Here, $p=\begin{pmatrix} 0 &0 \\
0 &1 \end{pmatrix}\in \overline{P(\hat \Z)}$, $\alpha=\begin{pmatrix} y &x \\
0 &1 \end{pmatrix}\in P^+(\R)$, and $z=x+iy\in \C$.
 \end{proposition}
 \begin{proof}  We first test the injectivity of $\theta$. Let $(\rho,\alpha)$ and $(\rho',\alpha')$ be elements of $\overline{P(\hat \Z)}\times P^+(\R)$ and assume that one has an equality of the form 
 $$
 (\rho',\alpha')=\gamma (\rho,\alpha \lambda), \ \gamma \in \Gamma, \ \lambda \in \C^\times.
 $$
 Let $\gamma=\begin{pmatrix} a &b \\
c &d \end{pmatrix}$ and $\rho=\begin{pmatrix} u &v \\
0 &1 \end{pmatrix}$. One has $\gamma \rho=\rho'\in \overline{P(\hat \Z)}$ and thus $cu=0$. This implies that either $c=0$ or $u=0$, since $c\in \Z$ and $u\in \hat \Z$.\newline
 Assume first that $c=0$. Then $\gamma \rho=\rho'$ implies $d=1$ and it follows that $a=1$ so that $\gamma\in P^+(\Z)$. Then the equality $\gamma\alpha \lambda=\alpha'$ implies  $\lambda\in P(\R)$ and since $\lambda \in \C^\times$, one concludes that $\lambda=1$ and $\gamma\alpha =\alpha'$. This shows that $(\rho,\alpha)$ and $(\rho',\alpha')$ are equal in $P^+(\Z)\backslash (\overline{P(\hat \Z)}\times P^+(\R))$. \newline
 Assume now that $c\neq 0$. Then one has $u=0$. 
Moreover since $\gamma \rho=\rho'\in \overline{P(\hat \Z)}$ one gets $cv+d=1$. This gives $v=(1-d)/c$ and since $v\in \hat\Z$ and $c,d\in \Z$ one gets that $v\in \hat\Z\cap \Q=\Z$. Then, by replacing $(\rho,\alpha)$ with $\delta(\rho,\alpha)$, where $\delta=\begin{pmatrix} 1 &-v \\
0 &1 \end{pmatrix}$ one obtains the equality, in $P^+(\Z)\backslash (\overline{P(\hat \Z)}\times P^+(\R))$ of $(\rho,\alpha)$, with an element of the form $(p,\alpha'')$. It remains to see when two such elements are equal in $\Gamma\backslash
(M_2(\hat\Z)\times \GL_2^+(\R)) /\C^\times$. Thus we now assume that $\rho=\rho'=p$. The equality 
$\gamma \rho=\rho'$ now means that $b=0$ and $d=1$. But since $\gamma\in \Gamma$ one gets also that $a=1$ and thus $\gamma= \begin{pmatrix} 1 &0 \\
c &1 \end{pmatrix}$. Now by \eqref{PCdec} there exists uniquely $\alpha''\in P(\R)$ and $\lambda'\in \C^\times$ such that $\gamma\alpha=\alpha''\lambda'$ and the equality $ (\rho',\alpha')=\gamma (\rho,\alpha \lambda)$ shows that $\alpha'=\alpha''$ and $\lambda'=\lambda^{-1}$. One has 
 $$
 \gamma\alpha=\begin{pmatrix} 1 &0 \\
c &1 \end{pmatrix}\alpha = \begin{pmatrix} 1 &0 \\
c &1 \end{pmatrix}\begin{pmatrix} y &x \\
0 &1 \end{pmatrix}=\begin{pmatrix} y &x \\
yc & xc+1 \end{pmatrix}=\alpha'\begin{pmatrix} xc+1 & -yc \\
yc & xc+1 \end{pmatrix}=\alpha' (1+c \bar z).
 $$
Thus $\alpha'=t_c(\alpha)$. The first line of $t_c(\alpha)$ is 
$$
\left(\frac{y}{y^2 c^2+(x c+1)^2},\frac{(x^2 +y^2) c+x}{y^2 c^2+(x c+1)^2}\right)
$$
and one has $t_{c'+c}(\alpha)=t_{c'}(t_c(\alpha))$ for all $c,c'\in \Z$ since
$$
\begin{pmatrix} 1 &0 \\
c'+c &1 \end{pmatrix}\alpha= \begin{pmatrix} 1 &0 \\
c' &1 \end{pmatrix}\gamma\alpha=\begin{pmatrix} 1 &0 \\
c' &1 \end{pmatrix}\alpha'(1+c \bar z)\in t_{c'}(t_c(\alpha))\C^\times.
$$
  By construction, the map $\theta$ is invariant under left multiplication by $\Gamma$ and right multiplication by $\C^\times$, thus one has, with $\gamma=\begin{pmatrix} 1 &0 \\
c &1 \end{pmatrix}\in \Gamma$, and $z=x+iy$ 
$$
\theta((p,\alpha))=
\theta((\gamma p,\gamma\alpha))=\theta((p,\alpha' (1+c\bar z)))=
\theta((p,\alpha'))=\theta(\tau(c)(p,\alpha))
$$
Finally, we claim that the pairs $(p,t_c(\alpha))$ for $c\in \Z$ are all distinct elements of $P^+(\Z)\backslash (\overline{P(\hat \Z)}\times P^+(\R))$. Indeed, if $(p,t_c(\alpha))=u(p,t_{c'}(\alpha))$ for some $u\in P^+(\Z)$ the equality $up=p$ implies $u=1$, thus it is enough to show that the $t_c(\alpha)$ are all distinct.  But since $y\neq 0$, the equality $t_c(\alpha)=t_{c'}(\alpha)$ implies in particular $(x^2 +y^2) c=(x^2 +y^2) c'$ and hence $c=c'$.\qed   \end{proof}

 Next, we  associate a character $\chi\in \Hom(E,\R/\Z)$, unique up to sign, to certain elements of $T(E)$.
 
\begin{lemma}\label{uniquetate}
Let $\xi$ be an element of the total Tate module $TE$ of the elliptic curve $E=\C/\Lambda$. Let 
$$
\xi^\perp:= \{\chi\in \Hom(E,\R/\Z)\mid T(\chi)(\xi)=0\}\subset \Hom(E,\R/\Z)\simeq \Lambda^\perp. 
$$
Then if $\xi \neq 0$ and $\xi^\perp\neq \{0\}$ one has $\xi^\perp= \Z\,\alpha$, for a primitive character $\alpha$ unique up to sign.
\end{lemma}
\begin{proof} By fixing  a basis of $\Hom(E,\R/\Z)\simeq \Lambda^\perp$ we may identify 
$$
TE=\hat \Z \times \hat \Z, \ \ \xi=(u,v), \ \ \xi^\perp=\{(n,m)\in \Z^2\mid nu+mv=0\}.
$$
Since $\xi \neq 0$, let $\ell$ be a prime such that $(u_\ell,v_\ell)\neq (0,0)$. Then since $\xi^\perp\neq \{0\}$ there exists relatively prime integers $(n,m)\neq (0,0)$ such that $n u_\ell+m v_\ell=0$ and any solution of $n' u_\ell+m' v_\ell=0$ is a multiple of $(n,m)$. If $nu+mv=0$ in $\hat \Z$ one has $\xi^\perp=\{k(n,m)\mid k\in \Z\}$. Otherwise, there exists $\ell'$ such that $n u_{\ell'}+m v_{\ell'}\neq 0$. Then $k(n u_{\ell'}+m v_{\ell'})\neq 0$ for any $k\neq 0$ and this contradicts $\xi^\perp\neq \{0\}$. This shows that $\xi^\perp= \Z\,\alpha$, where $\alpha=(n,m)$. Uniqueness up to sign is clear.\qed\end{proof}
 
Next result gives an intrinsic description of the range of the map $\theta$ in terms of the geometric interpretation of $\Q$-lattices up to scale  given in terms of an elliptic curve $E$, and a
pair of points $\xi=(\xi_1,\xi_2)$ in the total Tate module $TE$ (see also  Proposition 3.38 of \cite{cmbook}).

\begin{theorem}\label{ellipticcurve} Let $E$ be an elliptic curve together with a
pair of elements $(\xi,\eta)$ of the total Tate module $TE$. Assume $\xi \neq 0$. Then the corresponding $\Q$-lattice belongs to the range of the map $\theta$ if and only if one has $<\xi^\perp,\eta>=\Z$, where $<\xi^\perp,\eta>:=\{T(\chi)(\eta)\mid \chi \in \xi^\perp\}\subset \hat \Z$.	
\end{theorem}
\begin{proof} Assume first that the datum $(E;\xi,\eta)$ arises from a $\Q$-lattice  $(\Lambda,\phi)=(\alpha^{-1}\Lambda_0, \alpha^{-1}\rho)$, where $\rho\in \overline{P(\hat \Z)}$ and $\alpha\in P^+(\R)$. 	
The action of $\rho$ is given in \eqref{actrho}.
 This determines 
$$
\xi \in \Hom(\Q/\Z,E_{\rm tor}), \ \xi(x)=\alpha^{-1}(ax e_1+cx e_2)\in \Q\Lambda/\Lambda=E_{\rm tor}
$$
and similarly 
$$
\eta \in \Hom(\Q/\Z,E_{\rm tor}), \ \eta(y)=\alpha^{-1}(by e_1+dy e_2)\in \Q\Lambda/\Lambda=E_{\rm tor}.
$$
Since $\rho\in \overline{P(\hat \Z)}$, one has $c=0$ and thus $T(\chi)(\xi)=0$, where the character $\chi$ of $E$ is associated to the element $\alpha^t(e_2)=e_2$ of the dual lattice $\Lambda^\perp$. Since $\chi$ is primitive and $\xi\neq 0$ one has $\xi^\perp=\Z\chi$. Since $d=1$ one gets that $\chi\circ \eta(y)=y$ for all $y\in \Q/\Z$ and thus one obtains $<\xi^\perp,\eta>=\Z$ as required. 
Conversely, let us assume that $(E;\xi,\eta)$ fulfill $\xi \neq 0$ and $<\xi^\perp,\eta>=\Z$. By Lemma \ref{uniquetate} one has $\xi^\perp= \Z\,\chi$, for a primitive character $\chi$ unique up to sign. Moreover since $<\xi^\perp,\eta>=\Z$ one can choose the sign in such a way that 
$T(\chi)(\eta)=1$. Consider the pair $(E,\chi)$ of the elliptic curve $E$ and the primitive character $\chi$. Let $E=\C/\Lambda$, then $\chi\in \Lambda^\perp$, and using the scaling action of $\C^\times$, we can assume that $\chi=e_2$. Since $\chi$ is primitive one has  $\Im(\Lambda)=\Z$ and the 
 linear map $-\Im$ induces the group morphism $\chi:\C/\Lambda\to \R/\Z$.  We also have $T(\chi)(\xi)=0$ and $T(\chi)(\eta)=1$ so that Proposition \ref{extrastructure2} applies  showing that the $\Q$-lattice $(\Lambda,\phi)$, with $\phi=(\xi,\eta)$ is parabolic. Thus $(E;\xi,\eta)$ is in the range of the map $\theta$.\qed \end{proof} 
 
 \begin{definition}\label{triangdefn} A triangular structure on an elliptic curve $E$ is a pair $(\xi,\eta)$ of elements of the Tate module $T(E)$, such that $\xi \neq 0$ and $<\xi^\perp,\eta>=\Z$. 	
 \end{definition}
  In the following, we shall abbreviate ``elliptic curve with triangular structure" by ``triangular elliptic curve". 
  
  \no By Proposition 3.38 of \cite{cmbook} a triangular elliptic curve corresponds to a $\Q$-lattice $(\Lambda,\phi)$ unique up to scale, and by Proposition \ref{comparescal} this datum  corresponds  to a unique parabolic $\Q$-lattice which we call the associated $\Q$-lattice.

\subsection{Commensurability and isogenies}\label{subsec:6.5}

We  recall that an isogeny from an abelian variety $A$ to another $B$ is a surjective morphism with finite kernel. 
In this section we describe how a triangular structure behaves under isogenies.  At the geometric level, the commensurability relation is obtained from the following notion of isogeny between triangular elliptic curves

\begin{definition}\label{isogendefn} An {\em isogeny} $f:(E,\xi,\eta)\to (E',\xi',\eta')$ of triangular elliptic curves is an isogeny $f: E\to E'$ such that $T(f)(\xi)=\xi'$ and  $T(f)(\eta)=\eta'$.	
 \end{definition}
 For ordinary isogenies, one can use the dual isogeny to show that  the existence of an isogeny $E\to E'$ is a symmetric relation. This result uses the fact that multiplication by a positive integer $n$ is an isogeny. In our set-up the multiplication by $n$ gives $\xi'=n\xi$ and $\eta'=n\eta$. This modification does not alter the orthogonal, \ie one has  $\xi'^\perp=\xi^\perp$. But one has $<\xi'^\perp,\eta'>=n\Z$, thus the triangular condition is not fulfilled unless $n=\pm 1$.\newline
  The following result determines the equivalence relation generated by isogenies
  
\begin{proposition}\label{isogenies1} Let $(E;\xi,\eta)$ and $(E';\xi',\eta')$ be two triangular elliptic curves and $(\Lambda,\phi)$ and $(\Lambda',\phi')$ the associated parabolic $\Q$-lattices.\begin{enumerate}
\item[(i)] Let $f:(E;\xi,\eta)\to (E';\xi',\eta')$ be an isogeny of triangular elliptic curves. Then $\Lambda\subset \Lambda'$, $f:\C/\Lambda\to \C/\Lambda'$ is the  map induced by the identity and $\phi'=f\circ \phi$.	
\item[(ii)] The parabolic  $\Q$-lattices $(\Lambda,\phi)$ and $(\Lambda',\phi')$ are commensurable if and only if there exist two isogenies $f:(E,\xi,\eta)\to (E'',\xi'',\eta'')$ and $f':(E',\xi',\eta')\to (E'',\xi'',\eta'')$ to the same triangular elliptic curve.
\end{enumerate}
\end{proposition}
\begin{proof} {\it (i)} By definition, an isogeny $f:E\to E'$ is a holomorphic group morphism $f:\C/\Lambda\to \C/\Lambda'$, $f(z)=\lambda z$, $\forall z\in \C$, where the complex number $\lambda$ is such that $\lambda\Lambda\subset\Lambda'$. The characters $\chi$ and $\chi'$ uniquely determined by the triangular structure are given in both cases by minus the imaginary part, and one has $\chi'\circ f=\chi$. This shows that, modulo $\Z$, one has $\Im(\lambda z)=\Im(z)$ for all $z\in \C$, \ie $\Im((\lambda-1)\C)\subset \Z$. Thus $\lambda =1$, $\Lambda\subset \Lambda'$, and  $f$ is the  map induced by the identity.\newline
{\it (ii)} By applying Definition \ref{defncommens}, 
when the two $\Q$-lattices $\Lambda_1=\Lambda$, $\Lambda_2=\Lambda'$ are parabolic, one derives 
$$
\Im(\Lambda_j)=\Z \ \&  \ \chi\circ \phi_j(u)=y \qqq u=(x,y)\in \Q^2/\Z^2
$$
for the character $\chi=-\Im$. Then  the lattice $\Lambda''=\Lambda_1 + \Lambda_2$ fulfills $\Im(\Lambda'')=\Z$ and  the quotient maps $f_j:\C/\Lambda_j\to \C/\Lambda''$ fulfill $\chi\circ f_j=\chi$, since for $z\in \C$ one has $\Im(z+ \Lambda_j)=\Im(z)+\Z=\Im(z+ \Lambda'')$. It follows that the two equal maps 
$\phi:=f_j\circ \phi_j$ fulfill the condition 
$$
\chi\circ \phi(u)=y \qqq u=(x,y)\in \Q^2/\Z^2,
$$
so that the pair $(\Lambda'',\phi)$ is a parabolic $\Q$-lattice. Thus the inclusions $\Lambda_j\subset \Lambda''$ induce isogenies to the same triangular elliptic curve as required.  To prove the converse it is enough, using the transitivity of the commensurability relation for $\Q$-lattices, to show that if $f:(E,\xi,\eta)\to (E',\xi',\eta')$ is an isogeny of triangular elliptic curves the associated $\Q$-lattices are commensurable. This fact follows from $(i)$.\qed
\end{proof}

\subsection{The complex structure}\label{subsec:6.6}

Proposition \ref{comparescal} states that
 	the natural map $\theta$ as in \eqref{P2gl2scal} from parabolic $\Q$-lattices to $\Q$-lattices up to scale
 	is injective except in the degenerate case. Thus  $\theta$  provides, by pull back, a large class of functions, by implementing the arithmetic subalgebra of the $\GL_2$-system (\!\cite{cmbook} Chapter 3, \S 7). The functions in this  algebra are holomorphic for  the natural complex structure on the moduli space of elliptic curves and in this section we compare this complex structure 
with the one on the space $\Pi=P^+(\Z)\backslash (\overline{P(\hat \Z)}\times P^+(\R))$  defined using the right action of $P^+(\R)$ (Proposition \ref{ax+bact}).\newline
 We recall that the complex structure on the moduli space of elliptic curves is obtained by comparing two descriptions of the quotient space $\GL_2^+ (\R)/\C^\times$. The first one identifies $\GL_2^+ (\R)/\C^\times$ with the complex upper half plane $\H$ via the map
\begin{equation}\label{upper}
C:\alpha\in \GL_2^+ (\R)\mapsto z=\alpha(i)=\frac{ai+b}{ci+d}\in \H.
\end{equation}
The second description derives  from the space $\cB/\C^\times$ of pairs $(\xi_1,\xi_2)$ of $\R$-independent elements of $\C$ up to scale. The maps 
\begin{equation}\label{upperbis}
r:\cB/\C^\times\to\H^\pm =\H\cup -\H, \ \ r(\xi_1,\xi_2)= -\xi_2/\xi_1\in \C\setminus \R=\H^\pm
\end{equation}
and
\begin{equation}\label{upper1}
B: \GL_2^+ (\R)\to \cB/\C^\times, \ \ B(\alpha)= (\alpha^{-1}e_1,\alpha^{-1}e_2), 
\end{equation}
fulfill $r\circ B=C$. Indeed, both maps only depend on the right coset in $\GL_2^+ (\R)/\C^\times$. The right $\C^\times$-coset associated to
$z=x+i y\in \H$ contains, in view of \eqref{upper},  the matrix $$\alpha= \begin{pmatrix} y &x \\
0 &1
\end{pmatrix} \in \GL_2^+ (\R).$$ One can replace 
$\alpha^{-1}= y^{-1}\,\begin{pmatrix} 1 &-x \\
0 &y
\end{pmatrix}$, up to scale, by $g=\begin{pmatrix} 1 &-x \\
0 &y
\end{pmatrix}$ and obtain  
$$g(e_1)=\begin{pmatrix} 1 &-x \\
0 &y
\end{pmatrix}\begin{pmatrix} 1  \\
0 
\end{pmatrix}=e_1=1, \ g(e_2)=\begin{pmatrix} 1 &-x \\
0 &y
\end{pmatrix}\begin{pmatrix} 0  \\
1 
\end{pmatrix}=-x e_1+ y e_2=-z.$$
This shows that $r\circ B=C$. Notice that it is meaningless to use the action of $\GL_2^+ (\R)$ on the elements $(\xi_1,\xi_2)$ of a basis because this action does \emph{not} commute with scaling. 

Next, we consider the complex  structure on $\Pi=P^+(\Z)\backslash (\overline{P(\hat \Z)}\times P^+(\R))$ as defined in Lemma \ref{adelicomp1}, namely by means of  the ``dbar" operator $\partial_x+i\partial_y$ and in terms of the map $\iota$ of \eqref{algspaceiota}, \ie  of the matrix $\begin{pmatrix} y &x \\
0 &1
\end{pmatrix} \in P^+ (\R)$. The above calculation then shows that this complex structure is identical to the canonical complex structure on the moduli space of elliptic curves. We now verify that this complex structure  can also be described, as in Proposition \ref{ax+bact}, using the right action of $P^+(\R)$ on $\Pi$.\newline
 The ``dbar" operator for the latter structure is defined by $X+iY$, where the two vector fields $X,Y$ on $\Pi$ are defined as $X=y\partial_x$ and $Y=y\partial_y$, corresponding to the Lie algebra elements of the one parameter subgroups $u(\epsilon)$ for $X$  
$$
\epsilon \mapsto u(\epsilon):=\begin{pmatrix} 0 &\epsilon \\
0 &1
\end{pmatrix}, \ \  \  \begin{pmatrix} y &x \\
0 &1
\end{pmatrix} u(\epsilon)= \begin{pmatrix} y &x+ y \epsilon \\
0 &1
\end{pmatrix} $$
and $v(\epsilon)$ for $Y$
$$\epsilon \mapsto v(\epsilon):=\begin{pmatrix} e^\epsilon &0 \\
0 &1
\end{pmatrix},\ \  \  \begin{pmatrix} y &x \\
0 &1
\end{pmatrix} v(\epsilon)= \begin{pmatrix} y e^\epsilon & x \\
0 &1
\end{pmatrix}.
$$
The comparison of the complex structures is summarized in the following statement

 \begin{proposition}\label{comparescalcs}
 	The natural map from parabolic $\Q$-lattices to $\Q$-lattices up to scale
 	\begin{equation}\label{P2gl2scalbis}
\theta:\Pi\to \Gamma\backslash
(M_2(\hat\Z)\times \GL_2^+(\R)) /\C^\times	
\end{equation}
is holomorphic for the canonical complex structure on the moduli space of elliptic curves and the complex structure on $\Pi$ associated to the right action of $P^+(\R)$.
\end{proposition}
\begin{proof} It suffices  to check that $(X+iY)(f)=0$, where the function $f$ is the pullback by $\theta$ of the local parameter $z\in \H=\GL_2^+(\R)) /\C^\times$. This fact follows from the direct computation
$$
f(\begin{pmatrix} y &x \\
0 &1
\end{pmatrix})=x+iy, \ \ (X+iY)(f)=y\partial_x (x+iy)+iy \partial_y (x+iy)=0.
$$ 
\qed\end{proof}

 \subsection{The right action of $P(\hat \Z)$}\label{subsec:6.7}

In order to pass from $\caq=P(\Q)\backslash
\overline{P(\A_\Q)}$ to  $\gaq=P(\Q)\backslash
\overline{P(\A_\Q)}/\hatz$ one needs to divide the component in the ad\`ele class space by the action of $\hatz$ given by multiplication. This action is induced by the right action of $\hatz$  on  $\Pi=P^+(\Z)\backslash (\overline{P(\hat \Z)}\times P^+(\R))$, and it is meaningful even  before passing to commensurability classes. In this section we provide its geometric meaning  in terms of parabolic $\Q$-lattices.   
It turns out that this is the special case (obtained by restricting to the subgroup $\hatz\subset P(\hat \Z)$ of diagonal matrices) of the right action of $P(\hat \Z)$ on $\Pi$, whose  geometric meaning is given in the following Proposition \ref{rightactpz}.

\begin{proposition}\label{rightactpz}
Let $(E;\xi,\eta)$ be a triangular elliptic curve and $(\Lambda,\phi)$ the associated $\Q$-lattice. Its image, under the right action of $w=\begin{pmatrix} u &v \\
0 &1 \end{pmatrix} \in  P(\hat \Z)$, is the triangular elliptic curve $(E;\xi',\eta')$ where $\xi'=\xi\circ u$ and $\eta'=\eta+\xi\circ v$.	
\end{proposition}
\begin{proof}  Note that the condition $<\xi^\perp,\eta>=\Z$ of Theorem \ref{ellipticcurve} still holds for the transformed pair $(E;\xi',\eta')$, 
since $(\xi\circ u)^\perp=\xi^\perp$ and $<\xi^\perp,\xi\circ v>=0$. 
For $\rho=\begin{pmatrix} a &b \\
0 &1 \end{pmatrix} \in  \overline{P(\hat \Z)}$, one has  
$$
\xi(x)=\alpha^{-1}(ax e_1)\in \Q\Lambda/\Lambda, \ \ 
\eta(y)=\alpha^{-1}(by e_1+y e_2)\in \Q\Lambda/\Lambda
$$
and for $w=\begin{pmatrix} u &v \\
0 &1 \end{pmatrix} \in  P(\hat \Z)$, one obtains
$$
\begin{pmatrix} a &b \\
0 &1 \end{pmatrix} \begin{pmatrix} u &v \\
0 &1 \end{pmatrix}=\begin{pmatrix} au &av +b\\
0 &1 \end{pmatrix}=
\begin{pmatrix} a' &b' \\
0 &1 \end{pmatrix}.
$$
 Thus the right action of $w$ determines the new pair $(\xi',\eta')$ 
 $$
\xi'(x)=\alpha^{-1}(aux e_1)\in \Q\Lambda/\Lambda, \ \ 
\eta'(y)=\alpha^{-1}(by e_1+y e_2)+\alpha^{-1}(avx e_1)\in \Q\Lambda/\Lambda,
$$
and one concludes $\xi'=\xi\circ u$ and $\eta'=\eta+\xi\circ v$.\qed	
\end{proof}

\subsection{Boundary cases}\label{subsec:6.8}

Theorem \ref{ellipticcurve}  and Proposition \ref{comparescal}, show that triangular elliptic curves  are classified by the subspace
$$
\Pi':=P^+(\Z)\backslash\{(\rho,\alpha)\in \overline{P(\hat \Z)}\times P^+(\R)\mid \rho =\begin{pmatrix} u &v \\ 0 &1 \end{pmatrix}, \ u\neq 0\}\subset \Pi.
$$  
The condition $u\neq 0$ in this definition is meaningful  in the quotient since the left action of $P^+(\Z)$ leaves $u$ unaltered. Assume now that $u=0$ and $v\notin \Z$. Then, with the notations of Proposition \ref{comparescal}, $\rho\notin P^+(\Z)p$. This result thus shows that the corresponding triple $(E;\xi,\eta)$ still characterizes the element of $\Pi$. One has $\xi=0$ since $\rho(x,0)=u xe_1=0$, moreover one also gets  $\rho(0,y)=v ye_1+y e_2$. Let $\chi$ be a character of $E$, then $\chi=\alpha^t(ne_1+me_2)$,  with $n,m\in \Z$. Thus $T(\chi)(\eta)=nv+m$. Since $v\notin \Z$, while $v\in \hat \Z$, one derives  $v\notin \Q$. Thus $\chi=\alpha^t(e_2)$  is the only character which takes the value $1$ on $\eta$, \ie 
\begin{equation}\label{condcond}
\exists ! \chi\in \xi^\perp ~{\rm such~ that} ~T(\chi)(\eta)=1. 
\end{equation} 
Note that if  \eqref{condcond} holds, then, when $\xi\neq 0$,  Lemma \ref{uniquetate} shows  that there exists a primitive character $\chi_0$ of $E$ with $\xi^\perp=\Z\chi_0$. Since $\chi\in \xi^\perp$, one thus gets $\chi=n \chi_0$ for some $n\in \Z$. Thus  $T(\chi)(\eta)=n T(\chi_0)(\eta)$ and $n T(\chi_0)(\eta)=1$. But $T(\chi_0)(\eta)\in \hat \Z$ and one then gets $n=\pm 1$. This shows that one can refine the definition of a triangular structure using \eqref{condcond} in place of the condition
$$
\xi \neq 0\ \ \& \ <\xi^\perp,\eta>=\Z.
$$
  Thus, the case $u=0$ (\ie $\xi=0$) and $v\notin \Z$ is covered by the following
  
\begin{definition}\label{degeneratetr} A degenerate triangular structure on an elliptic curve $E$ is a pair $(\chi,\eta)$ of a character 	$\chi:E\to \R/\Z$ and an element  $\eta \in T(E)$ with  $T(\chi)(\eta)=1$.
\end{definition}
This notion also covers  the case  $u=v=0$, \ie of degenerate parabolic $\Q$-lattices.  Indeed, in this case Proposition \ref{comparescal} shows  that one needs to choose the character $\chi\in \xi^\perp$ so that $<\chi,\eta>=1$.
More precisely, let $\tilde \theta$ be the  map which associates to a degenerate parabolic $\Q$-lattice $(\Lambda,\phi)$  the degenerate triangular structure on $E=\C/\Lambda$ given by the pair $(\chi,\eta)$, where $\chi=-\Im$ and $\eta(y)=\phi((0,y))$ for all $y\in \Q/\Z$. Then we have the following

\begin{proposition}\label{degeneratetr1}  Two degenerate parabolic $\Q$-lattices are the same if and only if the  degenerate triangular elliptic curves, associated via the map $\tilde \theta$,  are isomorphic. 	
\end{proposition}
\begin{proof} Let $(\Lambda,\phi)$ and $(\Lambda',\phi')$ be degenerate parabolic $\Q$-lattices, and $E,(\chi,\eta)$, $E',(\chi',\eta')$ their images under $\tilde \theta$. By the  degeneracy hypothesis there exists uniquely $\alpha,\alpha'\in P^+(\R)$ such that 
$$
\Lambda=\alpha^{-1}\Lambda_0,\ \phi(x,y)=\alpha^{-1}(ye_2) \qqq x,y\in \Q/\Z
$$
and similarly for $(\Lambda',\phi')$. An isomorphism $j:E\to E'$ is implemented by the multiplication by a complex number $\lambda$ such that $\lambda\Lambda=\Lambda'$. 
  If $j$ preserves the degenerate triangular structure, one has $\chi'\circ j=\chi$, \ie $\Im (\lambda z)=\Im(z)$ modulo $\Z$ for all $z\in E=\C/\Lambda$ and thus $\lambda =1$. Hence one derives $\Lambda=\Lambda'$ and $j$ is the identity. Let us show that $\phi'=\phi$. One has $\phi(x,y)=\eta(y)$, $\forall x,y\in \Q/\Z$, and since by hypothesis $j\circ \eta=\eta'$ one concludes that $\phi'=\phi$.\qed \end{proof}

Next we consider the degeneracy  occurring when in the matrix $\alpha=\begin{pmatrix} a &b \\ 0 &1 \end{pmatrix}\in P^+(\R)$, $a$ tends to $0$. We follow the lattice $\Lambda=\alpha^{-1}\Lambda_0$ up to scale. One has 
$$
a\alpha^{-1}=a\left(
\begin{array}{cc}
 \frac{1}{a} & -\frac{b}{a} \\
 0 & 1 \\
\end{array}
\right)= \left(
\begin{array}{cc}
 1 & -b \\
 0 & a \\
\end{array}
\right), \ \ a\alpha^{-1}\begin{pmatrix} x \\ y \end{pmatrix}=
\left(
\begin{array}{cc}
 1 & -b \\
 0 & a \\
\end{array}
\right)\begin{pmatrix} x \\ y \end{pmatrix}=\begin{pmatrix} x-by \\ ay \end{pmatrix}.
$$
Thus when $a\to 0$ the lattice $\Lambda=\alpha^{-1}\Lambda_0$ up to scale converges pointwise (\ie for each fixed pair $(x,y)$) towards the subgroup of $\R\subset \C$ given by 
\begin{equation}\label{lambdab}
 \Lambda(b):=\Z+b\Z\subset \R\subset \C.
\end{equation}
The subgroup $\Lambda(b)$ only depends upon $b\in \R/\Z$ and the quotient $\R/\Lambda(b)$ corresponds to the noncommutative torus $\T^2_b$. In fact, the composition with $\rho=\begin{pmatrix} u &v \\ 0 &1 \end{pmatrix} \in  \overline{P(\hat \Z)}$ gives the following $\Q$-pseudolattice in the sense of Definition 3.106 of \cite{cmbook}
$$
\phi:\Q^2/\Z^2\to \Q \Lambda(b)/\Lambda(b), \ \ \phi((x,y))=u x+v y -b y.
$$
When $b\notin \Q$, the subspace $\Q \Lambda(b)\subset \R$ is two dimensional over $\Q$ and one defines a character on the $\Q$-rational points by setting $\chi(x-by):=y\in \Q/\Z$, for $x-by\in \Q \Lambda(b)/\Lambda(b)$. Using this character one gets $\chi\circ \phi((x,y))=y$ and this condition characterizes the relevant pseudolattices. 

\section{Lift of the Frobenius correspondences}\label{sect:7}

In this paper we have constructed the simplest complex lift of the Scaling Site, using an almost periodic compactification of the added imaginary direction. We have illustrated the role of the tropicalization map and found a surprising relation between the obtained complex lift and the $\GL_2$ system of \cite{cmbook}.

In order to complete the Riemann-Roch strategy in this lifted framework one meets a fundamental difficulty tied up to the loss of the one parameter group of automorphisms of $\rma$ in moving from characteristic $1$ to characteristic zero. The difficulty arises in the  construction of the proper lift of the correspondences  $\Psi_\lambda$ which are canonical in characteristic $1$. The natural candidates in characteristic zero come from the right action of $P_+(\R)$. This choice is justified using the tropicalization map which is given on all terms by the determinant (on $2$ by $2$ matrices of ad\`eles) and makes the following diagram commutative 
\begin{equation}\label{tropdet}
\xymatrix@C=20pt@R=35pt{
  P(\Q)\ar[d]^{\det}\ar[rr]^{} && \overline{P(\A_\Q)}\ar[d]_{\det} && P_+(\R)\ar[d]_{\det}\ar[ll]^{}
\\
\Q^\times \ \ar[rr]^{}&& \A_\Q &&  \R^*_+ \ar[ll]^{}}
\end{equation}
The problem arises because the right action of  $\R^*_+\subset P_+(\R)$ does not preserve the complex structure defined in 
 Proposition \ref{ax+bact}. More precisely this action preserves the foliation (since the leaves are precisely the orbits of the right action of $P_+(\R)$)  but it does not preserve  the complex structure of the leaves. Indeed, the right action of  $\R^*_+\subset P_+(\R)$ is of the form 
$$
\left(
\begin{array}{cc}
 y & x \\
 0 & 1 \\
\end{array}
\right)\left(
\begin{array}{cc}
 \lambda & 0 \\
 0 & 1 \\
\end{array}
\right)=
\left(
\begin{array}{cc}
  \lambda y & x\\
 0 & 1 \\
\end{array}
\right)
$$
and since it replaces $x+iy\in \H$ by $x+i \lambda y$  it does not respect the complex structure.

\subsection{Witt construction in characteristic $1$ and lift of the $\Psi_\lambda$}\label{subsect:7.1}
  In \cite{ccthickening} we already addressed the problem of the loss of the one parameter group of automorphisms of $\rma$ in moving from characteristic $1$ to characteristic zero.  In that paper we also developed the archimedean analogue of the basic steps of the construction of the rings of $p$-adic periods and we defined the universal thickening of the real numbers. In fact we showed that  when one applies the analogue of the Witt construction to the real tropical hyperfield  $\trop$,
 the universal $W$-model of $\trop$ exists and it coincides with the triple which was constructed in \cite{CCentropy,Cwitt}, by working with the tropical semi-field $\rmax$ of characteristic one and implementing  concrete formulas, involving entropy, which extend the Teichm\"uller formula for sums of Teichm\"uller lifts to the case of characteristic one. We use the notation $[x]=\tau(x)$ for the Teichm\"uller lift of $x\in \R_+^*$. These elements generate linearly the ring $W$, they fulfill the multiplication rule $[xy]=[x][y]$ and the automorphisms of $\rmax$ lift to automorphisms $\theta_\lambda$ of $W$ such that 
 $$
 \theta_\lambda([x])=[x^\lambda] \qqq x\in \rmax, \, \lambda \in \R_+^*.
 $$
 We shall use complex coefficients so that in first approximation $W$ is the complex group ring of the multiplicative group $\R_+^*$. We disregard here the nuances obtained from various completions of $W$ explored in \cite{ccthickening} and concentrate on the algebraic question of showing why the use of $W$ as coefficients resolves the problem of the lack of invariance of the complex structure under the right action of $\R_+^*$. For each $\lambda \in \R_+^*$ one has a ring homomorphism $\chi_\lambda:W\to \C$ which is $\C$-linear and such that $\chi_\lambda([x])=x^\lambda$ for any $x\in \R_+^*$. By construction one has $\chi_\lambda\circ \theta_\mu=\chi_{\lambda\mu}$ and the character $\chi_1$ plays a key role in  \cite{ccthickening} where it is denoted $\theta$. Now given a function with values in $W$ on  a leaf given by  an orbit of the right action of  $P_+(\R)$ we say that $f$ is holomorphic when 
 \begin{equation}\label{holom}
 \left(\lambda X+i Y\right) 	\chi_\lambda(f)=0 \qqq \lambda \in \R_+^*,
 \end{equation}
 where $X,Y$ correspond to the generators of the Lie algebra of  $P_+(\R)$ as in Proposition \ref{ax+bact}. 
We then restore the invariance under the right action of $\R_+^*\subset P_+(\R)$ by combining it with the automorphisms $ \theta_\lambda\in \Aut(W)$ as we did in the construction of the arithmetic Frobenius in \cite{CC4}. By construction the right action $R(\mu)$  of $\R_+^*\subset P_+(\R)$ extends to $W$ valued functions so that the following equation holds
\begin{equation}\label{holombis}
 	\chi_\lambda(R(\mu)(f))=\mu^Y \chi_\lambda(f) \qqq \lambda, \mu \in \R_+^*
 \end{equation}
\begin{proposition}\label{frobarith}  For $\mu \in \R_+^*$, the operation $\fr^a_\mu$,
\begin{equation}\label{holom1}
f\mapsto \fr^a_\mu(f), \ \fr^a_\mu(f):= \theta_\mu(R(\mu^{-1})( f))
 \end{equation}
	preserves the holomorphy condition \eqref{holom}.
\end{proposition}
\begin{proof} One has
$$
\chi_\lambda(\fr^a_\mu(f))=\chi_\lambda\circ\theta_\mu(R(\mu^{-1})( f))=
\chi_{\lambda \mu}(R(\mu^{-1})( f))=\mu^{-Y}\chi_{\lambda \mu}(f)
$$
using  \eqref{holombis} for the last equality.
By construction $Y$ commutes with $\mu^Y$, and since $[Y,X]=X$, one has $\mu^Y X \mu^{-Y}=\mu X$ so that
$$
\left(\lambda X+iY\right) \mu^{-Y}=\mu^{-Y}\left(\lambda \mu X+iY\right)
$$
and one gets
$$
\left(\lambda X+iY\right)\chi_\lambda(\fr^a_\mu(f))=
\left(\lambda X+iY\right) \mu^{-Y}\chi_{\lambda \mu}( f)=\mu^{-Y}\left(\lambda \mu X+iY\right)\chi_{\lambda \mu}( f)=0.
$$
Thus $\fr^a_\mu(f)$ fulfills \eqref{holom} if $f$ does.
\qed\end{proof}
To give a non-trivial example of a $W$-valued function which is holomorphic in the sense of \eqref{holom}, we take the classical orbit $\Gamma_{\Q,{\rm cl}}$ 
	which is the pro-\' etale cover $\tilde \D^*$. One can represent the elements of $\Gamma_{\Q,{\rm cl}}$ as pairs $(x,y)\in G\times \R_+^*$ and the following equality defines a function  
 \begin{equation}\label{holom2}
 q:\Gamma_{\Q,{\rm cl}}\to W, \ \  q(x+iy):=[e^{-2\pi y}]e^{2\pi i x}
  \end{equation}
  where,  by Lemma \ref{adelicomp} $(iii)$, $e^{2\pi i x}$ makes sense as a complex number for any $x\in G$. 
  \begin{proposition}\label{functionq} The function $q:\Gamma_{\Q,{\rm cl}}\to W$ of \eqref{holom2} is holomorphic in the sense of condition \eqref{holom}. Moreover it is invariant under the transformations $\fr^a_\mu$ for any $\mu \in \R_+^*$. The same holds for the rational powers $q^r$ of $q$, $\forall r\in \Q$.
\end{proposition}
\begin{proof} One has by construction
$$
 \chi_\lambda(q)(x+iy)=e^{-2\pi \lambda y}e^{2\pi i x}=e^{2\pi i (x+i\lambda y)}
 $$
 Moreover
 $$
 (\lambda y\partial_x +i y\partial_y)(x+i\lambda y)=0\ \Rightarrow \ \left(\lambda X+i Y\right) 	\chi_\lambda(q)=0
 $$
 which gives condition \eqref{holom}. Finally the equality
$$
 \theta_\mu(q(x+i \mu^{-1}y)=[(e^{-2\pi \frac{y}{\mu}})^\mu]e^{2\pi i x}
 =[e^{-2\pi y}]e^{2\pi i x}=q(x+iy)
 $$	
 shows that $q$ is invariant under the transformations $\fr^a_\mu$. The same proof applies to the rational powers of $q$ which make sense because of Lemma \ref{adelicomp} $(iii)$.
\qed\end{proof}
\begin{remark}\label{perfectoid} Proposition \ref{functionq} suggests that for the topos counterpart $\tilde \dst\rtimes \nt$ of the above adelic description of the complex lift  (see the left column of Figure \ref{globalpic}) the structure sheaf  involves the ring $W[q^r]$ generated by rational powers $q^r$ of $q$ over $W$. 	
\end{remark}

\subsection{Quantization}\label{subsect:7.2}
 It is still unclear in which precise sense the complex lift constructed in this paper can be used to ``quantize" the Scaling Site. One of the origins of the world of characteristic $1$ is the inverse process of quantization, it is called ``dequantization".  It was developed under the name of idempotent analysis by the school of Maslov, Kolokolstov and Litvinov \cite{Maslov,Lit}. One of their key discoveries is that the Legendre transform which plays a fundamental role in all of physics and in particular in 
 thermodynamics in the nineteenth century, is simply the Fourier transform  in the framework of  idempotent analysis. There is a whole  circle of ideas which compares tropicalization, dequantization and deformation of complex structures (see \eg \cite{Baier} and the references there) and these ideas should be carefully identified for the complex lift constructed in our paper. In particular the deformation of complex structures used in Sect. \ref{subsect:7.1} and the interpretation of its limit as a real polarization should be clarified.  In the process of quantizing a classical dynamical system the expected outcome is a self-adjoint operator in Hilbert space.
We expect here that the operator $X+iY$ will play a role as well as the ``transverse elliptic theory'' developed in \cite{cmindex}. Indeed, when viewing the ad\`ele class space $\Q^\times\backslash \A_\Q$  as a noncommutative space and the complex lift $\caq$  over it, the complex structure takes place in the transverse direction. In fact, as we have seen in Sect.  \ref{sec:5}, the space of points of $\caq$ fibers over $\Q^\times\backslash \A_\Q$ with fiber the almost periodic compactification $G$ of $\R$. The effect of the almost periodic  compactification is occurring purely in the transversal direction and it thus suggests that the $\bar \partial$ operator associated to the complex structure should be viewed as a $K$-homology class in the relative type II set-up. The results of \cite{CoMo1,CoMo2,CoMo3} on the transverse structure of modular Hecke algebras should then be brought into play.

\begin{acknowledgement}
The second named author would like to thank Alain Connes for introducing her to the noncommutative geometric vision on the Riemann Hypothesis and for sharing with her many mathematical ideas and insights.\newline
Caterina Consani is partially supported by the Simons Foundation collaboration Grant no. 353677. She would like to thank Coll\`ege de France for some financial support.

\end{acknowledgement}

\end{document}